\documentclass[11pt]{article}

\usepackage{lastpage}
\usepackage{amsmath}
\usepackage{amsfonts,bm}
\usepackage{color,tikz}
\usepackage{amssymb}
\usepackage{graphicx}
\usepackage{enumerate}
\usepackage{amsthm}
\usepackage[all]{xy}

\def\mod{{\rm mod\,}}

 \allowdisplaybreaks

\usepackage{hyperref} 

\newtheorem{theorem}{Theorem}[section]
\newtheorem{corollary}[theorem]{Corollary}
\newtheorem{definition}[theorem]{Definition}
\newtheorem{conjecture}[theorem]{Conjecture}

\newtheorem{lemma}[theorem]{Lemma}
\newtheorem{proposition}[theorem]{Proposition}
\newtheorem{example}[theorem]{Example}

\title{{\bf A survey on spectral conditions for some \\ 
extremal graph problems}\thanks{
This is the second revised version in which we added some new references. 
This paper was published on Advances in Mathematics (China), 51 (2) (2022) 193--258;
see the official website: \url{http://www.oaj.pku.edu.cn/sxjz/CN/1000-0917/home.shtml} or 
CNKI (China National Knowledge Internet) 
\url{https://navi.cnki.net/knavi/journals/SXJZ/detail}.  
This work was supported by  NSFC (Grant Nos. 
12071484, 11871479, 11931002),  
Hunan Provincial Natural Science Foundation (Grant Nos. 2020JJ4675, 2018JJ2479) 
and  Mathematics and Interdisciplinary Sciences Project of CSU.  
This article is a survey paper and it is very long with \pageref{LastPage} pages. 
In such survey, it is impossible to refer to all the nice papers. So if a paper is missing from this survey, which does not mean that it is not worth including it. Last but not least, if there are some elegant results and new progresses related to this topic but it missed in our survey, please let me know. Any suggestions and comments are welcome.  
E-mail addresses: \url{ytli0921@hnu.edu.cn} (Y\v{o}ngt\={a}o L\v{i}), 
\url{wjliu6210@126.com} (W\v{e}ij\`{u}n Li\'{u}), 
\url{fenglh@163.com} (L\`{i}hu\'{a} F\'{e}ng, corresponding author).} }

\author{Yongtao Li$^{\dag}$, Weijun Liu$^{\ddag}$, Lihua Feng$^{\ddag, *}$  \\  
{\small $^{\dag}$School of Mathematics, Hunan University} \\
{\small Changsha, Hunan, 410082, P.R. China}  \\ 
{\small $^{\ddag}$School of Mathematics and Statistics,
 Central South University} \\
{\small Changsha, Hunan, 410083, P.R. China}\\
}

\date{\today}

\begin{document}

\maketitle

\vspace{-0.5cm}

\begin{abstract} 
This survey is two-fold. We first report 
new progress on the spectral extremal  results on   
the  Tur\'{a}n type problems in graph theory. More precisely, 
we shall summarize the spectral Tur\'{a}n  function 
 in terms of the adjacency spectral radius 
and the signless Laplacian spectral radius for various graphs. 
For instance, 
the complete graphs, general graphs 
with chromatic number at least three, complete bipartite graphs, odd cycles, 
even cycles, color-critical graphs and intersecting triangles. 
 
 The second goal is 
 to conclude some recent results of  
 the spectral conditions on some graphical properties. 
 By a unified method, we present some sufficient conditions based on 
 the adjacency spectral radius and the signless Laplacian spectral radius 
 for a graph to be 
 Hamiltonian, $k$-Hamiltonian, $k$-edge-Hamiltonian, 
 traceable, $k$-path-coverable, 
 $k$-connected, $k$-edge-connected, 
 Hamilton-connected, perfect matching and 
   $\beta$-deficient.
 \end{abstract} 



{{\bf Chinese Library Classification: } O157.5} 

{{\bf Document code:} A}

\newpage 
\tableofcontents

\section*{Introduction}

 We only consider  simple  graphs throughout  paper.  The notations we used are standard from
the monograph written by Bondy and Murty \cite{Bondy76, Bondy08}. Let $G$ be a simple connected  graph with vertex set $V(G)$ and edge set $E(G)$ such that $|V(G)|=v(G)$ and $|E(G)|=e(G)$. We usually write $n$ and $m$ 
for the number of vertices and edges respectively. 
We use $d(v)$ to denote the degree of the vertex $v$ in $G$. 
The minimum degree is denoted by  $\delta(G)$.  
We write $K_n$ and $I_n$ 
for the complete graph and empty graph on $n$ vertices 
respectively. Let $K_{s,t}$ be the complete bipartite graph 
with  parts of sizes $s$ and $t$. We write $C_t$ 
for the cycle graph on $t$ vertices.  
Let $\omega (G) $ denote the clique number, 
which is defined as the number of vertices in a largest 
complete subgraph in $G$. 
The {vertex-chromatic number} $\chi (G)$ of  $G$ 
is the minimum integer $s\in \mathbb{N}^*$ such that there exists a coloring of 
$V(G)$ with $s$ colors and the adjacent vertices have different colors. 
For two vertex-disjoint graphs $G$ and $H$,   $G\cup H$ denotes the disjoint union of $G$ and $H$; 
 $G\vee H$ denotes the join of  $G$ and $H,$ which is obtained from $G\cup H$ by adding all possible edges between $G$ and $H$. 
 For instance, we have $K_{s,t}=I_s \vee I_t$ and $K_n=K_t \vee K_{n-t}$. 
We use the symbol $i\sim j$ to denote the vertices $i$ and $j$ are adjacent, and $i\nsim j$ otherwise.

 The {\em Tur\'an number} of a graph $F$ is the maximum number of edges 
  in an $n$-vertex graph without a subgraph isomorphic to $F$, and 
  it is usually  denoted by $\mathrm{ex}(n, F)$. 
  We say that a graph $G$ is $F$-free if it does not contain 
  an isomorphic copy of $F$ as a subgraph. 
  A graph on $n$ vertices with no subgraph $F$ and with $\mathrm{ex}(n, F)$ edges is called an {\em extremal graph} for $F$ and we denote by $\mathrm{Ex}(n, F)$ the set of all extremal graphs on $n$ vertices for $F$. 
  It is  a cornerstone of extremal graph theory 
  to 
  understand $\mathrm{ex}(n, F)$ and $\mathrm{Ex}(n, F)$ for various graphs $F$; 
  see \cite{FS13, Kee2011,Sim2013} for surveys. 
 
 \medskip 
 \noindent 
 {\bf Question 1.}(Extremal graph problem) 
{\it What is the maximum number of edges of a graph $G$ on $n$
vertices without a subgraph isomorphic to a given graph $F$?
}
\medskip 

  In 1941, Tur\'{a}n \cite{Tur1941} posed the  natural question of determining
 $\mathrm{ex}(n,K_{r+1})$ for $r\ge 2$. 
 Let $T_r(n)$ denote the complete $r$-partite graph on $n$ vertices where 
 its part sizes are as equal as possible. 
 In other words, $T_r(n)=K_{t_1,t_2,\ldots ,t_r}$, 
 the complete $r$-partite graph on vertex classes with sizes $t_1,t_2,\ldots ,t_r$, 
 where $\sum_{i=1}^r t_i =n$ 
 and $|t_i-t_j|\le 1$ for all $i\neq j$.  
Tur\'{a}n \cite{Tur1941} 
(see \cite[p. 294]{Bollobas78})  extended a result of Mantel \cite{Man1907} 
and  obtained that if $G$ is an $n$-vertex graph containing no $K_{r+1}$, 
then $e(G)\le e(T_r(n))$, equality holds if and only if $G=T_r(n)$.  
There are many extensions and generalizations on Tur\'{a}n's result.  
The problem of determining $\mathrm{ex}(n, F)$ is usually called the 
Tur\'{a}n-type extremal problem. 
 The most celebrated extension always attributes to a result of 
 Erd\H{o}s, Stone and Simonovits \cite{ES46,ES66}, which  states that 
 \begin{equation*} 
  \mathrm{ex}(n,F) = \left( 1- \frac{1}{\chi (F) -1} + o(1) \right) 
  \frac{n^2}{2}, 
  \end{equation*}
 where $\chi (F)$ is the vertex-chromatic number of $F$. 
 This provides good asymptotic estimates for the extremal numbers of non-bipartite graphs. 
 However, for bipartite graphs, where $\chi (F)=2$, it only gives the bound 
 $\mathrm{ex}(n,F)=o(n^2)$. 
 Although there have been numerous attempts on finding better bounds 
 of $\mathrm{ex}(n,F)$ for various bipartite graphs $F$, we know very little in this case. 
 The history of such a case began in 1954 with 
 the theorem of K\H{o}vari, S\'{o}s and Tur\'{a}n  \cite{KST54}, which states 
 that if $K_{s,t}$ is the complete bipartite graph with vertex classes of size $s\ge t$, 
 then $\mathrm{ex}(n,K_{s,t})=O(n^{2-1/t})$; see \cite{Furedi96,Furedi96b} 
 for more details.  
In particular, we refer the interested reader to 
  the comprehensive survey  \cite{FS13}.

  Let $G$ be a simple graph on $n$ vertices. 
The \emph{adjacency matrix} of $G$ is defined as 
$A(G)=[a_{ij}]_{n \times n}$ where $a_{ij}=1$ if two vertices $v_i$ and $v_j$ are adjacent in $G$, and $a_{ij}=0$ otherwise.   
We say that $G$ has eigenvalues $\lambda_1 ,\ldots ,\lambda_n$ if these values are eigenvalues of 
the adjacency matrix $A(G)$. 
Let $\lambda (G)$ be the maximum  value in absolute 
 among the eigenvalues of $G$, which is 
 known as the {\it spectral radius} of graph $G$, 
 that is, 
$ \lambda (G) = \max \{|\lambda | : \text{$\lambda$ 
 is an eigenvalue of $G$}\}$.
By the Perron--Frobenius Theorem \cite[p. 534]{Horn13}, 
the spectral radius of a graph $G$ is actually 
the largest eigenvalue of $G$ 
since the adjacency matrix $A(G)$ is nonnegative.  
We usually write $\lambda_1(G)$ for the spectral radius of $G$. 
The spectral radius of a graph sometimes can give some information  
about the structure of graphs. 
For example, it is well-known  that the average degree of $G$ is at most $\lambda (G)$, which is at most the maximum degree of $G$. 
 The \emph{diagonal matrix} of $G$ is $D(G)=[d_{ii}]_{n \times n}$ with diagonal entry  $d_{ii}=d(i)$, the degree of vertex $i$. 
 The \emph{signless Laplacian matrix} 
 is defined as $Q(G)= D(G)+A(G).$ The largest eigenvalue of $Q(G)$, denoted by $q (G)$, is called the  $Q$-\textit{index} or the {\it signless Laplacian  spectral radius} of $G$. 
For  people that working on spectral graph theory,   one of the most well-known problems  is the
Brualdi--Solheid problem \cite{brualdi}, 
which states that 

 \medskip 
 \noindent 
 {\bf Question 2.} (Brualdi--Solheid problem)
{\it Given a set ${\cal
{G}}$ of graphs, find a tight upper bound for the spectral radius in
${\cal
{G}}$ and characterize the extremal graphs.
}
\medskip 

This problem is well studied in the
literature for many classes of graphs, such as  graphs with 
cut vertices \cite{BermanZhangXD}, 
given diameter \cite{HanSte08}, edge chromatic number \cite{FengLAA16}, domination number
\cite{Stevanovic}.
   For the  $Q$-index counterpart of the above problem, Zhang \cite{Zhangxiaodong} gave  the  $Q$-index of graphs with given degree sequence, Zhou \cite{Zhoubo} studied the $Q$-index and Hamiltonicity.  Also, from both theoretical and practical viewpoint,
the  eigenvalues of graphs have been successfully used in many   other disciplines,  one may
refer to \cite{LiShiEnergy, ZhangMinjieDAM, ZhangMinjieAMC}.

In this paper we shall consider spectral analogues 
of Tur\'{a}n-type problems for graphs. 
That is, determining $\mathrm{ex}_{\lambda}(n,F)= \max \{\lambda (G) : 
|G|=n, F\nsubseteq G \}$.  It is well-known that 
\begin{equation*}
 \mathrm{ex}(n,F) \le \frac{n}{2} \mathrm{ex}_{\lambda}(n,F) 
 \end{equation*}
because of the  fundamental inequality ${2m}/{n} \le \lambda (G)$.  
For most graphs, this study is again fairly complete 
due in large part to a longstanding work of Nikiforov \cite{Niki2011}.  
The Brualdi--Solheid problem asks to determine the maximum spectral radius of a graph $G$ on 
$n$ vertices belonging to a specified class of graphs.  
Analogous to the Brualdi--Solheid problem, the following  
question regarding the adjacency  spectral radius
is a natural extension.

 \medskip 
 \noindent 
 {\bf Question 3.} (Spectral extremal problem)
{\it What is the maximum spectral radius of a graph $G$ on $n$
vertices without a subgraph isomorphic to a given graph $F$?
}
\medskip 

This problem regarding the adjacency  spectral radius
was early proposed in \cite{NikiCPC02}. 
Wilf \cite{Wilf86} and Nikiforov \cite{NikiCPC02} 
obtained spectral strengthening of Tur\'an's theorem 
when the forbidden  substructure is the complete graph. 
Soon after, 
Nikiforov \cite{Niki2007b} showed that if $G$ is a $K_{r+1}$-free graph on $n$ vertices, 
then $\lambda (G)\le \lambda (T_r(n))$, 
equality holds if and only if $G=T_r(n)$. 
Moreover, 
Nikiforov \cite{Niki2007b} (when $n$ is odd), 
and Zhai and Wang \cite{ZW2012} (when $n$ is even) 
determined the maximum spectral radius 
of $K_{2,2}$-free graphs. 
Furthermore, Nikiforov \cite{Niki2010}, Babai and Guiduli \cite{BG2007} 
independently  obtained  the spectral generalization of the 
theorem of K\H{o}vari, S\'os and Tur\'an  when the forbidden graph is the complete bipartite graph $K_{s,t}$. 
Finally, Nikiforov \cite{Niki2010b} characterized 
the spectral radius of graphs without paths and cycles of specified length. 
In addition, Fiedler and Nikiforov  \cite{FiedlerNikif} obtained tight sufficient conditions for
graphs to be Hamiltonian or traceable. 
For many other spectral analogues of results in extremal graph theory 
we refer the reader to the survey  \cite{Niki2011}. 
It is worth mentioning that a corresponding spectral extension \cite{Nikicpc2009} of 
the theorem of Erd\H{o}s, Stone and Simonovits \cite{ES46,ES66} states that 
\[  \mathrm{ex}_{\lambda}(n,F) = \left( 1- \frac{1}{\chi (F)-1} +o(1) \right)n. \]
In the following sections, 
we shall survey a large number of extremal graph results 
and its corresponding spectral  extremal 
results in some details.

\section{Extremal spectral  problem for Tur\'{a}n type}

\subsection{Spectral  problem for complete graphs} 

The first theorem of extremal graph theory is commonly regarded as  
a theorem of Mantel \cite{Man1907}, which shows
that the extremal triangle-free graphs are balanced complete  bipartite graphs. 
There are many proofs of this result; see, e.g., 
\cite{Bollobas78,Sha2016,Zhao2019}.  

\begin{theorem}[Mantel \cite{Man1907}] \label{thmmantel}
If $G$ is an $n$-vertex graph with no triangle, 
then 
\[   e(G) \le \left\lfloor \frac{n^2}{4} \right\rfloor,\] 
equality holds if and only if  $G=K_{\lfloor n/2\rfloor , \lceil n/2\rceil }$.  
\end{theorem}  

A natural extension of Mantel's theorem 
is the  famous Tur\'{a}n theorem, 
which extended the forbidden subgraph 
$K_3$ to large complete subgraph $K_{r+1}$ 
for every integer $r\ge 2$. 
The  {\it Tur\'{a}n graph} $T_r(n)$ 
is an $n$-vertex complete $r$-partite graph with each part of size $n_i$ such that  
$|n_i-n_j|\le 1$ for all $1\le i,j\le r$, that is, $n_i$ equals 
$\lfloor \frac{n}{r} \rfloor$  or $\lceil \frac{n}{r}\rceil$. 
Sometimes, we  call $T_r(n)$ the {\it balanced complete $r$-partite graph} and
denote the number of its edges   by $t_r(n)$. 
Clearly, we can verify that   
\[  e(T_r(n))=\sum_{0\le i<j \le r-1} 
\left\lfloor \frac{n+i}{r}\right\rfloor 
\left\lfloor \frac{n+j}{r}\right\rfloor. \] 
 Moreover, 
if we write $n=r\!\cdot \!\lfloor \frac{n}{r} \rfloor +s$ for 
some $0\le s<r$, then 
\[  e(T_r(n))=\left( 1-\frac{1}{r}\right)\frac{(n^2-s^2)}{2} 
+ {s \choose 2}.\] 
In particular, we have $T_2(n)=K_{\lfloor \frac{n}{2} \rfloor, \lceil \frac{n}{2} \rceil}$ 
and $t_2(n)=\lfloor \frac{n}{2} \rfloor \lceil \frac{n}{2}\rceil 
= \lfloor {n^2}/{4} \rfloor$.

\begin{theorem}[Tur\'{a}n \cite{Tur1941}, weak version] \label{thmturanweak}
If $G$ is a graph on $n$ vertices contains no copy of the complete graph $K_{r+1}$, 
then 
\[  e(G)\le \left( 1-\frac{1}{r}\right)\frac{n^2}{2}.\]
Furthermore, equality holds if and only if $r$ divides $n$ and $G=T_r(n)$. 
\end{theorem}

By an easy calculation, we can see that 
$e(T_r(n)) \le (1- \frac{1}{r}) \frac{n^2}{2}$, 
equality holds if and only if $r$ divides $n$. 
This inequality also can be seen  from Theorem \ref{thmturanweak} since $T_r(n)$ does not contain $K_{r+1}$ 
as a subgraph.  
More generally, one can prove the strong version 
of the Tur\'{a}n theorem in a similar way.

\begin{theorem}[Tur\'{a}n \cite{Tur1941}, strong version] \label{thmturanstrong} 
If  $G$ is a  graph on $n$ vertices with  no copy of $K_{r+1}$, 
then 
\[  e(G)\le e(T_r(n)),\]
equality holds if and only if $G$ is the Tur\'{a}n graph $T_r(n)$.
\end{theorem} 

Tur\'{a}n's theorem initiated the rapid development 
of extremal graph theory and was
rediscovered many times with various different proofs; 
see, e.g., \cite[Chapter 40]{AZ2014}. 
In the sequel, we shall present an elegant proof 
of the weaker version of Tur\'{a}n's theorem. This proof 
was proposed by Motzkin and 
Straus \cite{MS1965} 
based on applying the optimization method, 
which is now known as the Lagrangian method;  
see \cite{Tang2014,Tang2016} for more generalizations, and 
we recommend highly a recent survey written by 
Yuejian Peng, \href{https://t.cnki.net/kcms/detail?v=i1BWWZje93T3zm6lVz4Va_VOxAEHZeCRYeAb_rYSsCs6UqaeMjsSA20agt1my-PTL1zmrwMCkMW-fC8rAmOG8ihAVTJbNyzOvslPwbBV2GV3uiVvfRqBDN4BgRXWK3oy&uniplatform=NZKPT}{Hypergraph 
Lagrangian Function} (click). 
It is worth noting that the theorem of Motzkin and 
Straus was also applied to the spectral version of Tur\'{a}n's theorem; 
see, e.g., \cite{NikiCPC02,Niki2006laa,Niki2009jctb} for related results.  

\begin{definition} \cite{MS1965}
Let $G$ be a graph on vertex set $[n]$. We associate a  function 
\[ \lambda (G, \bm{x}):=\sum\limits_{\{i,j\}\in E(G)} x_ix_j,\quad (\forall 
\bm{x}\in \mathbb{R}^n). \] 
We denote $\Delta :=\{\bm{x}: 
x_1,\ldots ,x_n\ge 0, \sum_{i=1}^{n}x_i =1 \}$. 
The $Lagrangian$ of $G$ is defined as  
the maximum of $\lambda (G, \bm{x})$ over all 
$\bm{x}\in \Delta $, that is, 
\[ \lambda (G):=\max \bigl\{ \lambda (G, \bm{x}): 
{ \bm{x} \in \Delta } \bigr\}. \]
\end{definition}

We remark  that the notation $\lambda (G)$ 
is frequently denoted by the spectral radius of $G$ beyond this section.  
Recall that $\omega (G)$ is the 
number of  vertices in a largest complete subgraph of $G$, 
which is called the {\it clique number} of $G$. 
In 1965, Motzkin and Straus \cite{MS1965} published a new proof 
of the Tur\'{a}n theorem by proving the following 
result. 

\begin{theorem}[Motzkin--Straus \cite{MS1965}]  \label{thmmz}
For any graph $G$, we have 
\[ \lambda (G) = \frac{1}{2}\left( 1- \frac{1}{\omega (G)}\right). \]
\end{theorem}

\begin{proof}
First of all, we write  $\omega (G)=\omega$ for short. 
Let $K_{\omega}$ be the largest clique in $G$.  
We define the vector $\bm{w}$ by setting $w_i={1}/{\omega}$ for each $i\in V(K_{\omega})$ and 
$w_i =0 $ for $i \in V\setminus V(K_{\omega})$. 
Clearly, we have $\bm{w}$ is nonnegative and $\sum_{i=1}^n w_i =1$. 
Here, we say that a vector is nonnegative if all of its coordinates  are nonnegative. 
Hence, we can get 
\begin{equation} \label{eqlag1}
 \lambda (G)\ge 
\lambda ( G, \bm{w})=\sum_{\{i,j\}\in E(K_{\omega})} \frac{1}{\omega^2}= 
 \frac{1}{2}\left( 1-\frac{1}{\omega} \right). 
 \end{equation}

There may exist many nonnegative vectors $\bm{x}$ with $\sum_{i=1}^n x_i=1$ satisfying  
$\lambda (G,\bm{x})=\lambda (G)$,  we further select a vector $\bm{x}$ among these vectors such that 
{ the cardinality of its supporting set 
$|\mathrm{Supp}(\bm{x})|:=\# \{i\in [n]: x_i >0\}$ is minimal. }
We shall complete the proof by the following two steps. 

\medskip 
{\bf Step 1.}~
{\it We first show that the vertex set $\mathrm{Supp}(\bm{x})$  forms a clique in $G$. }

\medskip

If  $\mathrm{Supp}(\bm{x})$ does not form a clique in $G$, we may suppose that 
 $\{i,j\}\notin E(G)$ for some $i,j\in \mathrm{Supp}(\bm{x})$. 
 Denoting $s_v:=\sum_{u\sim v} x_u$ for each vertex $v\in V(G)$. 
We may assume without loss of generality that $s_i\ge s_j$.  
Now we move the weight $x_j$ from vertex $j$ to vertex $i$, 
that is, the new weight of vertex $i$ is $x_i+x_j$, while the weight of vertex $j$ drops to $0$. 
In other words, we construct a new nonnegative vector $\bm{y}$ such that 
$y_i=x_i+x_j,y_j=0$ and $y_k=x_k$ for $k \in V\setminus \{i,j\}$. We compute 
\[ f_G(\bm{y}) -f_G(\bm{x})= y_is_i +y_js_j - x_is_i -x_js_j= x_js_i -x_js_j \ge 0. \]
Therefore, $\bm{y}$  also satisfies $\lambda (G,\bm{y})= \lambda (G)$ 
and $|\mathrm{Supp}(\bm{y})|< |\mathrm{Supp}(\bm{x})|$, which contradicts the minimality of the supporting set  of $\bm{x}$.

\medskip 
 {\bf Step 2.}~
{\it According to Step 1, the clique formed by $\mathrm{Supp}(\bm{x})$ 
is now denoted by $K_t$ for some integer $t\le \omega$,  
we next prove that $x_i=1/t$ for all $i\in \mathrm{Supp}(\bm{x})$. }

\medskip 
Otherwise, if $x_i>x_j>0$ for some $i,j\in \mathrm{Supp}(\bm{x})$, then 
we choose $\varepsilon $ with $0<\varepsilon <x_i-x_j$, 
and change the weight $x_i$ to $x_i- \varepsilon$ and $x_j$ to $x_j +\varepsilon$, 
that is to say, we define a new non-negative vector $\bm{z}$ 
with $z_i=x_i -\varepsilon , z_j =x_j + \varepsilon$ and $z_k=x_k$ for $k \in V\setminus \{i,j\}$. 
We obtain from Step 1 that 
\[ \lambda (G,\bm{z}) -\lambda (G, \bm{x}) = \sum_{\{i,j\}\in E(K_t)} z_iz_j - 
\sum_{\{i,j\}\in E(K_t)}x_ix_j 
= \varepsilon (x_i-x_j)-\varepsilon^2 >0. \]
This contradicts with the fact $\lambda ( G, \bm{x})=\lambda (G)$. 
Therefore, the values of $x_i$ are all equal for 
every $i \in V(K_t)$, 
which gives $x_i=1/t$ for every $i\in V(K_t)$. Therefore, 
\begin{equation} \label{eqlag2} \lambda (G)= 
\lambda ( G, \bm{x})=\sum_{\{i,j\}\in E(K_{t})} 
\frac{1}{t^2}= 
 \frac{1}{2}\left( 1-\frac{1}{t} \right)\le 
  \frac{1}{2}\left( 1-\frac{1}{\omega} \right). 
  \end{equation}
Combining (\ref{eqlag1}) and (\ref{eqlag2}), we get 
$\lambda (G)= \frac{1}{2}(1-\frac{1}{\omega})$. 
Moreover, 
the argument also showed that 
the nonzero weights of the maximizer $\bm{x}$ with minimal supporting set 
are concentrated on a clique, 
which is a largest clique of $G$, 
and all weights on this clique are equal.
 \end{proof}

\begin{corollary}
Motzkin--Straus' Theorem \ref{thmmz} $\Rightarrow$  Weak Tur\'{a}n's Theorem \ref{thmturanweak}. 
\end{corollary}

\begin{proof}
Let $G$ be a graph that does not contain 
$K_{r+1}$ as a subgraph. Clearly, we have 
 $\omega (G)\le r$. 
We define the vector 
$ \bm{x}=\left( {1}/{n},{1}/{n},\ldots ,{1}/{n}\right)$, 
then  
\[ \lambda (G, \bm{x}) = \sum_{\{i,j\}\in E(G)} x_ix_j = \frac{m}{n^2}. \]
By the Motzkin--Straus theorem, we know that 
\[ \lambda (G, \bm{x})  \le \lambda (G)=\frac{1}{2}\left( 1-\frac{1}{\omega }\right)\le 
\frac{1}{2}\left( 1-\frac{1}{r }\right). \]
Combining the above two results, we can get  
$m\le (1-\frac{1}{r})\frac{n^2}{2}$. 
\end{proof}

For convenience, we usually write the Motzkin--Straus theorem 
as the following form. 

\begin{theorem}[Equivalent]
Let $G$ be an $n$-vertex graph. If $(x_1,\ldots ,x_n)\in \mathbb{R}^n$ such that  
$x_1,\ldots ,x_n\ge 0$ and $x_1+\cdots +x_n=1$, then 
\[  \sum_{\{i,j\}\in E(G)} x_ix_j \le \frac{\omega (G)-1}{2\omega (G)}.  \]
 \end{theorem}
 
 As promised, 
we shall  compare the extremal results based on the number of edges 
and the spectral extremal results  through the whole survey. 
 In what follows, we shall introduce 
 some spectral version of  Mantel's theorem 
 and Tur\'{a}n's theorem. 
 Interestingly, the spectral versions 
 of these extremal results 
 are greatly related to the Motzkin--Straus theorem. 

\begin{theorem}[Nosal, 1970] \label{thmnosal}
Let $G$ be a  graph with no triangle. Then 
\[  \lambda (G)\le \sqrt{e(G)} , \] 
 equality holds if and only if 
$G$ is a complete bipartite graph. 
\end{theorem} 

This theorem together with Mantel's Theorem \ref{thmmantel} 
yields the following corollary. 

\begin{corollary}\label{thm115}
If $G$ is an $n$-vertex graph without triangles, then 
\[ \lambda (G)\le \sqrt{\lfloor n^2/4\rfloor} \le {n}/{2}. \] 
\end{corollary}

In 1986, Wilf extended this corollary to the following theorem.

\begin{theorem}[Wilf  \cite{Wilf86}] \label{thmwilf}
Let $\lambda (G)$ be the largest eigenvalue of $G$. Then 
\begin{equation*} 
 { \lambda (G)\le \left(1-\frac{1}{\omega (G)}\right)n  } .   
\end{equation*}
\end{theorem}

\begin{corollary}[Equivalent]
Let $G$ be a $K_{r+1}$-free graph on $n$ vertices. Then  
\[ \lambda (G) \le \left(1-\frac{1}{r }\right)n. \]
\end{corollary} 

The proof is based on the Motzkin--Straus Theorem~\ref{thmmz}. 
Recalling that 
\begin{equation*}
{  1-\frac{1}{\omega (G)} =
\max\limits_{\bm{x}\ge 0,\lVert\bm{x} \rVert_1=1}
\bm{x}^TA(G)\bm{x} =\max\limits_{\bm{x}\ge 0}\frac{\bm{x}^TA(G)\bm{x}}{
\lVert \bm{x} \rVert_1^2}. }
\end{equation*}

\begin{proof} 
We show our proof in three steps. 
\begin{itemize}
\item 
By the Perron--Frobenius Theorem, there exists 
a nonnegative vector $\bm{z}$ such that   
\[ A(G)\bm{z} =\lambda_1 \bm{z} \Rightarrow 
\bm{z}^TA(G)\bm{z} =\lambda_1 \bm{z}^T\bm{z} 
\Rightarrow  {\lambda_1}= \frac{\bm{z}^T A(G)\bm{z}}{\bm{z}^T\bm{z}} . \]

\item 
By the Motzkin--Straus Theorem, we have 
\[ { \frac{\bm{z}^T A(G)\bm{z}}{\bigl( \sum_{i=1}^n z_i\bigr)^2} \le 1-\frac{1}{\omega}.  }  \]

\item 
By the Cauchy--Schwarz inequality, we get 
\[  \bm{z}^T \bm{z}=  \left( \sum_{i=1}^n z_i^2 \right) \ge 
  \frac{1}{n}\left( \sum_{i=1}^n z_i \right)^2.  
\]
\end{itemize}
Combining the above three results, we can get the required inequality. 
\end{proof}

\noindent 
{\bf Remark.} 
We have the following potential applications. 
Since the Tur\'{a}n graph $T_r(n)$ does not contain 
$K_{r+1}$ as a subgraph, 
the Wilf theorem implies that 
$\lambda (T_r(n)) \le \left(1- {1}/{r} \right)n$. 
Moreover, the Rayleigh theorem 
implies that $\lambda (G) =\max_{\bm{x}\in \mathbb{R}^n} 
{\bm{x}^TA\bm{x}}/{\bm{x}^T\bm{x}} \ge {2m}/{n}$, 
hence the Wilf theorem 
can imply the weak version of the Tur\'{a}n Theorem \ref{thmturanweak}.

\medskip  

Recall that $q(G)$ denotes the 
  signless Laplacian spectral radius, i.e., 
 the largest eigenvalue of 
 the {\it signless Laplacian matrix} $Q(G)=D(G) + 
 A(G)$, where $D(G)=\mathrm{diag}(d_1$, $\ldots ,d_n )$ 
 is the degree diagonal matrix and 
 $A(G)$ is the adjacency matrix of a graph $G$. 
Note that $Q(G)=D(G)+A(G)=L(G)+2A(G)$, 
where $L(G)=D(G)-A(G)$ is the so-called Laplacian matrix, 
which is a positive semidefinite matrix, so we can get 
$2\lambda (G)\le q(G)$.

In 2012, Abreu and Nikiforov \cite{AN2012ela,AN2013ela} 
enhanced the Wilf theorem in terms of the signless Laplacian spectral radius. 
They \cite{AN2013ela} proved that  
$ q(G)\le 2 \left(1- \frac{1}{\omega }\right)n $, 
where $\omega$ is the number of vertices of a largest clique of $G$. 
In other words, this result can be equivalently written as 
that if $G$ is a $K_{r+1}$-free graph on $n$ vertices, then 
$q(G)\le 2 \left(1- \frac{1}{r}\right)n $.

\begin{theorem}[Abreu--Nikiforov \cite{AN2013ela}] 
\label{thm212} 
Let $G$ be an $n$-vertex graph 
containing no copy of $K_{r+1}$. Then 
\[  q(G)\le \left(1-\frac{1}{r }\right)2n  .\]
\end{theorem}

We remark here that Abreu and Nikiforov did not characterize the 
extremal graphs attaining the equality. 
We can show by a careful examination that the equality holds 
if and only if $G$ is a complete bipartite  for $r=2$, 
or $r$ divides $n$ and $G=T_r(n)$ for $r\ge 3$.

\medskip 

In 2002, Nikiforov \cite{NikiCPC02} generalized the Wilf theorem 
by given the number of edges. 
The extremal graphs  attaining the equality  in 
Theorem \ref{thmnikiforov} 
was later characterized in \cite[Theorem 2]{Niki2009jctb}; 
 see \cite{Niki2006laa} for more discussions.

\begin{theorem}[Nikiforov \cite{NikiCPC02,Niki2009jctb}]  \label{thmnikiforov}
Let $\lambda (G)$ be the largest eigenvalue of $G$. Then 
\begin{equation*}
 { \lambda (G) \le \sqrt{2m\left( 1-\frac{1}{\omega 
 (G)}\right)} } . 
\end{equation*}
Moreover,  the equality holds 
if and only if 
$G$ is a complete bipartite graph for $r=2$, 
or a complete regular $r$-partite graph for $r\ge 3$ 
with possibly some isolated vertices.  
\end{theorem}

\begin{corollary}[Equivalent] \label{coroniki}
Let $G$ be a $K_{r+1}$-free graph with $m$ edges. Then  
\[ \lambda (G) \le \sqrt{2m\left( 1-\frac{1}{r}\right)}. \]
\end{corollary}

This result has some similarity with the Wilf theorem. 
Its proof is also based on the celebrated Motzkin--Straus Theorem, 
which states that for any  $x_1,x_2,\ldots ,x_n\ge 0$, we have 
\[ { \frac{ \sum_{\{i,j\}\in E(G)} x_ix_j }{\bigl( \sum_{i=1}^n x_i\bigr)^2}
 \le \frac{1}{2}\left( 1- \frac{1}{\omega}\right).} \]
 
 \begin{proof}
Let $\bm{y}=(y_1,\ldots ,y_n)$ 
be an (nonnegative) eigenvector of $\lambda (G)$. 
\begin{itemize}
\item 
Using the Motzkin--Straus Theorem with 
respect to $(y_1^2,y_2^2,\ldots ,y_n^2)$, we have 
\[  { \frac{ \left( \sum_{\{i,j\}\in E(G)} y_i^2y_j^2 \right)}{ 
\bigl( \sum_{i=1}^n y_i^2\bigr)^2}  \le 
\frac{1}{2}\left( 1- \frac{1}{\omega}\right).} \]

\item 
On the other hand, we can get  by Cauchy--Schwarz's inequality that 
\begin{align*} 
\lambda_1^2 
=\frac{(\bm{y}^TA(G)\bm{y})^2}{(\bm{y}^T\bm{y})^2} 
= \frac{\left( 2\sum_{\{i,j\}\in E(G)} y_iy_j \right)^2}{ 
\bigl( \sum_{i=1}^n y_i^2\bigr)^2} 
\le    \frac{4 m \left( \sum_{\{i,j\}\in E(G)} y_i^2y_j^2 \right)}{ 
\bigl( \sum_{i=1}^n y_i^2\bigr)^2} .
\end{align*}
\end{itemize}
This completes the proof by taking square root. 
\end{proof}

\noindent 
{\bf Remark.} 
We have the following interesting applications. 
Note that the Tur\'{a}n graph $T_r(n)$ 
does not contain $K_{r+1}$ as a subgraph,  
Nikiforov's result tells us that 
\[  \lambda (T_r(n)) 
\le \sqrt{2t_r(n)\left( 1-\frac{1}{r}\right)}
\le \left(1-\frac{1}{r} \right)n. \]  
Moreover, 
the case $r=2$ in Nikiforov's result reduces to 
 the Nosal Theorem \ref{thmnosal}. 
 Additionally, 
applying  the lower bound $\lambda (G)\ge {2m}/{n}$, 
we can see that Nikiforov's result implies 
the weak version of the Tur\'{a}n Theorem \ref{thmturanweak}. 
Furthermore, Nikiforov's result implies the Wilf Theorem \ref{thmwilf} 
 by applying the weak version of Tur\'{a}n theorem 
 $m\le (1- \frac{1}{r}) \frac{n^2}{2}$. 
 
 \medskip 
 
 We make the following conjecture for signless Laplacian 
 spectral radius. 

\begin{conjecture} \label{conj115}
Let $G$ be a $K_{r+1}$-free graph with $m$ edges. Then  
\[ q (G) \le \sqrt{8m\left( 1-\frac{1}{r}\right)}. \]
\end{conjecture}

\noindent 
{\bf Remark.} 
The answer of this conjecture is negative. 
After our manuscript announced on arXiv: 2111.03309, 
Clive Elphick and Vladimir Nikiforov independently told us that the star graph is a counter-example to 
Conjecture \ref{conj115}. Indeed, by setting $r=2$, the $n$-vertex star graph $K_{1,n-1}$ 
is clearly $K_3$-free, and $q(K_{1,n-1})=n$ but $\sqrt{8m\left( 1-\frac{1}{r}\right)}=\sqrt{4(n-1)}$. 
Moreover, the complete bipartite graph $K_{s,n-s}$ is obviously $K_{r+1}$-free for every $s\ge 1$ and $r\ge 2$, 
we observe that $q(K_{s,n-s}) =n$ but $\sqrt{8m\left( 1- \frac{1}{r}\right)} 
= \sqrt{s(n-s)(1-\frac{1}{r})}$.  

\medskip 
In 2002, Bollob\'{a}s and Nikiforov \cite{BNJCTB07} 
proposed the following conjecture in terms of the first and 
second largest eigenvalues of $G$. 
 Let $\lambda_k(G)$  be the  $k$-th largest eigenvalue of $G$.

\begin{conjecture}[Bollob\'{a}s--Nikiforov \cite{BNJCTB07}] 
If $G$ does not contain a copy of $K_{r+1}$, then  
\begin{equation*}
 \lambda_1^2(G)+ \lambda_2^2(G) 
\le 2m\Bigl( 1-\frac{1}{r}\Bigr) . 
\end{equation*}
\end{conjecture} 

In 2021, the special case $r=2$ 
 of this conjecture was proved by Lin, Ning and Wu \cite{LNW2021} 
 by using tools from doubly stochastic matrix theory, 
and also characterize all families of extremal graphs; 
see Nikiforov \cite{Niki2021} for related results. 
 The general case $r\ge 3$ remains open.  
 In 2022, Li, Sun and Yu \cite{LSY2022} presented the characterization of graphs without short odd cycles.
Later, Li and Peng \cite{LP2022oddcycle} discussed some related spectral problems for graphs with given size 
containing no short odd cycles.  

Recall in \autoref{thmnosal} that 
if $G$ is a $C_3$-free graph on $n$ vertices with $m$ edges, 
then $\lambda (G) \le \sqrt{m}$, 
equality holds if and only if  
$G =K_{p,q}$ with $n-p-q$ isolated vertices.  
In 2009, Nikiforov \cite{Niki2009} proved an interesting result, 
which gives an upper bound on the spectral radius for $C_4$-free graph 
with given number of edges.  
Note that the star graph $K_{m,1}$ has $m$ edges and $m+1$ 
vertices, and it contains no copy of $C_4$. 
Let $S_{m,1}$ be 
the graph obtained from the 
star $K_{m-1,1}$ by adding an edge within its independent set. 
Namely $S_{m,1}=K_{m-1,1}^{+}$. 
Note that the graph $S_{m,1}$ has $m$ edges and  $m$ vertices. 
For $m=4,5,6,7,8$, we can verify that $\lambda (S_{m,1})> \sqrt{m}$; 
for $m=9$, we have $\lambda (S_{9,1})=3$;  
for $m=10,11,\ldots $, we can get  
$\lambda (S_{m,1}) < \sqrt{m}$. For example, 
we can calculate that $\lambda (S_{10,1}) \approx 3.151 < \sqrt{10}\approx 3.162$.  

\begin{theorem}[Nikiforov \cite{Niki2009}] \label{thmniki2009}
Let $m\ge 9$ and $G$ be a graph with $m$ edges. 
If $G$ contains no $C_4$, then 
\[  \lambda (G)\le \sqrt{m},\]  
equality holds if and only if $G$ is a star $K_{1,m}$ or 
$S_{9,1}$ with possibly some isolated vertices. 
\end{theorem}

Motivated by  Nosal's Theorem \ref{thmnosal} 
and Nikiforov's Theorems \ref{thmniki2009}, 
Zhai, Lin and Shu \cite{ZLS2021} proved the following 
analogue when the forbidden graph is the bipartite graph $K_{2,r+1}$. 

\begin{theorem}[Zhai--Lin--Shu \cite{ZLS2021}] \label{thmzls}
Let $r\ge 2$ and $G$ be a $K_{2,r+1}$-free graph of size 
$m\ge 16r^2$. 
Then 
\[   \lambda (G) \le \sqrt{m},  \] 
equality holds if and only if $G$ is a complete bipartite graph 
with possibly some isolated vertices. 
\end{theorem}

Moreover, 
we denote by $B_{k}$  the book graph with $k$ pages, i.e., 
the graph obtained from $k$ triangles by sharing a common edge. 
In \cite[Conjecture 5.2]{ZLS2021}, Zhai, Lin and Shu presented a
conjecture, which states that 
if $G$ is a $B_k$-free graph with sufficiently large size $m$, then 
$\lambda (G) \le \sqrt{m}$,  equality holds if and only if $G$ is a complete bipartite graph 
with possibly some isolated vertices. Later, 
Nikiforov \cite{Niki2021} confirmed this conjecture  by the following stronger theorem. 
Let $bk(G)$ denote the booksize of $G$, that is, the maximum 
number of triangles with a common edge in $G$. 

\begin{theorem}[Nikiforov, 2021] \label{thmNiki2021}
If $G$ is a graph with $m$ edges and $\lambda (G)\ge \sqrt{m}$, 
then 
\[  bk(G)> \frac{1}{12}\sqrt[4]{m}, \] 
unless  $G$ is a complete bipartite graph with possibly some isolated vertices. 
\end{theorem}

We end this section with a more general conjecture. 
Let $B_{r,k}$ be the generalized book, 
which is a graph by joining every vertex of a clique $K_r$ 
to every vertex of an independent set $I_k$. 
In the language of the join of graphs, we have  
\[  B_{r,k}=K_r \vee I_k. \]
Clearly, we can see that $B_{2,k}=B_k$ 
and $B_{r,k}$ contains $k$ copies of $K_{r+1}$. 
Hence, our conjecture is more general than Corollary \ref{coroniki}.  

\begin{conjecture} 
Let  $G$ be a 
$B_{r,k}$-free graph with sufficiently large size $m$. 
Then  
\[   \lambda (G)\le \sqrt{\left( 1-\frac{1}{r} \right)2m},  \]  
equality holds if and only if $G$ is a complete bipartite graph for $r=2$, 
and $G$ is a complete regular $r$-partite graph for $r\ge 3$ 
with possibly some isolated vertices. 
\end{conjecture}

Note that the Wilf Theorem \ref{thmwilf} and 
Nikiforov Theorem \ref{thmnikiforov} can be viewed as 
the spectral generalization of the weak version of 
Tur\'{a}n Theorem \ref{thmturanweak}. 
In what follows, we shall introduce the spectral generalization  
of the strong version of the Tur\'{a}n Theorem \ref{thmturanstrong}. 
In 2007, Feng, Li and Zhang \cite{FengAML} proved that 
the Tur\'{a}n graph $T_r(n)$ attains the maximum spectral radius 
among all $n$-vertex graphs with given chromatic number.

\begin{theorem}[Feng--Li--Zhang \cite{FengAML}] \label{Feng}
Let $G$ be an $n$-vertex complete $r$-partite graph.   
Then 
\[   \lambda (G)\le \lambda ({T_r(n)}), \]   
equality holds if and only if  $G$ is the Tur\'{a}n graph $T_r(n)$. 
\end{theorem}

Recall that $Q(G)=D(G)+A(G)=L(G)+2A(G)$ and $
L(G) = D(G)-A(G)$ is positive semidefinite, 
which implies $q(G)\ge 2\lambda (G)$. 
The following result on the signless Laplacian spectral radius 
of a complete $r$-partite graph is also well-known 
in the literature.  

\begin{theorem}[Cai--Fan \cite{CF2009,Obo2019,YWS2011}] 
Let $G$ be an $n$-vertex complete $r$-partite graph.   
Then 
\[   q (G)\le q({T_r(n)}),  \]   
equality holds if and only if $r=2$ and $G=K_{s,n-s}$ for some $s$, or $r\ge 3$ and 
$G=T_r(n)$.
\end{theorem}

In the meantime, Nikiforov \cite{Niki2007b} published the following result. 

\begin{theorem}[Spectral Tur\'{a}n Theorem \cite{Niki2007b}] \label{thm460}
Let $G$ be a  graph on $n$ vertices not containing $K_{r+1}$ 
as a subgraph. Then 
\[ \lambda (G)\le \lambda ({T_r(n)}),\]  
equality holds if and only if  $G$ 
is the Tur\'{a}n graph $T_r(n)$.
\end{theorem}

Theorem \ref{thm460} generalized Feng--Li--Zhang's 
result since every $r$-partite graph is $K_{r+1}$-free, as well as 
enhanced Wilf's Theorem \ref{thmwilf} 
since we can verify  that 
$\lambda (T_r(n)) \le (1- \frac{1}{r})n$, 
where the equality holds if and only if $r$ divides $n$.  
Independently, this result on extremal graph problems 
was also established by Guiduli 
in his PH.D. dissertation \cite[pp. 58--61]{Gui1996} 
dating back to 1996 under the guidance of L\'{a}szl\'{o} Babai. 
In addition, the Wilf Theorem \ref{thmwilf} 
can be viewed as the weak version of the 
following spectral Tur\'{a}n Theorem.

The idea of the proof of Theorem \ref{thm460} 
in \cite{Niki2007b} is based on the characteristic polynomial of 
the Tur\'{a}n graph and  two important theorems. The first theorem is also 
presented early by Nikiforov in \cite[Theorem 3.1]{NikiCPC02}, which states that 
if $G$ is a $K_{r+1}$-free graph, then 
\[ \lambda^{r} \le k_2 \lambda^{r -2} +\cdots + 
(j-1)k_j \lambda^{r -j}
+\cdots +(r -1)k_{r }, \]  
where $\lambda =\lambda (G)$ and 
$k_i =k_i(G)$ are the spectral radius and 
the number of $i$-cliques of $G$ respectively. 
In particular, the case $r=2$ yields 
the Nosal Theorem \ref{thmnosal}, that is, 
$\lambda \le \sqrt{k_2} =\sqrt{e(G)}$ for all $K_3$-free graphs. 
The second used theorem is an old result of Zykov \cite{Zykov1949} and 
independently of Erd\H{o}s \cite{Erd1962}, 
which says that 
if $G$ is a $K_{r+1}$-free graph, then 
\[ k_i(G)\le k_i (T_r(n)). \] 
In other words, the Tur\'{a}n graph attains the maximum 
number of cliques of order $i$ among all $K_{r+1}$-free graphs. 
Clearly, the case $i=2$ reduces to the Tur\'{a}n theorem. 

We remark here that there is different proof \cite{Gui1996} for the spectral 
Tur\'{a}n theorem; see, e.g., \cite{KN2014}. The proof given in 
\cite{Gui1996} reduces $K_{r+1}$-free graphs in Theorem \ref{thm460} 
to complete $r$-partite graphs, then one can show that 
the balanced complete $r$-partite graph attains the maximum 
value of the 
spectral radius; see, e.g.,  \cite[Theorem 2.1]{FengAML} or \cite[Theorem 2]{KNY2015} for more details.  
The advantage of this proof is that we can avoid the use of 
the  characteristic polynomials of the adjacency matrices of graphs, 
and this proof allows us to 
extend the spectral result to the $p$-spectral radius.  
The $p$-spectral radius of a graph $G$ was introduced by Keevash, Lenz and Mubayi in \cite{Kee2014} and was defined as 
\[ \lambda^{(p)} (G):=\max \bigl\{ \bm{x}^TA(G) \bm{x}: 
{ \bm{x} \in \mathbb{R}^n, \lVert  \bm{x} \rVert_p =1 } \bigr\}. \] 
Theorem \ref{thm460} was extended by Kang and Nikiforov \cite{KN2014}
to  the $p$-spectral radius for every $p> 1$.  
Note that the $p$-spectral radius  is just an extremal function on graph, and  
there is no characteristic polynomial corresponding to the $p$-spectral radius. 

In 2022, Li and Peng \cite{LP2022second} 
proved a refinement on the spectral Tur\'{a}n Theorem \ref{thm460}. 
They determined the largest spectral radius for  
non-$r$-partite $K_{r+1}$-free graphs.

It is easy to see that 
Theorem \ref{thm460} implies  the weak version of Tur\'{a}n's Theorem. 
A more natural question one may ask might be
 the following one. 

\medskip 
\noindent 
{\bf Question.} 
{\it Does Theorem \ref{thm460} imply the 
strong  Tur\'{a}n's Theorem? }
\medskip 

This question was also proposed and answered in \cite{Gui1996, Niki2009jctb}.  
Speaking rather broadly, the spectral  bound can imply the 
edge bound of Tur\'{a}n's theorem.  
It is well-known that 
${e(G)} \le \frac{n}{2}\lambda (G)$, 
with equality if and only if  $G$ is  regular. 
Although the Tur\'{a}n graph $T_r(n)$
is sometimes not regular, but it is nearly regular. 
Upon calculation, we can verify that 
\begin{equation} \label{eqdeng}
   {e(T_r(n))} = 
    \left\lfloor \frac{n}{2}\lambda (T_r(n)) \right\rfloor.  
  \end{equation}
Let $n=rs  +t$ where $0\le t<r$ 
and $s= \lfloor n/r \rfloor$. 
Hence the Tur\'{a}n graph $T_r(n)$ 
has $t$ parts of size $s+1$ and $r-t$ parts of size $s$, 
and we can compute that 
\begin{align*} 
 e(T_r(n)) &= 
\frac{1}{2} \bigl( (n-s-1)(s+1)t + (n-s)s(r-t) \bigr) .
\end{align*}
and 
\[  \lambda (T_r(n)) = 
\frac{1}{2}(n-2s -1 + \sqrt{n^2 -4sn -2n+4rs^2 +4rs +1}).  \] 
With the help of the observation in (\ref{eqdeng}), 
the spectral Tur\'{a}n theorem implies that  
\[ e(G) \le 
\left\lfloor \frac{n}{2} \lambda (G) \right\rfloor 
\le \left\lfloor \frac{n}{2} \lambda (T_r(n)) \right\rfloor =e(T_r(n)). \] 
Thus the bound of spectral Tur\'{a}n Theorem \ref{thm460} implies the edge bound of the
classical Tur\'{a}n Theorem \ref{thmturanstrong}. 
To some extent, this interesting implication has shed new lights on 
the study of spectral extremal graph theory. 
In addition, the main difficulty in this relation lies in deducing the implication of case of equality, 
since we can not get $\lambda (G)=\lambda (T_r(n))$ by using 
$\lfloor \frac{n}{2} \lambda (G) \rfloor 
\le \lfloor \frac{n}{2} \lambda (T_r(n)) \rfloor$. 
We would like to thank Vladimir Nikiforov for pointing out to us this observation.

Recall that Abreu and Nikiforov (Theorem \ref{thm212}) 
provided a signless Laplacian spectral version of Wilf's theorem, namely $q(G)\le 2 \left(1-  \frac{1}{r}\right)n$ 
for every $n$-vertex $K_{r+1}$-free graph $G$. 
In 2013, 
He, Jin and Zhang \cite[Theorem 1.3]{HJZ2013} proved further that $q(G) \le q(T_r(n))$. 
Their result is regarded as the signless Laplacian spectral version of the strong Tur\'{a}n theorem.  
Since $q(T_r(n))\le 2 \left(1-  \frac{1}{r}\right)n $, and if $r$ divides $n$, then 
$q(T_r(n))= 2 \left(1- \frac{1}{r }\right)n $, we know that Theorem \ref{thmHJZ} 
improved the result of Abreu and Nikiforov and the result of Wilf as well. 
Moreover, they also illustrated that the signless Laplacian spectral version 
 implies the original Tur\'{a}n theorem. 

\begin{theorem}[He--Jin--Zhang \cite{HJZ2013}] \label{thmHJZ}
Let $G$ be a $K_{r+1}$-free graph of order $n$. Then 
\[   q (G)\le q ({T_r(n)}), \] 
equality holds if and only if $r=2$ and $G=K_{s,n-s}$ for some $s$, or $r\ge 3$ and 
$G=T_r(n)$.
\end{theorem}

As mentioned in \cite[Corollary 2.5]{HJZ2013}, 
this spectral bound of the signless Laplacian radius 
implies the edge bound in strong version of the Tur\'{a}n Theorem \ref{thmturanstrong} since 
\[ e(G) \le 
\left\lfloor \frac{n}{4} q (G) \right\rfloor 
\le \left\lfloor \frac{n}{4} q (T_r(n)) \right\rfloor =e(T_r(n)). \] 
From what has been discussed above, 
we  conclude that the edge version of the Tur\'{a}n theorem 
can be deduced from  
the spectral version of either the adjacency matrix or 
the signless Laplacian matrix.  
Along with the development of the 
spectral graph theory, 
it is commonly recognized that the spectral versions 
of extremal graph problems are more  challenging and interesting.

\subsection{Spectral problem for general graphs}

In this section, 
we shall consider the extremal number $\mathrm{ex}(n,F)$ 
for general graph $F$. 
In 1946, Erd\H{o}s and Stone \cite{ES46} proved 
a celebrated theorem, which states that 

\begin{theorem}[Erd\H{o}s--Stone \cite{ES46}] \label{thm360}
For every  $r\ge 1,t\ge 1$ and $\varepsilon >0$, there is $n_0=n_0(r,t,\varepsilon )$ 
such that every graph $G$ on $n\ge n_0$ vertices 
with $e(G)\ge (1-\frac{1}{r}+\varepsilon )\frac{n^2}{2}$ 
 contains  a copy of a complete $(r+1)$-partite graph 
 $K_{r+1}(t)$, whose every vertex class has $t$ vertices. 
\end{theorem} 

This result is sharp in the following qualitative sense. 
Let $T_r(n)$ denote the complete $r$-partite graph of order $n$ 
whose vertex classes have the number of vertices as equal as possible. 
Clearly, $T_r(n)$ contains no complete subgraph of order $r+1$ 
and $e(T_r(n))=  (1-\frac{1}{r}) \frac{n^2}{2}+O(n)$. 
From this observation, we can get a lower bound on 
$\mathrm{ex}(n,K_{r+1}(t))$. 
Hence Theorem \ref{thm360} yields 
the following two corollaries. 

\begin{corollary}
For every integers $r\ge 1$ and $t\ge 1$, then 
\begin{equation*}
\mathrm{ex}(n,K_{r+1}(t)) =\left( 1-\frac{1}{r}+ o(1) \right) \frac{n^2}{2} ,
\end{equation*}
where $o(1)  \to 0$ whenever $n\to \infty$. 
\end{corollary}

 \begin{corollary}[Equivalent]
 For every integer $t\ge 1$, we have 
 \[  \lim\limits_{n\to \infty} 
 \frac{\mathrm{ex}(n,K_{r+1}(t)) }{
 {n \choose 2} } = 1- \frac{1}{r}.   \] 
 \end{corollary}

A proper vertex-coloring of a graph is a coloring of its vertex
set for which each color class forms an independent set.
The chromatic number of a graph $G$,  denoted by $\chi(G)$, 
is the minimum number of colors needed in a proper coloring of $G$. 
A graph is $r$-partite if it has chromatic number at most $r$. 
 In other words, the {\it vertex-chromatic number} $\chi (G)$ of  $G$ 
is the minimum integer $s\in \mathbb{N}^*$ such that there exists a coloring of 
$V(G)$ with $s$ colors and the adjacent vertices have different colors. 
For example, we can see that $\chi (K_{r+1}) =r+1$ and $\chi (K_{s,t})=2$. 
In fact, a graph $H$ has $\chi (H)=r+1$ 
if and only if $H$ is an $(r+1)$-partite graph.

This powerful 
result of Erd\H{o}s and Stone has many important 
applications in extremal graph theory. 
Consequently, the Erd\H{o}s--Stone theorem 
implies the following theorem since if $F$ is a graph with 
chromatic number $\chi (F)=r+1$, then $F$ is contained in 
$K_{r+1}(t)$ for an appropriate integer $t$.  
By applying the monotonicity of 
the Tur\'{a}n function, we get 
$\mathrm{ex}(n,F) \le \mathrm{ex}(n,K_{r+1}(t))$. 
This theorem was established by 
Erd\H{o}s and Simonovits in \cite{ES66}.

\begin{theorem}[Erd\H{o}s--Simonovits \cite{ES66}] \label{thm361}
For every graph $H$, we have 
\begin{equation*}
 \mathrm{ex}(n,H)=\left( 1-\frac{1}{\chi (H)-1}\right) \frac{n^2}{2} + o(n^2).  
\end{equation*}
\end{theorem}

Equivalently, we can write the above result as the limitation. 
\[ \boxed{ \lim_{n\to \infty} 
 \frac{\mathrm{ex}(n,H) }{  {n \choose 2}  } = 
\frac{1}{\chi (H) -1}.}\]

Although this result was formally established by Erd\H{o}s and 
Simonovits, we always call it the 
Erd\H{o}s--Stone--Simonovits theorem 
for the sake of their equivalence relation.  
We mention here that  Erd\H{o}s, Stone 
and Simonovits never wrote a paper 
together. First Erd\H{o}s and Stone solved it for $H$ a complete multipartite graph, and then Erd\H{o}s and
Simonovits  proved it for general $H$ in this way.

\begin{proof}
First of all, we denote $\chi (H)=r+1$.  
It is easy to see that the $n$-vertex 
complete $r$-partite 
Tur\'{a}n graph $T_r (n)$ contains no copy of 
$H$ otherwise $H$ is $r$-partite, 
so $H$ can be colored by $r$ colors, 
a contradiction. Thus, we can  get 
\[ \mathrm{ex}(n,H) \ge t_{r}(n)\ge \left(1-\frac{1}{r} \right) 
\frac{n^2}{2} - \frac{r}{8} =
 \left(1-\frac{1}{r} - o(1) \right)\frac{n^2}{2}  . \] 
 
For the upper bound,  
 since $\chi (H)=r+1$, 
 there exists an integer $t$ such that 
 $H$ is a subgraph of $K_{r+1}(t)$. 
By the Erd\H{o}s--Stone Theorem \ref{thm360}, we can get 
\[ \mathrm{ex}(n,H) \le \mathrm{ex}(n,K_{r+1}(t) ) 
\le t_r(n) + o(n^2).  \]
Hence, we have proved that for some fixed integer $v$, 
\[ t_r(n) \le 
\mathrm{ex}(n,H) \le \mathrm{ex}(n,K_{r+1}(t)) 
\le t_r(n) + o(n^2).  \]
This completes the proof. 
\end{proof}

The Erd\H{o}s--Stone--Simonovits theorem 
implies straightforward the following extremal number 
for graphs corresponding to 
the five regular polyhedrons.

\begin{itemize}

\item 
The regular tetrahedron 
forms the graph $K_4$, we have 
$\mathrm{ex}(n,K_4)=e(T_3(n))=\lfloor n^2/3\rfloor$. 

\item 
Let $Q_8$ be the cube graph on $8$ vertices. Then $\chi (Q_8)=2$ and 
$\mathrm{ex}(n,Q_8)=o(n^2)$. 

\item 
Let $O_6$ be the octahedron graph on $6$ vertices.  It is easy to see that 
$O_6$ is a complete $3$-partite graph with each part of size $2$, i.e., $O_6=K_{2,2,2}$. 
Then $\chi (O_6)=3$ and 
\[  \mathrm{ex}(n,O_6 )=\frac{n^2}{4}+o(n^2). \]  

\item 
Let $D_{20}$ be the dodecahedron graph on $20$ vertices. Then $\chi (D_{20})=3$ and 
\[  \mathrm{ex}(n,D_{20} )=\frac{n^2}{4} +o(n^2). \] 

\item 
Let $I_{12}$ be the icosahedron graph on $12$ vertices. Then $\chi (I_{12})=4$ and 
\[ \mathrm{ex}(n,I_{12} )=\frac{n^2}{3} +o(n^2). \] 
\end{itemize}

We remark here that 
although the Erd\H{o}s--Stone--Simonovits theorem 
provides good asymptotic estimates for the extremal numbers of non-bipartite graphs. 
 However, for bipartite graphs, where $\chi (F)=2$, it only gives the bound 
 $\mathrm{ex}(n,F)=o(n^2)$. 
 Although there have been numerous attempts of finding better bounds 
 on $\mathrm{ex}(n,F)$ for various bipartite graphs $F$, 
 we know very little in this case; see \cite{FS13} for a survey.

\medskip 
In 2009, Nikiforov \cite{Nikicpc2009} proved a spectral generalization 
of 
the Erd\H{o}s--Stone--Simonovits theorem, 
as we know that the Rayleigh formula implies 
$\lambda (G) \ge 2e(G)/n$.

\begin{theorem}[Nikiforov \cite{Nikicpc2009}] \label{thmniki177}
Let $n\ge r\ge 3$ be integers with $(c/r^r)^r\ln n\ge 1$. 
If $G$ is a graph on $n$ vertices with 
\[  \lambda (G)\ge \left( 1-\frac{1}{r-1}+c \right)n,\]   
then $G$ contains a copy of $K_{r}(s,\ldots ,s,t)$ with $s\ge  (c/r^r)^r\ln n $ 
and $t>n^{1-c^{r-1}}$. 
\end{theorem}

\noindent 
{\bf Remark. }  
The main idea of the proof of Nikiforov   
is based on two elegant results from \cite{BNJCTB07} 
and \cite{NikiLMS09}.  
Moreover, Nikiforov also gave an explicit bound on the size of 
each vertex part of the complete $r$-partite subgraph 
in terms of the adjacency spectral radius.  
Furthermore, Theorem \ref{thmniki177} is very precise and does not use the Regularity Lemma. 
Its edge version is a lot stronger than the Erdos--Stone and the Erdos--Simonovits theorems. Thus, it can be applied to obtain extremal results for concrete graphs, in particular, graphs whose order is not fixed and 
grows with $n$.

Analogous to the Tur\'{a}n function, 
we define the spectral Tur\'{a}n function as 
$\mathrm{ex}_{\lambda}(n,F)$, 
which is the maximum spectral radius in 
an $n$-vertex graph without subgraph isomorphic to $F$. 
It is important and meaningful to 
determine the function 
$\mathrm{ex}_{\lambda}(n,F)$ for various graphs $F$.  
From Theorem \ref{thmniki177}, one can easily get the following weak corollary, 
which gives the asymptotic spectral extremal number for 
general graphs. This corollary was also proved by 
Guiduli  in his PH.D. thesis \cite{Gui1996} in 1996.  

\begin{corollary} \label{coro198}
For every integers $r\ge 2$ and $t\ge 1$, we have 
\[ \mathrm{ex}_{\lambda}(n,K_{r}(t))= 
\left( 1-\frac{1}{r-1} \right) n + o(n).\]  
Consequently, 
if $F$ is a graph with chromatic number $\chi (F)$, 
then 
\[ \mathrm{ex}_{\lambda} (n,F) = 
\left( 1-\frac{1}{\chi (F)-1} \right) n + o(n).\]  
\end{corollary} 

This result can be rewritten as the following limitation. 
\[ \boxed{ \lim_{n\to \infty} 
\frac{\mathrm{ex}_{\lambda}(n,F) }{ n} 
= 1-\frac{1}{\chi (F) -1}.} \]

We mentioned here that Guiduli \cite{Gui1996}
 proved Corollary \ref{coro198} by using 
 an analogous result 
proved by  Erd\H{o}s, Frankl and R\"{o}dl \cite{EFR1986}, 
which states that for any $r,t\ge 1$ and $\varepsilon >0$, 
there exists an integer $n_0$ such that 
if $G$ is a $K_{r}(t)$-free graph on $n\ge n_0$ vertices, 
then we can remove at most $\varepsilon n^2$ 
edges from $G$ to make it being $K_r$-free. 
This result is a direct consequence of 
the celebrated {\sc Szemer\'{e}di Regularity Lemma.} 

Next, we shall provide an alternative 
proof of Corollary \ref{coro198} 
by applying the {\sc Graph Removal Lemma}, 
which states that 
for every graph $H$ on $h$ vertices and every $\varepsilon >0$, 
there exists $\delta =\delta_H(\varepsilon) >0$ such that 
any graph on $n$ vertices with at most $\delta n^h$ 
copies of $H$ can be made $H$-free 
by removing at most $\varepsilon n^2$ edges; 
see, e.g., \cite{Fox2011,CF2013}. 

\begin{proof}[{\bf Alternative Proof}] 
Let $F$ be a fixed graph with $\chi (F)=r$ for some 
$r\ge 2$. We choose $t$ as an integer such that $F \subseteq K_r(t)$. 
For any $\varepsilon >0$,   
if $G$ is an $n$-vertex graph 
with $\lambda (G) \ge (1- \frac{1}{r-1} + \varepsilon )n$, 
the Wilf theorem implies that 
$G$ contains at least one copy of $K_{r}$. 
Next, we  shall show that the number of copies 
of $K_{r}$ is at least $\delta_{K_{r}}
({\varepsilon}/{2})\cdot  n^{r}$ where 
 $\delta_H(\varepsilon )$ is the function determined 
 in the Graph Removal Lemma. 
This is the spectral version of the supersaturation 
lemma for $K_{r}$. 
Indeed, suppose on the contrary that 
$G$ has less than $\delta_{K_{r}}
({\varepsilon}/{2})\cdot n^{r}$ copies 
of $K_{r}$, then 
the Graph Removal Lemma yields that 
we can remove at most $\frac{\varepsilon}{2}n^2$ edges from $G$ 
such that the remaining subgraph $G^*$ is $K_{r}$-free. 
The Wilf theorem gives that 
$\lambda (G^*) \le (1-\frac{1}{r-1})n$. Hence, 
the Rayleigh formula yields 
\[  \lambda (G) \le 
\lambda (G^*) + \lambda (G\setminus G^*)
<  \lambda (G^*) +  \sqrt{2e(G\setminus G^*)} 
\le \left( 1- \frac{1}{r-1} \right) n + \varepsilon n,  \]
a contradiction. Hence $G$ has at least 
$\delta_{K_{r}}({\varepsilon}/{2})\cdot n^{r}$ copies 
of $K_{r}$. 

Now, we define an $r$-uniform hypergraph 
$\mathcal{G}^{(r)}$ on the vertex set $V(G)$ 
where an $r$-element subset $E\subseteq V(G)$ 
is defined to be a hyperedge if and only if 
it induces  a copy of $K_{r}$ in the $2$-graph $G$. 
Under this definition, we know that 
$\mathcal{G}^{(r)}$ has at least 
$\delta_{K_{r}}({\varepsilon}/{2})\cdot n^{r}$ edges. 
We denote by $K^{(r)}_{t,t,\ldots ,t}$ the 
complete $r$-partite $r$-uniform hypergraph 
with each vertex part of size $t$.  
A well-known result of Erd\H{o}s \cite{Erd1964} states that 
\[ \mathrm{ex}(n,K^{(r)}_{t,t,\ldots ,t}) =O(n^{r- ({1}/{t})^{r-1}}).\] 
Therefore, for sufficiently large $n$, the hypergraph 
$\mathcal{G}^{(r)}$ contains a copy of 
$K^{(r)}_{t,t,\ldots ,t}$. Note that every copy of 
  $K^{(r)}_{t,t,\ldots ,t}$ in $\mathcal{G}^{(r)}$ 
  corresponds to a copy of $K_r(t)$ in $G$. 
  Hence $G$ contains a copy of $K_r(t)$ for 
  any $\varepsilon >0$ and large enough $n$. 
\end{proof}

\begin{conjecture} \label{conjse}
If $G$ is an $F$-free graph with $m$ edges, then 
\[ \lambda(G)\le 
\sqrt{\left(1-\frac{1}{\chi (F)-1} +o(1) \right)2m}. \]
\end{conjecture} 

After our manuscript announced, 
Nikiforov (private communication) told us that he can prove Conjecture \ref{conjse} for $3$-chromatic graphs. 
A natural question is to extend the above-mentioned
 results on the adjacency spectral radius 
 to that of the signless Laplacian spectral radius. 
We define $\mathrm{ex}_{q}(n,F)$ 
to be the largest eigenvalue of the signless Laplacian   matrix 
in an $n$-vertex  graph that contains no copy of $F$. That is, 
\[ \mathrm{ex}_{q}(n,F):=\max \bigl\{ q(G): |G|=n~\text{and}~F\nsubseteq G \bigr\}. \]
Note that $Q(G)=D(G)-A(G) + 2A(G)$ and $D(G)-A(G)$ 
is positive semidefinite. It is known by the Weyl theorem 
for monotonicity of eigenvalues that  
$q(G)\ge 2\lambda (G)$. 
Comparing to Corollary \ref{coro198}, 
one may ask the following question:

\medskip 
\noindent 
{\bf Question.} 
{\it Let $F$ be a fixed graph. 
Is it true that  
$ \mathrm{ex}_{q} (n,F) = 
\bigl( 1-\frac{1}{\chi (F)-1} \bigr) 2n + o(n)$.}

\medskip 
The answer of this question is negative. 
Let  $k\ge 2$ be a positive integer and 
 $S_{n,k}$ be the graph consisting of a clique on $k$ vertices and an independent set on $n-k$ vertices in which each vertex of the clique is adjacent to each vertex of the independent set. 
 In the language of join of graphs, we have 
$S_{n,k}=K_k \vee I_{n-k}$. 
We can observe that $S_{n,k}$ does not contain $C_{2k+2}$ as a subgraph. 
Furthermore, 
let $S_{n,k}^+$ be the graph obtained from $S_{n,k}$ 
by adding an edge to the independent set $I_{n-k}$. 
Namely, we have  
$S_{n,k}^+=K_k \vee I_{n-k}^+$.  
Clearly, we can see that 
$S_{n,k}^+$  is still $C_{2k+2}$-free and 
\[  q(S_{n,k}^+) >  q(S_{n,k}) = \frac{n+2k-2+ \sqrt{(n+2k-2)^2 - 8k^2 +8k
}}{2}.   \]
Hence either the graph $S_{n,k}$ or $S_{n,k}^+$ can yield $\mathrm{ex}_q(n,C_{2k+2}) \ge n+ o(n)$. 
However we have $\chi (C_{2k+2})=2$ and $(1-\frac{1}{r})2n + o(n)=o(n)$. 
In 2015, Nikiforov and Yuan \cite{NY2015} proved that 
$\mathrm{ex}_q(n,C_{2k+2}) = q(S_{n,k}^+)$ 
for every $n\ge 400 k^2$. 
However $\chi (C_{2k+2})=2$ and $q(S_{n,k}^+) \neq o(n)$.

\subsection{Spectral problem for bipartite graphs} 

 The history of studying the extremal number 
 for bipartite graphs began in 1954 with 
 the K\H{o}vari--S\'{o}s--Tur\'{a}n theorem \cite{KST54}, which states 
 that if $K_{s,t}$ is the complete bipartite graph with vertex classes of size $s\ge t$, 
 then $\mathrm{ex}(n,K_{s,t})=O(n^{2-1/t})$; see \cite{Furedi96,Furedi96b} 
 for more details.  
In particular, we refer the interested reader to 
  the comprehensive survey 
  by F\"{u}redi and Simonovits \cite{FS13} 
  for the history of bipartite extremal graph problems. 

\begin{theorem}[K\H{o}vari--S\'{o}s--Tur\'{a}n, 1954]  \label{thmkst}
For all $t\ge s\ge 2$, If $G$ is an $n$-vertex graph and contains no copy of $K_{s,t}$, then 
\[  e(G) \le 
\frac{1}{2}(t-1)^{\frac{1}{s}}n^{2-\frac{1}{s}} +\frac{1}{2}(s-1)n . \]
In other words, we have 
\[ \mathrm{ex}(n,K_{s,t}) \le 
\frac{1}{2}(t-1)^{\frac{1}{s}}n^{2-\frac{1}{s}} +\frac{1}{2}(s-1)n 
{ \,\, \approx  \frac{1}{2}(t-1)^{\frac{1}{s}}n^{2-\frac{1}{s}}.} \]
\end{theorem} 

For  convenience of readers, we here include a standard 
proof, which is a well application of the double counting technique. 

\begin{proof}
Let $G$ be a $K_{s, t}$-free $n$-vertex graph with $m$ edges.
We denote the number of copies of $K_{s, 1}$ in $G$ as $\# K_{s, 1}$. The proof establishes an upper bound and a lower bound on $\#K_{s, 1}$, and then gets a bound on $m$ by combining the upper bound and the lower bound. 
Since $K_{s,1}$ 
is a tree, we can call the side with $s$ vertices the leaf vertices, 
and the side with $1$ vertices the root vertex.

On the one hand, we can count $\#K_{s, 1}$ by enumerating the leaf vertices. 
For any subset $S$ of $s$ vertices, the number of $K_{s, 1}$ 
 is exactly the number of common neighbors of these $s$ vertices. 
 We write $d(S)$ for 
 the number of common neighbors of vertices of $S$, 
 i.e., $d(S)=|\cap_{v\in S}N(v)|$. 
 Since $G$ is $K_{s, t}$-free, 
 we get $d(S)\le t-1$. Thus, we establish that 
 \[ \#K_{s, 1}=\sum_{S\in {V \choose s}} d(S) \le \binom{n}{s}(t-1).\] 

On the other hand, for each vertex $v \in V(G)$, the number of copies of $K_{s, 1}$ where $v$ is the root vertex is exactly $\binom{d(v)}{s}$. 
Therefore,
\[   \#K_{s, 1} = \sum_{v \in V(G)} \binom{d(v)}{s}
    \ge n \binom{\frac{1}{n} \sum_{v \in V(G)} d(v)}{s}
    = n\binom{2m/n}{s}, \]
where the inequality  uses the convexity of $x \mapsto \binom{x}{s}$. Here we regard $\binom{x}{s}$ as a degree $s$ polynomial in $x$, so it makes sense for $x$ to be non-integers. 

Combining the upper bound and lower bound of $\#K_{s, 1}$, we obtain that $n\binom{2m/n}{s} \le \binom{n}{s} (t-1)$. For constant $s$, we can use $\binom{x}{s} = (1+o(1)) \frac{x^s}{s!}$ to get $n\left(\frac{2m}{n}\right)^s \le (1+o(1)) n^s (t-1)$. The above inequality simplifies to 
$m \le \left(\frac{1}{2} + o(1)\right) (t-1)^{1/s} n^{2-{1}/{s}}$. 
\end{proof}

\begin{definition}[Zarankiewicz]
Let $z(m,n,s,t)$ be the maximum number of edges in a bipartite graph $G$ with vertex parts 
$M$ of size $m$ and $N$ of size $n$ such that 
$G$ does not contain a copy of $K_{s,t}$ with vertex part of size $s$ in $M$ and 
vertex part of size $t$ in $N$. 
\end{definition}

In the language of matrix, 
the value $z(m,n,s,t)$ can  be defined as the maximum number of 
ones in a matrix consisting only of zeros and ones of size $m\times n$ 
that does not contain an all ones submatrix of size $s\times t$. 
 The problem of finding this number is known as the Zarankiewicz problem.

Clearly, we can see from the definition that 
$z(n,n,s,t) \le 
\mathrm{ex}(2n,K_{s,t})$. 
We next show that $2\mathrm{ex}(n,K_{s,t}) \le z(n,n,s,t)$. 
To prove this inequality, we just need to 
construct a bipartite graph $G$ that has  
$2\mathrm{ex}(n,K_{s,t})$ edges and contains no copy of $K_{s,t}$. 
Let $G$ be an extremal graph with $\mathrm{ex}(n,K_{s,t})$ edges 
and contains no copy of $K_{s,t}$. We define a bipartite graph $H$ on 
vertex sets $V_1$ and $V_2$ where $V_1=V_2=V(G)$ 
with edge set $E(H)=\{\{u,v\} : u\in V_1, v\in V_2, \{u,v\}\in E(G)\}$. 
We can verify that $H$ does not contain any copy of $K_{s,t}$ 
and $H$ has $2e(G)$ edges. Thus, we get $2\mathrm{ex}(n,K_{s,t}) \le z(n,n,s,t)$. 
We now state the above observation as the following proposition. 

\begin{proposition}
$2\mathrm{ex}(n,K_{s,t}) \le z(n,n,s,t) \le 
\mathrm{ex}(2n,K_{s,t})$.
\end{proposition}

From this proposition, 
it turns the problems of finding an upper bound 
for the Tur\'{a}n number into an upper bound of the Zarankiewicz problems. 
Using the counting method in the proof of the 
K\H{o}v\'{a}ri--S\'{o}s--Tur\'{a}n theorem, we can similarly prove that 
\[ z(m,n,s,t)\le (s-1)^{\frac{1}{t}}\cdot (n-t+1)m^{1-\frac{1}{t}} +(t-1)m. \]

In 1996, F\"{u}redi \cite{Furedi96} improved the coefficient of this  bound 
to $(s-t+1)^{1/t}$. 

\begin{theorem}[F\"{u}redi \cite{Furedi96}] \label{thmfuredi} 
For all $s\ge t\ge 1$, we have  
\begin{equation}\label{Fu96}
 z(m,n,s,t) \le { (s-t+1)}^{\frac{1}{t}}nm^{1-\frac{1}{t}}+tm^{2-\frac{2}{t}}+tn. 
\end{equation}
\end{theorem} 

In 2010, Nikiforov \cite{Niki2010} proved by an elegant induction that 

\begin{theorem}[Nikiforov \cite{Niki2010}] \label{thmniki2010}
If $s,t\ge 2$, then for every $k=0,1,\ldots ,s\!-\!2$, 
\[  z(m,n,s,t)\le (s-k-1)^{\frac{1}{t}}nm^{1-\frac{1}{t}} +(t-1)m^{1+\frac{k}{t}}+kn. \]
\end{theorem}

For $k=0$, the upper bound of K\H{o}vari, S\'{o}s and Tur\'{a}n is obtained, 
and, if $s\ge t$, the author also improved  F\"{u}redi's bound,  
since the case $k=t-2$ leads to 
\[  z(m,n,s,t)\le (s-t+1)^{{1}/{t}}nm^{1-{1}/{t}} +(t-1)m^{2-{2}/{t}}+(t-2)n. \]

We remark that finding the lower bound of 
$\mathrm{ex}(n,K_{s,t})$ is extremely difficult. 
By using the Probabilistic Method, 
one can get a lower bound of $\mathrm{ex}(n,K_{s,t})$.

\begin{theorem} \label{thm211}
For all $t\ge s\ge 2$, we have   
\[ \mathrm{ex}(n,K_{s,t}) \ge \frac{1}{16}n^{2-\frac{s+t-2}{st-1}}. \]
\end{theorem}

\begin{corollary}
For all positive integer $s$, we have 
\[ c_1(s)\cdot n^{2-\frac{2}{s+1}} \le \mathrm{ex}(n,K_{s,s}) \le c_2 (s)\cdot n^{2-\frac{1}{s}}, \]
where $c_1(s)$ and $c_2(s)$ are constants depending only on $s$. 
\end{corollary}

\begin{proof}[Proof of Theorem \ref{thm211}]
We consider the random graph $G(n,p)$, 
a graph on $n$ vertices where each each is chosen 
independently from $K_n$  with probability $p$. 
Let $\# K_{s,t}$ denote the number of 
copies of $K_{s,t}$ in $G(n,p)$. 
Then 
the expected number of edges 
and copies of $K_{s,t}$ are given by 
$ \mathbb{E}[e(G)] = p {n \choose 2} $ and  
$\mathbb{E}[\# K_{s,t}] = 
{n \choose s }{n-s \choose t} p^{st}$ respectively.  
By the linearity 
of expectation, we have  
\[ \mathbb{E}[e(G)- \# K_{s,t}]  = 
\mathbb{E}[e(G)] - \mathbb{E}[\# K_{s,t}] 
\ge \frac{1}{4}pn^2 - n^{s+t}p^{st}.
 \] 
 We now want to pick $p$ so that 
 the right hand side in above is as large as possible. 
 By setting $p=\frac{1}{2}n^{-\frac{s+t-2}{st-1}}$, 
 we can get 
 $\mathbb{E}[e(G)-\# K_{s,t}] \ge 
 \frac{1}{8}pn^2= \frac{1}{16}n^{2-\frac{s+t-2}{st-1}}$. 
 Thus there exists a graph 
 $G$ in which $e(G)-\# K_{s,t} \ge 
 \frac{1}{16}n^{2-\frac{s+t-2}{st-1}}$, 
 and by removing one edge 
 from each copy of $K_{s,t}$, 
 we can get a $K_{s,t}$-free graph with at least 
 $\frac{1}{16}n^{2-\frac{s+t-2}{st-1}}$ edges. 
\end{proof}

\begin{conjecture}
For $t\ge s\ge 2$, we have 
$\mathrm{ex}(n,K_{s,t})=\Theta (n^{2-{1}/{s}})$. 
\end{conjecture}

We here list the recent progress of this conjecture 
in the literature. 
In 1996, Koll\'{a}r, R\'{o}nyai and Szab\'{o} \cite{KRS1996} confirmed this conjecture 
for the case $t> s!$.  In 1999, Alon, R\'{o}nyai and Szab\'{o} \cite{ARS1999}
 slightly improved to  the case $t> (s-1)!$. 
In 2013, Blagojevi\'{c}, Bukh and Karasev \cite{BBK2013} 
 and later Bukh \cite{Buk2015} 
 used the random algebraic method to give different constructions 
 which yield the same lower bound $\mathrm{ex}(n, K_{s,t}) =\Omega (n^{2-1/s})$, provided that $t$ is sufficiently large. 
 In 2021, Bukh made a further improvement \cite{Buk2021} 
 by applying the random algebraic construction again. 

Note that ${2e(G)}/{n} \le \lambda (G)$. 
Hence finding the upper bounds on the adjacency spectral radius 
extends the general extremal bounds on number of edges. 
In 2007, Babai and Guiduli \cite{BG2007} published 
an asymptotic bound of the K\H{o}v\'{a}ri--S\'{o}s--Tur\'{a}n 
upper bound on the spectral radius (Their work dates back to 
1996 on the PH.D. thesis of Guiduli \cite{Gui1996}). 

\begin{theorem}[Babai--Guiduli \cite{Gui1996, BG2007}]
Let $\lambda (G)$ denote the spectral radius of $G$. \\ 
(1) Let $s\ge 2$ be an integer. If $G$ is an $n$-vertex $K_{s,2}$-free graph, then 
\begin{equation*}
 \lambda (G)\le (s-1)^{\frac 1 2}n^{\frac 1 2} +O(n^{\frac 1 4}). 
\end{equation*}
(2) 
Let $s\ge t\ge 2$ be integers. If $G$ is an $n$-vertex $K_{s,t}$-free graph, then
\begin{equation*}
\lambda (G)\le \Bigl( (s-1)^{\frac{1}{ t}}+o(1) \Bigr) n^{1-\frac{1}{t} }. 
\end{equation*}
\end{theorem}

Nikiforov \cite{Niki2007b} proved the following spectral extremal problem 
for $K_{2,\ell}$ and the book graph $B_{k+1}=K_2 \vee I_{k+1}$, 
a graph obtained by joining each vertex of $K_2$ to each vertex 
of an independent set  of size $k+1$.

\begin{theorem}[Nikiforov \cite{Niki2007b}]
Let $\ell \ge k\ge 0$ be integers. If $G$ is a  graph on $n$ vertices 
and contains no copy of $K_{2,\ell +1}$ and 
$K_2 \vee {I_{k+1}}$, 
then 
\begin{equation*}
\lambda (G) \le \frac{1}{2}\Bigl( k-\ell +1 +\sqrt{(k-\ell +1)^2+4\ell (n-1)}\Bigr). 
\end{equation*}
\end{theorem}

In particular, this theorem implies the following result.

\begin{corollary}
Let $s\ge 2$. If $G$ is an $n$-vertex $K_{s,2}$-free graph, then 
\begin{equation*}
 \lambda (G)\le \frac{1}{2}+\sqrt{(s-1)(n-1)+{1}/{4}}. 
\end{equation*}
\end{corollary}

As a consequence of Theorem \ref{thmniki2010}, 
Nikiforov \cite{Niki2010} 
improved the result of Babai and Guiduli 
 concerning the largeness of the spectral radius of a graph of certain order that does not contain a complete bipartite subgraph $K_{s,t}$.

\begin{theorem}[Nikiforov \cite{Niki2010}] \label{thm245}
Let $s\ge t\ge 3$ be integers. If $G$ is an $n$-vertex $K_{s,t}$-free graph 
with spectral radius $\lambda (G)$. Then 
\begin{equation*}
 \lambda (G)\le { (s-t+1)^{\frac{1}{t}} } n^{1-\frac{1}{t}} 
 +(t-1)n^{1-\frac{2}{t}}+t-2. 
\end{equation*}
\end{theorem}

A natural question is that 
how to characterize the extremal graphs attaining the equality in Theorems  \ref{thmniki2010} and 
\ref{thm245}.  For the edge version, 
it is extremely difficult to find the graph attaining the upper bound 
even for the special graph $K_{2,2}$. 
Surprisingly, the extremal graphs of the spectral  problem 
for $K_{2,2}$ was determined by Nikiforov \cite{Niki2007b}, 
and Zhai  and Wang \cite{ZW2012}. 
We infer the readers to Subsection \ref{subsec2.5} for more details.

\subsection{Spectral problem for odd cycles} 

Let $C_{2k+1}$ denote the cycle on $2k+1$ vertices. 
Note that $C_{2k+1}$  has chromatic number $\chi (C_{2k+1})=3$. 
The Erd\H{o}s--Stone Theorem \ref{thm360} 
or Erd\H{o}s--Simonovits Theorem \ref{thm361} 
implies that $\mathrm{ex}(n,C_{2k+1})= {n^2}/{4} + o(n^2)$. 
In 1971, Bondy \cite{Bon1971} 
and Woodall \cite{Woo1972} proved independently 
a more stronger result, 
which implies that the extremal number for odd-length cycle 
is exactly equal to $n^2/4$; 
see \cite[p. 150]{Bollobas78} for more details.

\begin{theorem}[Bondy \cite{Bon1971}; Woodall \cite{Woo1972}] \label{coro2103}
If $G$ is a graph on $n$ vertices with 
$e(G)> \lfloor {n^2}/{4} \rfloor$, 
then $G$ contains 
 a cycle of length $\ell$ for every $\ell 
=3,4,\ldots,\lfloor {(n+3)}/{2}\rfloor$. 
\end{theorem}

From this theorem, we can easily get the following corollary.

\begin{corollary} \label{corooddcy}
If  $k\ge 1$ and $n\ge 4k$, then  
\[  \mathrm{ex}(n,C_{2k+1})=\left\lfloor \frac{n^2}{4}\right\rfloor. \]
\end{corollary}

This result can also be proved by the well-known stability method 
for $n$ sufficiently large; see \cite[Theorem 5.3]{Kee2011} for more details.  
In 2015, F\"{u}redi and Gunderson \cite{FG2015cpc} determined 
the exact value of $\mathrm{ex}(n,C_{2k+1})$ for all 
$n\ge 1$ and $k\ge 2$. The case $k=1$ was early characterized by Mantel 
who showed $\mathrm{ex}(n,C_3)=\lfloor n^2/4 \rfloor$ for all $n$.

A wheel $W_n$ is a graph on $n$ vertices obtained from $C_{n-1}$ by adding 
one vertex $w$ and making $w$ adjacent to all vertices of the $C_{n-1}$. 
In the language of join graph, 
we have 
\[   W_n=K_1 \vee C_{n-1}. \] 
It is easy to see that the vertex-chromatic number $\chi (W_{2k})=4$ 
and $W_{2k}$ is color-critical, that is, 
there is an edge $e\in W_{2k}$ such that $\chi (W_{2k}\setminus e) =3$. 
The Erd\H{o}s--Stone--Simonovits Theorem \ref{thm361}
implies that $\mathrm{ex}(n,W_{2k})= e(T_3(n)) + o(n^2)$. 
In 2013, Dzido proved the  exact Tur\'{a}n number by 
using Corollary \ref{corooddcy}. 

\begin{theorem}[Dzido \cite{Dzi2013}] \label{thmdzi}
For all $k\ge 3$ and $n\ge 6k-10$, then 
\[ \mathrm{ex}(n,W_{2k})=e(T_3(n))=\left\lfloor \frac{n^2}{3}\right\rfloor. \]
\end{theorem}

For odd wheels, we know that 
$\chi (W_{2k+1}) =3$ and  the Erd\H{o}s--Stone--Simonovits Theorem \ref{thm361} 
implies the asymptotic value $\mathrm{ex}(n,W_{2k+1})= n^2/4 + o(n^2)$. 
We remark here that the exact extremal number of odd wheels 
was recently determined by Yuan \cite{Yuan2021} for  sufficiently large $n$ by applying the stability method. 
The problem for odd wheels is more complicated than 
that of even wheels since the odd wheels are not color-critical. 
In 2022, the spectral extremal problem was proved by  Cioab\u{a}, Desai  and Tait \cite{CDT21}.

In 2008, Nikiforov \cite{Niki2008} 
proved an analogue of the Bondy--Woodall Theorem \ref{coro2103}  
in terms of the spectral radius 
for the existence of cycles with consecutive lengths.

\begin{theorem}[Nikiforov \cite{Niki2008}] \label{thmniki2008}
Let $G$ be a graph of sufficiently large order $n$ with 
$\lambda (G)>\sqrt{\left\lfloor {n^2}/{4}\right\rfloor}$. Then $G$ contains a cycle of length $t$ for 
$t\le {n}/{320}$. 
\end{theorem} 

The result is sharp because the complete bipartite graph $T_2(n)$ 
with parts of size $\lfloor n/2\rfloor$ and $\lceil n/2 \rceil $ 
has no odd cycles, and its largest eigenvalue is $\sqrt{\left\lfloor {n^2}/{4}\right\rfloor}$. 
Moreover, 
it is clear that the constant $1/320$ in Theorem \ref{thmniki2008} 
can be increased even with careful calculations in the 
Nikiforov methods. 
Thus, the following question may arise: 

\begin{conjecture}
What is the maximum $C$ such that for all positive $\varepsilon <C$ 
and sufficiently large $n$, every graph $G$ of order $n$ with 
$\lambda (G)>\sqrt{\left\lfloor {n^2}/{4}\right\rfloor}$ 
contains a cycle of length $t$ for every $t\le (C-\varepsilon )n$. 
\end{conjecture}

It is known from Theorem \ref{coro2103} that if $G$ is a graph 
of order $n$ with $e(G)> \left\lfloor {n^2}/{4}\right\rfloor$, 
then $G$ contains a cycle of length $t$ for every $3\le t \le \lceil n/2\rceil$. 
Therefore, one can conjecture that $C = 1/2$. Unfortunately, 
this is not true by  taking the join of a complete graph of order $k$ 
and an empty graph of order $n-k$, we can obtain a graph 
\[   S_{n,k}=K_k \vee I_{n-k}.\]   
Setting $k=\lceil (3-\sqrt{5})n/4 \rceil$, 
an easy calculation  gives 
\[  \lambda (S_{n,k}) 
= \frac{k-1}{2} + \sqrt{kn -\frac{1}{4}(3k^2 +2k -1)} 
>\frac{n}{2}\ge \left\lfloor \frac{n^2}{4}\right\rfloor. \] 
However, 
$S_{n,k}$ does not contain cycles longer than $2k \approx
0.382 n < n/2$. 
Recently, Ning and
Peng \cite{NP2020} slightly refined this as $C=1/160$. 
Zhai and 
Lin  \cite{ZL2021} improved these results to $C\ge 1/7$, 
and Li and Ning \cite{LN2021} proved that $C\ge 1/4$ 
by applying a method different from previous ones. 

As a consequence of Theorem \ref{thmniki2008}, 
we can get the following corollary. 

\begin{corollary}
Let $k\ge 1$ and 
$G$ be a graph of sufficiently large order $n\ge 640 k + 320$. 
If $G$ does not contain a copy of $C_{2k+1}$, 
then $\lambda (G) \le \sqrt{\left\lfloor {n^2}/{4}\right\rfloor}$. 
\end{corollary}

For the extremal problems on the signless Laplacian spectral radius, 
the following results on odd cycles are proved in  \cite{FNP2013}
 and \cite{Yuan2014}. 
 
\begin{theorem}[Freitas--Nikiforov--Patuzzi \cite{FNP2013}]
If $G$ is a graph on $n\ge 6$ vertices with no copy of 
$C_5$, then either 
$ q(G)<  q(K_2\vee I_{n-2})$ or $G=K_2 \vee I_{n-2}$. 
\end{theorem}

In 2014, Yuan proved the general result 
for the signless Laplacian radius. 

\begin{theorem}[Yuan \cite{Yuan2014}]
Let $k\ge 3, n\ge 100k^2$ and $G$ be an $n$-vertex graph. 
If $G$ has no copy of $C_{2k+1}$, then either 
$q(G) < q(K_k \vee I_{n-k})$ 
or  $G=K_k \vee I_{n-k}$. 
\end{theorem}

Recall that  $S_{n,k}= K_k \vee I_{n-k}$, 
it is not hard to show that 
\[  q(S_{n,k}) = 
\frac{n+2k-2 + 
\sqrt{(n+2k-2)^2- 8k(k-1)}}{2}.  \]

Note that the extremal graphs for spectral problem 
 and signless Laplacian spectral problem are different 
 when we forbid the odd-length cycles, 
 the former is the balanced complete bipartite graph $T_2(n)$ 
 and the latter is the split graph $S_{n,k}$.

\subsection{Spectral problem for even cycles} 

\label{subsec2.5}

In this section, we shall review extremal problems involving 
the cycles of even length. 
Note that even cycles are clearly bipartite, 
the K\H{o}vari--S\'{o}s--Tur\'{a}n Theorem \ref{thmkst} 
implies that 
$\mathrm{ex}(n,C_{2k})\le \mathrm{ex}(n,K_{k,k})=O(n^{2-1/k})$ 
for every integer $k\ge 2$. 

There have abundant references of the study of $\mathrm{ex}(n,C_4)$ in the literature; see, e.g., 
\cite{ERS1966,Bro1966, Furedi96c}. 
By double counting method, 
Reiman showed an upper bound that 
\begin{equation} \label{eqeq}
 \mathrm{ex}(n,C_4) \le \frac{n}{4}(1+\sqrt{4n-3}).
 \end{equation}
Furthermore, the equality holds if and only if 
every pair of vertices has exactly one 
common neighbor. However, 
The well-known  Friendship Theorem (see, e.g., 
\cite{ERS1966} or \cite[Chapter 43]{AZ2014}) states that  
if $G$ is an $n$-vertex graph in which every pair of vertices has exactly one 
common neighbor, then $n$ is odd and there is a vertex which is adjacent to
all other vertices in $G$, i.e., 
$G=K_1\vee \frac{n-1}{2}K_2$.  
This implies that $e(G)={3(n-1)}/{2}$. 
Hence, the equality in (\ref{eqeq})  never holds.  
Such  graph $G$ is called the friendship graph 
or the windmill graph.

Using orthogonal polarity graphs constructed from 
finite projective planes, Erd\H{o}s, R\'{e}nyi and S\'{o}s \cite{ERS1966} 
and Brown \cite{Bro1966} proved a lower bound, which states that 
\[ \mathrm{ex}(n, C_4)\ge \frac{1}{2}n^{\frac 3 2} + o(n^{\frac 3 2}). \]  
This  result together with (\ref{eqeq})  implies  the following 
asymptotic Tur\'{a}n number. 

\begin{theorem}[Erd\H{o}s--R\'{e}nyi--S\'{o}s, 1966]
$ \mathrm{ex}(n,C_4) = \frac{1}{2}n^{3/2}+o(n^{3/2})$. 
\end{theorem}

We remark here that there are many constructions of graphs attaining this lower bound 
asymptotically; see, e.g., \cite{FS13}. 
We next include a famous construction by using 
the projective plane over the finite field, 
which is now known as 
the Erd\H{o}s--R\'{e}nyi projective polarity graph. 

\begin{proof}
Let $q$ be a prime power. 
Then the field of order $q$ exists and denoted by $\mathbb{F}_q$. 
We consider the three-dimensional vector space $\mathbb{F}_q^3$ 
over the field $\mathbb{F}_q$. 
Since space $\mathbb{F}_q^3$ contains $q^3-1$ non-zero vectors, 
$q^2+q+1$ one-dimensional subspaces and 
$q^2+q+1$ two-dimensional subspaces. 
Moreover,  each one-dimensional subspace contains $q-1$ nonzero vectors,  
every two distinct one-dimensional subspaces  intersect only in  $(0,0,0)$. 
Furthermore,  each one-dimensional subspace is contained in
 $q+1$ two-dimensional subspaces, and  
each two-dimensional subspace contains $q+1$ one-dimensional subspaces. 
For every nonzero vector 
$\vec{a}= (a_1,a_2,a_3) \in \mathbb{F}_q^3$, 
let $A$ denote the one-dimensional 
subspace 
spanning by $\vec{a}$, that is, 
$ A :=\{k(a_1,a_2,a_3):k\in \mathbb{F}_q\}$. 

{\bf Step 1.}~We now construct the desired graph $G$ whose  
vertices are all one-dimensional subspaces of  $\mathbb{F}_q^3$, that is, 
$ V(G)= \left\{ A=\mathrm{Span}\{\vec{a}\} : \vec{a}\in \mathbb{F}_q^3
 \right\} $. 
 Therefore, we get 
$ |V(G)|=\frac{q^3-1}{q-1}=q^2+q+1$.  
Let two  vertices $A$ and 
$B $ be adjacent if and only if
\[ a_1b_1+a_2b_2+a_3b_3=0 ~(\text{over}~\mathbb{F}_q). \]
This definition is clearly well-defined, in other words, 
it does not depend on the choices of representative elements of the one-dimensional subspaces.

Next, we are going to count the number of edges. 
For each vertex $A=\mathrm{Span}\{\vec{a} \}$,  
since $a_1x_1+a_2x_2+a_3x_3=0$ has $q^2-1$ nonzero solutions in $\mathbb{F}_q^3$, 
which  forms exactly $q+1$ one-dimensional subspaces (vertices). 
Note that the graph $G$ may have loops (the case $a_1^2+a_2^2+a_3^2=0$ 
is possible), 
we remove all loops of $G$ to get a simple subgraph $\widetilde{G}$ on the same vertex set. 
Thus, 
\[ d_{\widetilde{G}}(A)=\left\{ \begin{array}{ll}
q,& \text{if}~a_1^2+a_2^2+a_3^2=0,\\
q+1,&\text{otherwise}. 
\end{array}  \right. \] 
Anyway, we have $d_{\widetilde{G}}(A)\ge q$ for every  vertex $A$. Therefore, we get 
\[ e(\widetilde{G}) =\frac{1}{2}\sum_{v\in V} d(v)\ge \frac{1}{2}(q^2+q+1)q. \]

{\bf Step 2.}~
{\it We are going to prove that $\widetilde{G}$ is $K_{2,2}$-free. }

\medskip 

Let $V$ 
and $W$ be two distinct vertices of $G$, 
if $U=\mathrm{Span}\{\vec{u}\}$ is a common neighbor of 
$V$ 
and $W$, then $u_1,u_2,u_3$ is a solution of the linear equations 
\[ \begin{cases}
v_1u_1 + v_2u_2 + v_3u_3 =0, \\
w_1u_1 + w_2 u_2 + w_3 u_3=0.
\end{cases} \]
Since $V$ 
and $W$ are two distinct one-dimensional subspaces, 
then the vectors $(v_1,v_2,v_3)$ and $(w_1,w_2,w_3)$ are linearly independent, 
which means that the  space of solutions has dimension one. 
 Hence the number of common neighbors of $V$ and $W$ in $G$ is 
exactly one. 
Note that the unique common neighbor of $V$ and $W$ in $G$  
may be $V$ or $W$ itself, 
and we have removed all loops from $G$ to obtain $\widetilde{G}$
so  
$V$ and $W$  is at most one 
common neighbor in the loopless graph $\widetilde{G}$. 
 For example, when $q=2$, 
the vertices $\mathrm{Span}\{(0,0,1)\}$ and $
\mathrm{Span}\{(1,1,0)\}$  
have the common neighbor $\mathrm{Span}\{(1,1,0)\}$ in $G$. 
However, in the loopless graph $\widetilde{G}$, $\mathrm{Span}\{(0,0,1)\}$ and $
\mathrm{Span}\{(1,1,0)\}$ have no common neighbor.

Combining the above two steps, we conclude that 
\begin{equation*}
 \mathrm{ex}(q^2+q+1,K_{2,2})\ge e(\widetilde{G})\ge \frac{1}{2}(q^2+q+1)q . 
\end{equation*}
By the denseness of the prime numbers, we get  
$\mathrm{ex}(n,K_{2,2})= ( \frac{1}{2} +o(1))n^{3/2}$
\end{proof} 

We mention here that the graph $\widetilde{G}$ constructed in the above 
is called the Er\H{o}s--Renyi projective polarity graph, and 
denoted by $ER_q$. 
Furthermore, we can compute the exact number of edges of $ER_q$
by applying the argument of the standard linear algebra method. 
More precisely, 
the graph $ER_q$ has $q+1$ vertices with degree $q$,  
and $q^2$ vertices with degree $q+1$. Consequently, we get 
\[  e( {ER_q})=\frac{1}{2} \bigl( (q+1)q+q^2(q+1) \bigr)=\frac{1}{2}(q+1)^2q.\]

Generally speaking, 
determining the exact value of $\mathrm{ex}(n,C_4)$ 
 seems to be an extremely difficult problem and far beyond
reach. In 1996,  F\"{u}redi \cite{Furedi96c} showed that 
for $n$ large enough with the form 
$n=q^2+q+1$, the extremal number is  
attained by the polarity graph $ER_q$ of a projective plane. 

\begin{theorem}[F\"{u}redi \cite{Furedi96c}]
{ Let $q\ge 15$ be a prime power}. Then  
\[ \mathrm{ex}(q^2+q+1, C_4)= \frac{1}{2}q(q+1)^2. \]
\end{theorem}

For the extremal number of $C_6$, 
the best possible is the following. 

\begin{theorem}[F\"{u}redi--Naor--Verstra\"{e}te \cite{FNV2006}]
\[ 0.5338 n^{\frac 4 3} \le \mathrm{ex}(n,C_6) \le 
0.6272 n^{\frac 4 3}. \]
\end{theorem}

The asymptotic extremal number for $C_8$ remains open. 

\begin{conjecture}
Let $C_8$ be the cycle on $8$ vertices. Then  
\[ \mathrm{ex}(n,C_8) =\Theta (n^{\frac 5 4}).  \]
\end{conjecture}

\begin{theorem}[Lazebnik--Ustimenko--Woldar \cite{LUW1999}]
\[ \mathrm{ex}(n,C_{10}) \ge (4\cdot 5^{- \frac 6 5} +o(1)) 
n^{\frac 6 5} \approx (0.5798 +o(1)) n^{\frac 6 5}.  \]
\end{theorem}

\begin{theorem}[Bondy--Simonovits \cite{BS1974}]
If $G$ is an $n$-vertex graph with
\[ e (G)> 100k\cdot n^{1+\frac{1}{k}}, \] 
then $G$ contains a copy of $C_{2\ell}$ for every integer $\ell \in [k,kn^{1/k}]$. 
\end{theorem}

The Bondy--Simonovits theorem implies that 
$\mathrm{ex}(n,C_{2k})\le 100k n^{1+1/k}$. Next, 
we shall list some improvements on this upper bound in the literature. 

\begin{itemize} 
\item (Verstra\"{e}te \cite{Ver2000}) $ \mathrm{ex}(n,C_{2k}) \le 8(k-1)n^{1+{1}/{k}}$.

\item (Pikhurko \cite{Pik2012}) 
$ \mathrm{ex}(n,C_{2k}) \le (k-1)n^{1+{1}/{k}} +16(k-1)n$.

\item (Bukh--Jiang \cite{BJ2017}) $ \mathrm{ex}(n,C_{2k}) \le 80\sqrt{k}\log k\cdot n^{1+1/k} +O(n)$. 

\item (He \cite{He2021}) 
$ \mathrm{ex}(n,C_{2k}) \le \bigl(16\sqrt{5}\sqrt{k\log k} + o(1)\bigr) \cdot n^{1+1/k} $. 
\end{itemize}

\begin{conjecture}[Erd\H{o}s--Simonovits, 1982] 
\[  \mathrm{ex}(n,C_{2k})=\Theta (n^{1+ \frac{1}{k}}).  \]
\end{conjecture} 

In the sequel, 
we shall review the results on extremal 
spectral problems for even cycles. 
Let $n$ be an odd integer and $G=K_1 \vee \frac{n-1}{2}K_2$ 
be the friendship graph on $n$ vertices. 
By an easy calculation, we know that 
$e(G)=\frac{3}{2}(n-1)$ and $\lambda (G)= \frac{1+ \sqrt{4n-3}}{2}$. 

\begin{theorem}[Nikiforov \cite{Niki2007b}] \label{Niki07}
Let $G$ be a graph on $n$ vertices. If $G$ has no $4$-cycle $C_4$, 
then
 $ \lambda (G)\le \frac{1+ \sqrt{4n-3}}{2}$, 
  equality holds 
if and only if  
$n$ is odd and $G=K_1 \vee \frac{n-1}{2} K_2$. 
\end{theorem}

This theorem states that if $G$ is $C_4$-free, then  
\begin{equation}  \label{eqeq49}
\frac{2m}{n} \le \lambda (G) \le \frac{1+ \sqrt{4n-3}}{2}.  
 \end{equation}
This spectral result implies $\mathrm{ex}(n,C_4)\le (\frac{1}{2} + o(1))n^{3/2}$.

The  Friendship Theorem \cite[Chapter 43]{AZ2014} states that 
if $G$ is a graph on $n$ vertices such that 
every two distinct vertices have exactly one common neighbor, 
then $n$ is odd and $G=K_1 \vee \frac{n-1}{2}K_2$,  
a graph obtained from $\frac{n-1}{2}$ triangles 
sharing a single common vertex.
Such a graph is usually called the friendship graph. 
The equality in Theorem \ref{Niki07} is best possible only for $n$ odd, and 
Nikiforov's result may be improved for even $n$. 

 Erd\H{o}s and Renyi  showed that 
if $q$ is a prime power, then the polarity graph 
$ER_q$ is $C_4$-free on 
$n=q^2+q+1$ vertices with 
$e(ER_q)= \frac{q(q+1)^2}{2}$. Thus, we have 
\begin{equation*}
  \lambda (ER_q) \ge \frac{2e(ER_q)}{n} 
= \frac{q(q+1)^2}{q^2+q+1} >q+1 -\frac{1}{q} 
= \frac{1+ \sqrt{4n-3}}{2} - \frac{1}{\sqrt{n} -1}.  
\end{equation*}
This lower bound is quite close to the upper bound 
in (\ref{eqeq49}).

\medskip 

In 2012, Zhai and Wang improved the result of Nikiforov 
when $n$ is even. 

\begin{theorem}[Zhai--Wang \cite{ZW2012}] \label{thmzw}
Let $G$ be a graph of even order $n$. 
If $G$ has no copy of $C_4$, 
then $\lambda (G)\le x_0$, where 
$x_0$ is the largest root of the equation 
\[ x^3 - x^2 - (n-1)x +1 = 0.\] 
Moreover, the equality holds if and only if 
$G$ is a graph obtained from the star  $K_{1,n-1}$ 
by adding $\frac{n}{2} -1$ disjoint  edges in its independent set. 
\end{theorem}

Let $K_{k}\vee I_{n-k}$ 
be the graph obtained by joining each vertex of $K_k$ 
to $n-k$ isolated vertices of $I_{n-k}$, 
and let $K_{k}\vee I_{n-k}^+$ be the graph obtained by adding 
one edge within the independent set of $K_{k}\vee I_{n-k}$.   
We can see that 
\begin{itemize}
\item 
$C_{2k+1} \nsubseteq K_{k}\vee I_{n-k}$ and $C_{2k+2} \nsubseteq K_{k}\vee I_{n-k}$.  

\item $C_{2k+1} \subseteq K_{k}\vee I_{n-k}$ and 
$C_{2k+2} \nsubseteq K_{k}\vee I_{n-k}^+$. 

\end{itemize}

\begin{conjecture}[Nikiforov \cite{Niki2010b}] \label{conj1}
Let $k\ge 2$ and $G$ be a graph of sufficiently large order $n$. 
If $G$ is $ C_{2k+2}$-free, then 
 \[ \lambda (G) \le \lambda (K_{k}\vee I_{n-k}^+),\]  
 equality holds if and only if $G=K_{k}\vee I_{n-k}^+$. 
\end{conjecture}

In 2020, Zhai and Lin \cite{ZL2020} proved Conjecture \ref{conj1} 
for the case $k=2$ and $n\ge 23$. 

\begin{conjecture}[Nikiforov  \cite{Niki2010b}] \label{conj2}
Let $k\ge 2$ and $G$ be a graph of sufficiently large order $n$. 
 If $G$ is both $C_{2k+1}$-free and $C_{2k+2}$-free, then 
 \[ \lambda (G) \le \lambda (K_{k}\vee I_{n-k}), \]  
 equality holds if and only if $G=K_{k}\vee I_{n-k}$. 
\end{conjecture} 

The case $k=1$ in Conjecture \ref{conj2} 
was early proved by Favaron, Mah\'{e}o and Sacl\'{e} \cite{FMS1993}. 
In 2012, Yuan, Wang and Zhai \cite{YWZ2012} proved 
conjecture \ref{conj2} for the case $k=2$ and $n\ge 6$.  
 In 2019, Gao and Hou  \cite{GH2019} proved  Conjectures 
 \ref{conj1} and \ref{conj2} under  stronger conditions. 
 More precisely, Conjectures \ref{conj1}  holds when we forbid all cycles of length at least $2k+2$. 
 Conjectures \ref{conj2}  holds when we forbid all cycles of length at least $2k+1$.

 \begin{theorem}[Gao--Hou \cite{GH2019}]
 Let $k\ge 2$ and $G$ be a graph of order $n\ge 13k^2$. 
 If $G$ does not contain a cycle of length at least $2k+2$, then 
  $\lambda (G) \le \lambda (K_{k}\vee I_{n-k}^+)$, 
 equality holds if and only if $G=K_{k}\vee I_{n-k}^+$. 
 \end{theorem}
 
  \begin{theorem}[Gao--Hou \cite{GH2019}]
 Let $k\ge 2$ and $G$ be a graph of order $n\ge 13k^2$. 
 If $G$ does not contain a cycle of length at least $2k+1$, then 
  $\lambda (G) \le \lambda (K_{k}\vee I_{n-k})$, 
 equality holds if and only if $G=K_{k}\vee I_{n-k}$. 
 \end{theorem}
 
 \noindent 
 {\bf Remark. } 
In 2022,    Cioab\u{a}, Desai and Tait \cite{CDT2022} confirmed Nikiforov's  conjecture.

\medskip 
In what follows, we shall introduce some extremal graph results in terms of the signless Laplacian spectral radius. 
For   odd $n$, we  write $F_n$ 
for the friendship graph on $n$ vertices, that is, 
$F_n=K_1\vee \frac{n-1}{2}K_2$;  
for  even $n$, we write $F_n$ for the graph 
obtained from $F_{n-1}$ by adding an extra 
edge hung to its center. 
In other words, the $F_n$ can be viewed as 
a graph obtained from $K_{1,n-1}$ by adding a maximum 
matching within the independent set $I_{n-1}$. 
The following result is an analogue of 
both Theorem \ref{Niki07} and Theorem \ref{thmzw} for 
the signless Laplacian spectral radius. 

\begin{theorem}[Freitas--Nikiforov--Patuzzi \cite{FNP2013}]
If $G$ is a graph on $n\ge 4$ vertices with no copy of $C_4$, 
then $q(G)\le  q(F_n)$, equality holds if and only if  $G=F_n$. 
\end{theorem}

In addition, 
Freitas et. al. \cite{FNP2013} 
 proposed a conjecture for forbidden even cycles of length
at least $6$, which was solved by Nikiforov and Yuan \cite{NY2015}. 

\begin{theorem}[Nikiforov--Yuan \cite{NY2015}]
Let $k\ge 2, n\ge 400k^2$ and $G$ be a graph on $n$ vertices. 
If $G$ has no copy of $C_{2k+2}$, then 
  $q(G) \le q (K_{k}\vee I_{n-k}^+)$, 
 equality holds if and only if $G=K_{k}\vee I_{n-k}^+$. 
\end{theorem}

\subsection{Spectral problem for color-critical graphs}

In this section, we shall review a special class of graphs, 
which is now known as the color-critical graphs 
(also known as edge-color-critical graphs). 
The definition is stated as below. 

\begin{definition}
Let $e$ be an edge of graph $F$. 
We say that $e$ is a color-critical edge if  
$\chi (F-e)<\chi (F)$. We say that $F$ is color-critical 
if $F$ contains a color-critical edge. 
\end{definition}  

The following are some examples. 
Clearly, the complete graph $K_{r+1}$ is color-critical and $\chi (K_{r+1})=r+1$.   
The odd cycle $C_{2k+1}$ is color-critical and $\chi (C_{2k+1})=3$. 

Let $G$ and $H$ be two graphs. 
The join of $G$ and $H$ is defined as a graph 
$G\vee H $  
obtained from the disjoint union $G \cup H$ by 
adding all edges connecting every vertex of 
$G$  to every vertex of $H$. 
For example, the wheel graph $W_n=K_1 \vee C_{n-1}$ and 
the book graph $B_t =K_2 \vee \overline{K_t}$.  
In addition, we can also see that  $K_{s_1,s_2,\ldots ,s_r} =
I_{s_1} \vee K_{s_2,\ldots ,s_r} = K_{s_1,s_2} \vee K_{s_3,\ldots ,s_r}$. 

\begin{example}[The wheel graph]
 Let $W_{2k}=K_1 \vee C_{2k-1}$ be the wheel graph, 
a vertex that connects all  vertices of an odd cycle $C_{2k-1}$. 
Then $W_{2k}$ is color-critical and $\chi (W_{2k})=4$. 
However, we can check that $W_{2k+1}=K_1 \vee C_{2k}$ is not color-critical. 
\end{example} 

Moreover, for integer $s\ge 2$, 
we define  $W_{s,m}:=K_s \vee C_m$ as the {\it generalized wheel graph}. 
Since $s\ge 2$, the generalized wheel graph $W_{s,m}$ 
is always color-critical. 
Moreover, we have $\chi (W_{s,2k-1}) =s+3$ and $\chi (W_{s,2k}) =s+2$.

\begin{example}[The book graph]
 Let $B_k=K_2 \vee I_{k-2}$ be the book graph, 
that is, $k$ triangles sharing a common edge. Then 
$B_k$ is color-critical and $\chi (B_k)=3$. 
\end{example} 

To extend the above result, we can consider 
the extremal number of the {\it generalized book graph} $B_{s,t}:=K_s \vee \overline{K_t}$ 
for some integer $s\ge 3$. This graph can also be viewed as 
 $t$ cliques $K_{s+1}$ sharing a common sub-clique $K_s$.   
We can see that $B_{s,t}$ is color-critical and 
$\chi (B_{s,t}) =s+1$.

On the other hand, the book graph 
$B_t$ is defined as the $t$ triangles sharing a common edge. 
Let $SE_{G,t}$ be the graph obtained from  
$t$ copies of $G$ by sharing a common edge.  
We can extend this definition to $t$ cliques $K_{r}$ 
sharing a common edge.  
It is easy to see that $SE_{K_{r},t}$ 
is color-critical and $\chi (SE_{K_{r},t}) =r$. 
In addition, we also can extend the book graph 
to $t$ odd-length cycles $C_{2k+1}$ sharing a common edge. 
It is  clear that $SE_{C_{2k+1},t}$ is a color-critical graph 
and $\chi (SE_{C_{2k+1},t})= 3$.

Let $K_{1,1,n_3,\ldots ,n_{r}}$ be 
the complete $r$-partite graph with two parts of order $1$ and other parts 
of order $n_i$ for each $i\in [3,r]$. 
Then $K_{1,1,n_3,\ldots ,n_{r}}$ is color-critical and 
$\chi (K_{1,1,n_3,\ldots ,n_{r}})=r$. 
Note that $SE_{K_r,t} \subseteq K_{1,1,t,\ldots ,t}$. 
Let $K_r(s_1,s_2,\ldots ,s_r)$ be the complete $r$-partite 
graph on the vertex sets $V_1$, $V_2$, $\ldots ,V_s$,  
where each part $V_i$ has  $s_i$ vertices.

\begin{example}
Let $s_1\ge 2$ be an integer. 
We write $K_{r}^+(s_1,s_2,\ldots ,s_r)$ for the graph obtained from 
$K_r(s_1,s_2,\ldots ,s_r)$ by adding an edge within the part $V_1$.  
It is easy to see that $K_{r}^+(s_1,s_2,\ldots ,s_r)$
is color-critical and $\chi =r+1$. 
\end{example}

It is important that for any color-critical graph $G$ with $\chi (G)=r+1$,
there exist some 
integers $s_1,\ldots ,s_r$ such that $G$ is a subgraph 
of $K_{r}^+(s_1,s_2,\ldots ,s_r)$.

The famous Erd\H{o}s--Stone--Simonovits Theorem \ref{thm361}
states that if $F$ is a graph with the vertex-chromatic 
number $\chi (F)=r+1$, then 
\[ t_r(n)\le \mathrm{ex}(n,F)= 
\left( 1-\frac{1}{r} \right) \frac{n^2}{2} + o(n^2) 
=t_r(n) + o(n^2).\] 
In previous section, 
we have proved that for some special graphs, 
we can remove the error term $o(n^2)$. 
For example, 
the Tur\'{a}n Theorem \ref{thmturanstrong} states that 
$\mathrm{ex}(n,K_{r+1}) =t_r(n)$, 
the Bondy--Woodall Theorem \ref{coro2103} implies that 
$\mathrm{ex}(n,C_{2k+1}) =t_2(n)$ holds for $n\ge 4k+1$, 
and the Dzido Theorem \ref{thmdzi} states that 
$\mathrm{ex}(n,W_{2k}) =t_3(n)$ holds for $n\ge 6k-10$. 
Moreover, 
a result of Edwards \cite{Edw1977},  
and Khad\v{z}iivanov--Nikiforov \cite{KN1979} 
independently 
implies that $\mathrm{ex}(n,B_{k}) =t_2(n)$ 
holds for $n\ge 6k$; see, e.g., \cite{BN2005,BN2008,BN2011} 
for more details.   

In what follows,   we present  a more general result 
proved by Simonovits.

\begin{theorem}[Simonovits \cite{Sim1966}] \label{thmSim66}
Let $F$ be a  graph with $\chi (F)=r+1$ where $r\ge 2$. 
If $F$ is color-critical, 
then   for sufficiently large $n$, 
\[ \mathrm{ex}(n,F)=t_r(n). \]
Moreover, the unique extremal graph is 
Tur\'{a}n graph $T_r(n)$. 
\end{theorem}

In 2009, Nikiforov \cite{Nikiejc2009} extended 
Theorem \ref{thmSim66} by showing that 
there is $n_0$ such that if $G$ has $n\ge n_0$  vertices and 
$\lambda (G) > \lambda (T_r(n))$, then 
$G$ contains a copy of the complete $r$-partite graph with 
parts of size $\Omega (\ln n)$ plus an extra edge. 

\begin{theorem}[Nikiforov \cite{Nikiejc2009}] \label{thmniki2009ejc}
Let $r\ge 2 $ and $2/\ln n\le c\le 1/r^{(2r+9)(r+1)}$. 
If  $G$ is a graph on $n$ vertices with 
 $\lambda (G)>\lambda (T_r(n) )$, then $G$ contains a copy of 
\[ K_r^+(\lfloor c\ln n\rfloor, \ldots ,
\lfloor c\ln n\rfloor,\lceil n^{1-\sqrt{c}}\rceil). \]
\end{theorem}

As a consequence of Theorem \ref{thmniki2009ejc}, 
we can easily get the following weak corollary.  

\begin{corollary} \label{coro173}
If $F$ is color-critical  and $\chi (F)=r+1$  where $r\ge 2$, 
then   
\[ \boxed{\mathrm{ex}_{\lambda}(n,F)=\lambda (T_r(n)) }   \]
 holds for sufficiently large $n$, and 
the unique extremal graph is $T_r(n)$. 
\end{corollary}

\noindent 
{\bf Remark. } 
Recall that the $p$-spectral radius \cite{Kee2014} of a graph $G$ is defined as 
$ \lambda^{(p)}(G) := 
\max\{ \bm{x}^T A(G) \bm{x} : \bm{x}\in \mathbb{R}^n, 
\lVert \bm{x} \rVert_p =1 \}$. 
In 2014, Keevash, Lenz and Mubayi \cite[Corollary 1.5]{Kee2014} 
extended the above spectral version in Corollary \ref{coro173} to the $p$-spectral radius 
$\lambda^{(p)}(G)$ for every $p>1$. 
Furthermore, Kang and Nikiforov \cite[Theorem 6]{KN2014} proved the $p$-spectral version 
of Theorem \ref{thmniki2009ejc} and hence generalized the result in \cite{Kee2014}, 
since the order in each part of $K_r^+(\lfloor c\ln n\rfloor)$ grows with $n$, instead of a fixed integer 
in the forbidden subgraph $K_r^+(t)$. This difference  makes the proof longer with more advanced techniques.

\subsection{Spectral  problem for intersecting triangles}
Let $F_k$ denote the $k$-fan graph which is the graph  
consisting of $k$ triangles 
that intersect in exactly one common vertex. 
Note that $F_k$ has $2k+1$ vertices. 
This notation is slightly different from that in Section \ref{subsec2.5}. 
This graph is known as the friendship graph 
because it is the only extremal graph in the 
well-known Friendship Theorem \cite[Chapter 43]{AZ2014}. 
Since $\chi (F_k)=3$, the Erd\H{o}s--Stone--Simonovits 
Theorem \ref{thm361} implies that 
$\mathrm{ex}(n,F_k)= n^2/4 + o(n^2)$. 
In 1995, Erd\H{o}s, F\"{u}redi, 
Gould and Gunderson \cite{Erdos95} proved  the following  exact result.

\begin{theorem}[Erd\H{o}s et al. \cite{Erdos95}]\label{thmErdos95} 
For every $k \geq 1$ and $n\geq 50k^2$, we have 
\[ \mathrm{ex}(n, F_k)= \left\lfloor \frac {n^2}{4}\right \rfloor+ \left\{
  \begin{array}{ll}
   k^2-k, \quad~~  \mbox{if $k$ is odd,} \\
    k^2-\frac32 k, \quad \mbox{if $k$ is even}.
  \end{array}
\right. \]
\end{theorem}

The extremal graphs of Theorem \ref{thmErdos95} are as follows.
For odd $k$ (where $n\geq 4k-1$), 
the extremal graphs are constructed by taking $T_2(n)$,
the balanced complete bipartite  graph, and embedding two vertex
disjoint copies of $K_k$ in one side.
For even $k$ (where now $n\geq 4k-3$), the extremal graphs 
 are constructed by taking $T_2(n)$ and embedding
 a graph with $2k-1$ vertices, $k^2-\frac{3}{2} k$ edges with maximum degree $k-1$ in one side. 
 \medskip
 
 The $(k,r)$-fan is the graph consisting of $k$ copies of 
 the clique $K_r$ which intersect on a single vertex, and is denoted by $F_{k,r}$. 
 In particular, we have $F_{1,r}=K_r$ and $F_{k,3}=F_k$. 
 Note that $\chi (F_{k,r})= r$.  
Similarly,  the Erd\H{o}s--Stone--Simonovits Theorem \ref{thm361}  
also implies that 
 $\mathrm{ex}(n,F_{k,r})=(1- \frac{1}{r-1})\frac{n^2}{2} + o(n^2)
 =t_{r-1}(n) + o(n^2)$. 
 In 2003, Chen, Gould, Pfender and Wei \cite{Chen03} 
proved an exact answer and generalized 
 Theorem \ref{thmErdos95} as follows. 
 
 \begin{theorem}[Chen et al. \cite{Chen03}] \label{thmChen}
 For every $k\ge 2$ and $r\ge 3$, if $n\ge 16k^3r^8$, then 
 \[ \mathrm{ex}(n,F_{k,r}) = t_{r-1}(n) +
 \left\{
  \begin{array}{ll}
   k^2-k, \quad~~  \mbox{if $k$ is odd,} \\
    k^2-\frac32 k, \quad \mbox{if $k$ is even}.
  \end{array}
\right.   \]
 \end{theorem}

 We remark here that there is a typo in \cite{Chen03} since 
 the correct condition should be $r\ge 3$. For the $r=2$ case, 
 we observe that $F_{k,2}$ is a star graph and $\mathrm{ex}(n,F_{k,2})=\Theta (n)$.

 The extremal graphs of Theorem \ref{thmChen} are constructed by taking the $(r-1)$-partite Tur\'{a}n graph $T_{r-1}(n)$ and embedding a graph $G_0$ in one vertex part, denoted by $G_{n,k,r}$. If $k$ is odd,  $G_{0}$ is isomorphic to two vertex disjoint copies of $K_k$. If $k$ is even,  $G_{0}$ is isomorphic to the graph with $2k-1$ vertices, $k^2-\frac{3}{2}k$ edges with maximum degree $k-1$.

 Let $C_{k,q}$ be the graph consisting of $k$ cycles of length $q$ which intersect 
 exactly in one common vertex. Clearly, when we set $q=3$, 
 then $C_{k,3}$ is just the $k$-fan graph; see Theorem \ref{thmErdos95}. 
 When $q$ is an odd integer, we can see that $\chi (C_{k,q})=3$, 
 the Erd\H{o}s--Stone--Simonovits theorem also implies that 
 $\mathrm{ex}(n,C_{k,q})=n^2/4 + o(n^2)$. 
 In 2016, Hou, Qiu and Liu \cite{HQL16} determined exactly the extremal number 
 for $C_{k,q}$ with $k\ge 1$ and odd integer $q\ge 5$.

 \begin{theorem}[Hou--Qiu--Liu \cite{HQL16}]  \label{thmhql16}
 For an integer $k\ge 1$ and an odd integer $q\ge 5$, 
 there exists $n_0(k,q)$ such that 
 for all $n \ge n_0(k,q)$, we have 
 \[  \mathrm{ex}(n,C_{k,q}) =  \left\lfloor \frac {n^2}{4}\right \rfloor 
 + (k-1)^2. \]
The extremal graphs are a 
 Tur\'{a}n graph $T_2(n)$ with a $K_{k-1,k-1}$ embedding 
 into one class.  
 \end{theorem} 
 
 \noindent 
 {\bf Remark.} 
 We emphasize here that when $q$ is even, then $C_{k,q}$ is a bipartite graph 
 where every vertex in one of its parts
has degree at most $2$. For such a sparse   bipartite graph, 
a classical result of F\"{u}redi \cite{Furedi91} or 
Alon, Krivelevich and Sudakov \cite{Alon03} implies that 
 $\mathrm{ex}(n,C_{k,q})= O(n^{3/2})$. 
 Recently, a breakthrough result of Conlon, Lee and Janzer \cite{CL19,CJL20} shows that 
 for even $q\ge 6$ and $k\ge 1$, we have  
$\mathrm{ex}(n,C_{k,q})= O(n^{3/2-\delta})$ for some $\delta =\delta (k,q)>0$. 
  It is a challenging problem to determine the value $\delta (k,q)$. 
 For instance, the special case $k=1$, this problem 
 reduces to determine the extremal number for even cycle.

\medskip 
 Next, we  introduce a unified extension of both Theorems \ref{thmErdos95} 
 and  \ref{thmhql16}. 
 Let $s$ be a positive integer and $t_1,\ldots ,t_k\ge 5$ be odd integers.  
We write $H_{s,t_1,\ldots ,t_k}$ for 
 the graph consisting of $s$ triangles and 
 $k$ odd cycles of  lengths $t_1, \ldots ,t_k $
 in which these triangles and cycles intersect in
exactly one common vertex. 
 The graph $H_{s,t_1,\ldots ,t_k}$ is also known as the flower graph 
with $s+k$ petals. 
We remark  that the $k$ odd cycles 
 can have different lengths. 
 Clearly, when $t_1=\cdots =t_k=0$, then $H_{s,0,\ldots ,0}=F_s$, 
 the $s$-fan graph; see Theorem \ref{thmErdos95}.  
  In addition,  when $s=0$ and  $t_1=\cdots =t_k=q$, 
  then $H_{0,q,\ldots ,q}=C_{k,q}$; 
 see Theorem \ref{thmhql16}. 
 
 In 2018, Hou, Qiu and Liu \cite{HQL18} and Yuan \cite{Yuan18} 
 independently determined the extremal number 
 of $H_{s,k}$ for $s\ge 0 $ and $k\ge 1$.  
 Let $\mathcal{F}_{n,s,k}$ be the family of graphs with each member being a 
 Tur\'{a}n graph $T_2(n)$ with
a graph $H$ embedded in one partite set, where 
 \[  H = \begin{cases}
 K_{s+k-1,s+k-1}, & \text{if $(s,k)\neq (3,1)$,} \\
 K_{3,3} ~\text{or}~ 3K_3, & \text{if $(s,k)=(3,1)$,}
 \end{cases}  \]
where $3K_3$ is the union of three disjoint triangles.

 \begin{theorem}[Hou--Qiu--Liu \cite{HQL18}; 
 Yuan \cite{Yuan18}] \label{thmHY}
    For every graph $H_{s,t_1,\ldots ,t_k}$  with $s\ge 0$ and $ k\ge 1$, 
 there exists $n_0$ such that 
 for all $n \ge n_0$, we have 
 \[  \mathrm{ex}(n,H_{s,t_1,\ldots ,t_k}) =  \left\lfloor \frac {n^2}{4}\right \rfloor 
 + (s+k-1)^2. \]
Moreover, the only extremal graphs for $H_{s,t_1,\ldots ,t_k}$ are members of $\mathcal{F}_{n,s,k}$.
 \end{theorem}

Another interesting problem 
on this topic is to 
determine the Tur\'{a}n number  
of $C_{k,q}$ for even $q$. 
More general, it is  challenging 
to determine the  Tur\'{a}n number 
of $H_{s,t_1,\ldots ,t_k}$ where the cycles have even lengths.

Recall that $\mathrm{Ex}(n,F)$ denotes the set of graphs 
that contain no copy of $F$ and have maximum number of edges. 
For example, 
the Tur\'{a}n theorem 
gives $\mathrm{Ex}(n,K_{r+1})= \{T_r(n)\}$ for every $r\ge 2$, 
and the Simonovits theorem yields that 
for every fixed color-critical graph $F$ with $\chi(F)=r+1\ge 3$, 
we have $\mathrm{Ex}(n,F)=\{T_r(n)\}$ for sufficiently large $n$. 
In 2020, Cioab\u{a}, Feng, Tait and Zhang \cite{CFTZ20} proved 
the spectral version of Theorem \ref{thmErdos95}.

 \begin{theorem}[Cioab\u{a} et al. \cite{CFTZ20}]    \label{thmCFTZ20}
Let $k\ge 2$ and $G$ be a graph of order $n$ that does not contain a copy of $F_k$.  
For sufficiently large $n$, if $G$ has the maximal spectral radius, then
$$G \in \mathrm{Ex}(n, F_k).$$
\end{theorem}

In 2021, it was  proved by Li and Peng in 
\cite{LP2021} that the spectral version of 
intersecting odd cycles is also similar. 
Soon after, Desai, Kang, Li, Ni, Tait and Wang \cite{DKLNTW2021} 
proved the analogous result for intersecting cliques. 

 \begin{theorem}[Li--Peng \cite{LP2021}] \label{thmlp}
Let $G$ be a graph of order $n$ that does not contain a copy of $H_{s,t_1,\ldots ,t_k}$, where $s\ge 0$ and $k \geq 1$.  
For sufficiently large $n$, if $G$ has the maximal spectral radius, then
$$G \in \mathrm{Ex}(n, H_{s,t_1,\ldots ,t_k}).$$
\end{theorem}

\begin{theorem}[Desai et al. \cite{DKLNTW2021}]  \label{thmDe}
Let $G$ be a graph of order $n$ that does not contain a copy of $F_{k,r}$, where $k\ge 2$ and $r \geq 3$.  
For sufficiently large $n$, if $G$ has the maximal spectral radius, then
\[ G \in \mathrm{Ex}(n, F_{k,r}).\] 
\end{theorem}

Recently, 
Cioab\u{a}, Desai and Tait \cite{CDT21}  
investigated the largest spectral radius of
 an $n$-vertex graph that does not contain the odd-wheel 
 graph $W_{2k+1}$, 
which is the graph obtained by joining a vertex to a cycle 
of length $2k$. Moreover, they raised 
the following more general conjecture. 

\begin{conjecture} \label{conj}
Let $F$ be any graph such that the graphs in $\mathrm{Ex}(n,F)$ 
are Tur\'{a}n graphs plus $O(1)$ edges. 
Then for sufficiently large $n$, 
a graph attaining the maximum spectral radius 
among all $F$-free graphs is a member of $\mathrm{Ex}(n,F)$. 
\end{conjecture} 

\noindent 
{\bf Remark.} 
This conjecture was recently confirmed by Wang, Kang and Xue \cite{Wang2022}.

\medskip 
Recall that $F$ is called edge-color-critical if  
there exists an edge $e$ of $F$ such that 
$\chi (F-e) < \chi (F)$. 
Let $F$ be an edge-color-critical graph with $\chi (F)=r+1$. 
By a result of Simonovits \cite{Sim1966}
and a result of Nikiforov \cite{Nikiejc2009} or 
Keevash et al. \cite{Kee2014}, 
we know that $\mathrm{Ex}(n,F)=\mathrm{Ex}_{\lambda}(n,F) =
\{T_r(n)\}$ for sufficiently large $n$, this  
 shows that Conjecture \ref{conj}  is true for 
all edge-color-critical graphs.  
As we mentioned before, Theorems \ref{thmCFTZ20}, 
\ref{thmDe} and \ref{thmlp} 
say that Conjecture \ref{conj} 
holds for the $k$-fan graph $F_k$, 
the $(k,r)$-fan graph $F_{k,r}$ and the intersecting odd cycle 
$H_{s,t_1,\ldots ,t_k}$. 
Note that these graphs
 $F_k,F_{k,r}$ and $H_{s,t_1,\ldots ,t_k}$ 
 are not edge-color-critical. 

Recall that  $S_{n,k}=K_k \vee I_{n-k}$.  
Clearly, we can see that 
$S_{n,k}$  does not contain $F_k$ as a subgraph. 
Recently, Zhao, Huang and Guo \cite{ZHG21} 
proved that $S_{n,k}$ is 
the unique graph attaining the maximum signless Laplacian spectral radius among all graphs of large order $n$ containing no 
copy of $F_k$. 

\begin{theorem}[Zhao--Huang--Guo  \cite{ZHG21}]\label{thmZHG}
Let $k \geq 2$ and $n\ge 3k^2-k-2$. 
If $G$ is an $n$-vertex graph  that does not contain a copy of $F_k$,  then $q(G)\le q(S_{n,k})$, 
with equality holding if and only if $G=S_{n,k}$. 
\end{theorem} 

It is worth mentioning that the extremal graphs in 
Theorem \ref{thmZHG} 
are not the same as those of Theorem \ref{thmCFTZ20}. 
In addition, for $k=1$, i.e., $G$ is triangle-free, 
from Theorem \ref{thmHJZ} we know that $q(G)\le n$, 
equality holds  
if and only if $G$ is a complete bipartite graph.  

It is a natural question 
to consider the maximum signless Laplacian spectral radius 
among all graphs containing no copy of $C_{k,q}$ or 
$H_{s,t_1,\ldots ,t_k}$. 
When the paper \cite{LP2021} was 
announced (arXiv:2106.00587), 
Chen, Liu and Zhang \cite{CLZ2021} 
proved quickly the signless Laplacian spectral version 
for $C_{k,q}$. More generally, 
they showed the result for graph $H_{s,t_1,\ldots ,t_k}$ 
for odd integers $t_1,\ldots ,t_k$.

\begin{theorem}[Chen--Liu--Zhang \cite{CLZ2021}]
For integers $k\ge 2, t\ge 1$ and $q=2t+1$, 
there exists an integer $n_0$ such that 
if $n\ge n_0$ and 
$G$ is a $C_{k,q}$-free graph on $n$ 
vertices, then 
\[  q(G) \le q(S_{n,kt}), \]
equality holds if and only if $G=S_{n,kt}$. 
\end{theorem}

It is natural to ask the following problem, 
which is the signless Laplacian spectral version of 
the extremal problem for intersecting cliques. 
Clearly, when $r=3$, 
our conjecture reduces to the result of Zhao et al. \cite{ZHG21}. 

\begin{conjecture} \cite{DKLNTW2021}  
For integers $k\ge 2$ and $ r\ge 3$, 
there exists an integer $n_0(k,r)$ such that 
if $n\ge n_0(k,r)$ and 
$G$ is a $F_{k,r}$-free graph on $n$ 
vertices, then 
\[  q(G) \le q(S_{n,k(r-2)}), \]
equality holds if and only if $G=S_{n,k(r-2)}$. 
\end{conjecture}

\section{Spectral  problem for Hamiltonianity} 

A cycle passing through all  vertices of a graph is called a  Hamilton cycle. A graph containing a Hamilton cycle is called a Hamiltonian graph. A path passing through all  vertices of a graph is called a Hamilton path and a graph containing a Hamilton path is said to be traceable.

Every complete graph on at least three vertices is evidently Hamiltonian. Indeed,
the vertices of a Hamilton cycle can be selected one by one in an arbitrary order.
But suppose that our graph has considerably fewer edges. 
In particular, we may
ask how large the minimum degree must be in order to guarantee the existence of
a Hamilton cycle. The  celebrated  Dirac theorem 
\cite{Dirac52} answers this question. 
It states that every graph with $n\ge 3$ vertices and minimum degree at least $n/2$ has
a Hamilton cycle. 
The condition is sharp when we consider the complete bipartite 
graph with the parts of sizes $\lfloor \frac{n-1}{2}\rfloor$ and 
$\lfloor  \frac{n+1}{2}\rfloor$.

The {\it closure operation} introduced by Bondy and Chv\'{a}tal \cite{Bondy} 
is a powerful tool for the problems of Hamiltonicity of graphs. 
Let $G$ be a graph of order $n$. The $s$-closure of $G$, 
denoted by $\mathrm{cl}_s(G)$, 
is the graph obtained from $G$ by recursively joining pairs of 
non-adjacent vertices whose degree sum is at least $s$ until no such pair 
remains.  
It is not hard to prove that the $s$-closure 
of $G$ is uniquely determined; see, e.g., \cite{Bondy}.  
Clearly, if $G$ contains a Hamilton cycle, 
then so does $\mathrm{cl}_{s}(G)$ for every $s$ 
since $G$ is a subgraph of $\mathrm{cl}_s(G)$.   
Surprisingly, Ore \cite{ore60} proved that both 
$\mathrm{cl}_n(G)$ and $G$ keep in line with the existence of Hamilton cycle. 

\begin{theorem}[Ore \cite{ore60}] 
A graph $G$ is Hamiltonian 
if and only if the closure graph $\mathrm{cl}_n (G)$ is Hamiltonian. 
In particular, if $d(u)+d(v)\ge n$ for all non-edges $\{u,v\}$, 
then $\mathrm{cl}_n (G)=K_n$ and $G$ is Hamiltonian. 
\end{theorem}

The proof is simple and similar with that of the Dirac theorem, 
so we here conclude the proof briefly.  
Without loss of generality, 
we only consider the case $\mathrm{cl}_n(G) =G+ uv $, 
where $d_G(u)+d_G(v)\ge n$. If $G+uv$ is Hamiltonian, 
then $G$ has a Hamiltonian path $u=u_1,u_2,\ldots ,u_n=v$. 
We denote by $S=\{i : uu_{i+1}\in E(G)\}$ 
and $T=\{i: u_iv \in E(G)\}$. Clearly, we have 
$|S|=d_G(u) $ and $|T|=d_G(v)$.  Note that $u_n \notin S\cup T$, 
which implies that $|S\cup T| \le n-1$. 
Since $d_G(u)+d_G(v)\ge n$, we can get 
$|S\cap T| = |S| + |T| -|S\cup T| \ge d_G(u) + d_G(v) - (n-1)\ge 1$. 
Thus there must be some $i$ such that $u$ is adjacent to $u_{i+1}$ and 
$v$ is adjacent to $u_i$, which implies that $G$ has the Hamilton cycle 
$u_1u_{i+1} u_{i+2}\cdots u_n u_i u_{i-1} \cdots u_1$.

In 1972, 
Chv\'{a}tal \cite{Chvatal72} proved 
the following theorem, which 
characterizes the degree sequence 
of Hamiltonian graph. 
Theorem \ref{thmchv} 
becomes a classic conclusion in many comprehensive textbooks; 
see \cite[pp. 485--488]{Bondy08} 
or \cite[pp. 288--290]{West96} or \cite[pp. 308--312]{Diestel17} 
for more details.

\begin{theorem}[Chv\'{a}tal \cite{Chvatal72}] \label{thmchv}
Let $G$ be an $n$-vertex graph with degree sequence 
$d_1\le d_2 \le \cdots \le d_n$.    
If $G$ is not Hamiltonian, 
then there exists an integer $i< {n}/{2}$ such that 
$d_i \le i$ and $d_{n-i} \le n-i-1$. 
\end{theorem}

The following result is due to 
Ore \cite{ore60} and Bondy \cite{Bondy72} independently. 
It is a direct consequence of the Chv\'{a}tal theorem; 
see \cite[p. 60]{Bondy76} for more details. 

\begin{theorem}[Ore \cite{ore60}, Bondy \cite{Bondy72}] 
\label{coroore}
Let $G$ be a graph on $n\ge 3$ vertices. If 
\[ e(G)\ge {n-1 \choose 2} +1, \]
then $G$ has a Hamilton cycle unless 
$G=K_1\vee (K_1\cup K_{n-2}) $ or $n=5$ 
and  $G= K_2\vee I_3$.  
\end{theorem}

In 1962, Erd\H{o}s improved the above result 
for graphs with given minimum degree. 

\begin{theorem}[Erd\H{o}s \cite{erdos62}] \label{thmerdos62}
Let $G$ be a graph on $n$ vertices. If 
the minimum degree $\delta (G)\ge \delta $  and 
\[ e(G)>\max\left\{ {n-\delta \choose 2}+\delta^2, 
{n- \lfloor \!\frac{n-1}{2} \! \rfloor \choose 2}+
\left\lfloor \! \frac{n-1}{2} \! \right\rfloor^2 \right\}, \]
then $G$ has a Hamilton cycle. 
\end{theorem}

We mention here that $f(\delta) = {n-\delta \choose 2} + \delta^2$ 
is decreasing with respect to $\delta$ for fixed $n\ge 6\delta$. 
It is sufficient to prove the theorem 
under the condition $\delta (G)=\delta$.  
Indeed, if $G$ satisfies $\delta (G)=t >\delta$ and  $e(G)> f(\delta)$, 
then we have 
$e(G)> f(t)$. 
The  special case $\delta (G)=t$ can deduce  
the general case $\delta (G)>\delta$. 
Hence, the condition can be reduced to $\delta (G)=\delta$.  

To see the sharpness of the bound in Theorem \ref{thmerdos62}, 
we consider the graph 
$H_{n,\delta}$ obtained from a copy of $K_{n-\delta} $ 
by adding an independent set of $\delta$ vertices with degree $\delta$ 
each of which is adjacent to the same $\delta$ vertices in $K_{n-\delta}$. 
In the language of graph join and union, 
that is, 
\begin{equation} \label{eqhnd}
  \boxed{H_{n,\delta}:=K_{\delta} \vee ( K_{n-2\delta}\cup I_{\delta}). } 
  \end{equation}
Clearly, we can verify that $\delta (H_{n,\delta})=\delta$ and 
$H_{n,\delta}$ does not contain Hamilton cycle.  
Moreover, we have 
$  e(H_{n,\delta})={n-\delta \choose 2} + \delta^2$ 
and $\lambda (H_{n,\delta}) > 
\lambda (K_{n-\delta}) =n-\delta -1$ 
since $K_{n-\delta}$ is a proper subgraph of $H_{n,\delta}$.   
When $n\ge 6\delta$, we have  
$e(H_{n,\delta}) \ge  e(H_{n,\lfloor {(n-1)}/{2} \rfloor})$.  
Thus, we  get the following corollary. 

\begin{corollary}[Erd\H{o}s] \label{coro34}
Let $k\ge 1$  and $n\ge 6\delta$. 
If $G$ is an $n$-vertex graph with $\delta (G)\ge \delta$ and 
$e (G)\ge e(H_{n,\delta})$,
then either 
$G$ has a Hamilton cycle 
or $G=H_{n,\delta}$. 
\end{corollary}

\subsection{Spectral results for Hamilton cycle}

Apparently there are very few sufficient conditions
for the existence of a Hamilton cycle in  graphs. 
As it turns out spectral properties of graphs can supply rather powerful
sufficient conditions for Hamiltonicity.  
We denote by $\mu :=\max\{ |\lambda_i| : 2\le i\le n\}$ the second largest eigenvalue in absolute value. 
A famous such result  proved by Krivelevich and Sudakov \cite{KS03}
 states that if $G$ is a $d$-regular graph with sufficiently large order $n$ such that 
\[  \mu \le \frac{(\log \log n)^2}{1000 \log n (\log \log \log n)}d, \]
then $G$ has a Hamilton cycle. 
The proof of this result is quite involved technically. 
We omit the proof details here, referring the reader to \cite{KS03}. 
The parameter $\mu$ in the above formulation is usually called  the second eigenvalue 
of the $d$-regular graph $G$ (the first and the trivial eigenvalue being $\lambda_1=d$). 
To some extent, this terminology is  inaccurate, as in fact $\mu = \max\{\lambda_2,-\lambda_n\}$. 
We will call a $d$-regular graph $G$ on $n$ vertices in which all eigenvalues, but the first one, are at most 
$\mu$ in their absolute values, an $(n,d,\mu)$-graph. 
Krivelevich and Sudakov \cite{KS03} also conjectured that 
an even stronger result is true, namely, 
there exists a positive constant $C$ such that for large enough $n$, any $(n,d,\mu)$-graph 
that satisfies $ \mu /d <  C$ contains a Hamilton cycle.

\medskip 
In 2010, Fiedler and Nikiforov 
proved a spectral version of  Theorem \ref{coroore}. 

\begin{theorem}[Fiedler--Nikiforov \cite{FiedlerNikif}] \label{thmFN}
If $G$ is a graph on $n\ge 3$ vertices and 
\[  \lambda (G)> n-2, \] 
then either $G$ has a Hamilton cycle 
or $G=K_1\vee (K_1\cup K_{n-2})$. 
\end{theorem}

This spectral result can easily be deduced from 
 Theorem \ref{coroore}. 

\begin{proof}
The well-known Stanley inequality asserts that 
$\lambda (G)\le -1/2 + \sqrt{2e(G) + 1/4}$, which together with $\lambda (G)> n-2$, implies that 
$e(G) > {n-1 \choose 2}$. 
Hence, 
for $n\neq 5$, 
the desired result follows from  Theorem \ref{coroore}. 
For $n=5$, if $G$ is the another possible exception   $K_2 \vee I_3$,  
we can calculate that $\lambda (K_2\vee I_3)=3$, 
which contradicts the condition $\lambda (G)> n-2$. This completes the proof.  
\end{proof}

In 2013, Yu and Fan \cite{YuGuidong} 
gave the corresponding spectral version 
for the signless Laplacian radius. 
Recall that $q(G)$ stands for the 
  signless Laplacian spectral radius, i.e., 
 the largest eigenvalue of 
 the {\it signless Laplacian matrix} $Q(G)=D(G) + 
 A(G)$, where $D(G)=\mathrm{diag}(d_1,\ldots ,d_n )$ 
 is the degree diagonal matrix and 
 $A(G)$ is the adjacency matrix. 

\begin{theorem}[Yu--Fan \cite{YuGuidong}] \label{thmYF}
If $G$ is a graph on $n\ge 3$ vertices and 
\[ q(G)>2(n-2), \] 
then $G$ has a Hamilton cycle  
or $G=K_1\vee (K_1\cup K_{n-2})$, 
or $n=5$ and $G=K_2\vee I_3$. 
\end{theorem} 

In \cite{YuGuidong}, 
the counterexample of 
$n= 5,G=K_2\vee I_3$ is missed.  
This tiny flaw has already been pointed 
out  by Liu et al. \cite{Liuruifang} and 
by Li and Ning \cite{LiBinlong} as well.  

\begin{proof}
An important inequality in \cite{FengPIMB09} states that 
$q(G)\le \frac{2e(G)}{n-1} + n-2$, 
which together with $q(G)> 2(n-2)$ yields that 
$e(G)> {n-1 \choose 2}$.  
Hence, 
for $n\neq 5$, 
the desired result follows from  Theorem \ref{coroore}. 
For $n=5$, when $G$ is the another possible exception   $K_2 \vee I_3$,  
we can calculate that $q (K_2\vee I_3)= \frac{7+\sqrt{33}}{2} 
\approx 6.372$. 
This completes the proof.  
\end{proof}

Observing that a graph $G$ with minimum degree $\delta (G)=1$ 
does not contain Hamilton cycle. Thus  
$\delta (G) \ge 2$ is 
a trivial necessary condition for $G$ to be Hamiltonian. 
In 2015, 
Ning and Ge \cite{NGLMA15} refined the 
extremal result of Ore and 
the spectral result of Fiedler and Nikiforov 
for graphs with minimum degree at least two. 

\begin{theorem}[Ning--Ge \cite{NGLMA15}] 
Let $G$ be a graph on $n\ge 5$ vertices and $\delta (G)\ge 2$. 
If 
\[ e(G) \ge {n-2 \choose 2} +4, \]
then $G$ has a Hamilton cycle unless 
$G\in \{K_2 \vee (K_{n-4} \cup I_2), 
K_3 \vee I_4, 
K_2 \vee (K_{1,3} \cup K_1), 
K_1 \vee K_{2,4}, 
K_3 \vee (K_2 \cup I_3), 
K_4 \vee I_5, 
K_3 \vee (K_{1,4} \cup K_1), 
K_2 \vee K_{2,5}, 
K_5 \vee I_6\}$. 
\end{theorem}

\begin{theorem}[Ning--Ge \cite{NGLMA15}] \label{thmng}
Let 
$n\ge 14$ be an integer. 
If $G$ is a graph on $n$ vertices with  minimum degree 
$\delta (G)\ge 2$ and 
 \[ \lambda (G)\ge \lambda (K_2\vee (K_{n-4}\cup 
I_2)), \] 
 then either $G$ has a Hamilton cycle or $G=K_2\vee (K_{n-4}\cup 
I_2)$.  
\end{theorem}

Upon some computations, 
we can find the following exceptions: 
when $n=7$, $K_3\vee I_4$ is an exception 
that contains no Hamilton cycle 
since  $\lambda (K_3 \vee I_4) =1+ \sqrt{13} \approx 4.605 
> \lambda (K_2 \vee (K_3 \cup I_2)) \approx 4.404$;  
when $n=9$, $K_4 \vee I_5$ contains no Hamilton cycle 
and $\lambda (K_4 \vee I_5) =\frac{3 +\sqrt{89}}{2} 
\approx 6.217 > \lambda (K_2 \vee (K_5 \cup I_2)) \approx 6.197$.  
Moreover, Ning and Ge \cite{NGLMA15} conjecture that the bound $n\ge 14$ 
in Theorem \ref{thmng} 
can be sharpened to $n\ge 10$.  
Soon after, Chen, Hou and Qian \cite[Theorem 1.6]{CHQ2018} proved that 
 Theorem \ref{thmng} holds for every integer $n\ge 10$.

In addition, Chen et al. \cite{CHQ2018} also proved the version 
for the signless Laplacian radius. 

\begin{theorem}[Chen--Hou--Qian \cite{CHQ2018}] \label{thmchq}
Let $n\ge 11$ be an integer. 
If $G$ is a graph on $n$ vertices with  minimum degree 
$\delta (G)\ge 2$ and 
\[ q(G)\ge q(K_2\vee (K_{n-4}\cup I_2)), \] 
then $G$ has a Hamilton cycle unless 
$G=K_2\vee (K_{n-4}\cup I_2)$. 
\end{theorem}

In 2016,  Benediktovich \cite{Ben2016}
 improved slightly the result of 
Ning and Ge. 

\begin{theorem}[Benediktovich \cite{Ben2016}] \label{thmben}
Let $G$ be a graph on $n\ge 8$ vertices with $\delta \ge 2$. 
If 
\[ \lambda (G)\ge n-3,\]
 then $G$ has a Hamilton cycle  
unless $G=K_2 \vee (K_{n-4} \cup I_2)$ or  
$G=K_1 \vee (K_{n-3} \cup K_2)$, or $n=9$, 
there are two more exceptions  
$G\in \{ I_5 \vee K_4, K_3 \vee (K_{1,4} \cup K_1)\}$. 
\end{theorem}

By an easy computation, we know that 
$\lambda (I_5 \vee K_4) = \frac{3+\sqrt{89}}{2} \approx 6.217$ 
and $\lambda (K_3 \vee (K_{1,4} \cup K_1)) \approx 6.032$. 
These two graphs do not contain Hamilton cycle. 

By introducing the minimum degree of a graph as a new parameter, Li and Ning \cite[Theorems 1.5, 1.8]{LiBinlong} obtained 
the  spectral analogue 
of the result of Erd\H{o}s (Corollary \ref{coro34}).

\begin{theorem}[Li--Ning \cite{LiBinlong}] \label{thmln16a}
Suppose $\delta \ge 1$  and $n\ge \max\{6\delta +5,(\delta^2+6\delta+4)/2\}$. 
If $G$ is an $n$-vertex graph with $\delta (G)\ge \delta$ and 
\[ 
\lambda (G)\ge \lambda (H_{n,\delta}),
\] 
then either 
$G$ has a Hamilton cycle 
or $G=H_{n,\delta}$. 
\end{theorem}

\begin{theorem}[Li--Ning \cite{LiBinlong}] \label{thmln16b}
Suppose $\delta \ge 1$ and $n\ge \max\{6\delta+5,(3\delta^2+5\delta+4)/2\}$. 
If $G$ is an $n$-vertex graph with $\delta (G)\ge \delta$ and 
\[ 
q (G)\ge q (H_{n,\delta}),
\] 
then either 
$G$ has a Hamilton cycle 
or $G=H_{n,\delta}$. 
\end{theorem} 

Clearly, Theorem \ref{thmln16a} and 
Theorem \ref{thmln16b} extended 
 Theorem \ref{thmng} and 
 Theorem \ref{thmchq} 
 for sufficiently large $n$, respectively. 
Although Theorem \ref{thmln16a} 
and Theorem \ref{thmln16b} are algebraic, 
their proofs  need some detailed 
graph structural analysis. 
The key part of the proof of these theorems 
attributes to the stability result for Hamilton cycle.

Soon after, 
Nikiforov \cite[Theorem 1.4]{Nikiforov} 
further proved the following result, 
which improved  Theorem \ref{thmln16a}. 
Recall that $H_{n,\delta}=K_{\delta} \vee ( K_{n-2\delta}\cup I_{\delta})$ is not Hamiltonian 
and $\delta (H_{n,\delta})=\delta$. 
For notational convenience, we need to introduce 
a new graph. We denote  
\begin{equation} \label{eqlnd}  
  \boxed{ L_{n,\delta}:= K_1 \vee ( K_{n-\delta-1}\cup K_{\delta}). }  
  \end{equation}
   Trivially, the case $\delta =1$ yields 
   $H_{n,1}=L_{n,1}=K_1\vee (K_{n-2}\cup K_{1})$. 
   We can observe that $L_{n,\delta}$ does not contain 
    Hamilton cycles and $\delta (L_{n,\delta}) =\delta$. 
   Moreover, we have 
   $e(L_{n,\delta})= {n-\delta \choose 2} + {\delta+1 \choose 2}$ 
   and $\lambda (L_{n,\delta}) > \lambda (K_{n-\delta})=n-\delta-1$.

\begin{theorem}[Nikiforov \cite{Nikiforov}] \label{thmniki}
Suppose that $\delta \ge 1$ and $n\ge \delta^3+\delta+4$. 
If $G$ is an $n$-vertex graph with $\delta (G)\ge \delta$ and 
\[ 
\lambda (G) \ge n-\delta -1,
\] 
then $G$ has a Hamilton cycle, 
or $G= H_{n,\delta}$, 
or $G= L_{n,\delta}$.  
\end{theorem}

Note that $\lambda (H_{n,\delta}) > 
\lambda (K_{n-\delta}) =n-\delta -1$, 
so the condition in Theorem \ref{thmniki} 
 is weaker and succincter than  that in Theorem \ref{thmln16a}.  
 Moreover, we can show that 
   $e(H_{n,\delta}) > e(L_{n,\delta})$ 
   and $\lambda (H_{n,\delta}) > \lambda (L_{n,\delta})$. 
  Theorem \ref{thmniki} 
   is a unified improvement on 
   Theorem \ref{thmben}) as well as 
Theorem \ref{thmln16a}. 
Indeed, we can easily see that 
  $\lambda (H_{n,\delta}) > n-\delta -1$. 
Furthermore, by applying the Kelmans operation, 
 we get  $\lambda (H_{n,\delta}) > \lambda (L_{n,\delta})$; see \cite[Lemma 6]{LiBinlong} for more details.

Although Theorem \ref{thmniki} 
is indeed a generalization of Theorem \ref{thmln16a}. 
Unfortunately, one dissatisfaction in Theorem \ref{thmniki} 
is that 
the requirement of the order of graph is stricter than 
that in Theorem \ref{thmln16a}.  
One open problem is to relax this requirement. 
In fact, by a tiny modification of the proof of 
Theorem \ref{thmln16a} in \cite{LiBinlong}, 
we can prove that if $G$ is non-Hamiltonian and 
$\lambda (G)\ge n-\delta -1$, then $e(G)> e(H_{n,\delta +1})$. 
The stability result (see \cite[Lemma 2]{LiBinlong}) implies that 
$G$ is a  subgraph of $H_{n,\delta}$ and $L_{n,\delta}$. 
As pointed out by Nikiforov \cite{Nikiforov}, the crucial point of 
the argument of Theorem \ref{thmniki} 
is based on proving that for large $n\ge \delta^3+\delta +4$, 
if $G$ is a subgraph of $H_{n,\delta}$ with $\delta (G)\ge \delta$, 
i.e., $G$ is obtained from $H_{n,\delta}$ by deleting edges from the clique $K_{n-\delta}$, then one can show that 
$\lambda (G)< n-\delta-1$, unless $G=H_{n,\delta}$. 
The same argument holds for the subgraph of $L_{n,\delta}$. 
 Generally speaking, for sufficiently large $n$ 
 with respect to $\delta$, 
 both $H_{n,\delta}$ and $L_{n,\delta}$ consist 
 of a large clique $K_{n-\delta}$ together with 
 a few number of outgrowth edges. 
 We know that $\lambda (H_{n,\delta})$ and $\lambda (L_{n,\delta})$ 
 are very close to $\lambda (K_{n-\delta})=n-\delta-1$. 
 The key idea of Nikiforov exploits the fact that if 
 $G$ is a subgraph of $H_{n,\delta}$ or $L_{n,\delta}$ with $\delta (G)\ge \delta$, then 
 all the outgrowth edges contribute to $\lambda (G)$ 
 much less than a single edge within the dense clique $K_{n-\delta}$, 
 thus $\lambda (G)< \lambda (K_{n-\delta})=n-\delta-1$;  
see \cite[Theorem 1.6]{Nikiforov} for more details.

On the other hand, a natural question is that whether 
the value bound 
$q(G)\ge 2(n-\delta -1) $ corresponding to Theorem  \ref{thmln16b}
 holds or not. The similar problems under the condition of signless Laplacian spectral
radius of graph seems much more complicated. 
In 2018, Li, Liu and Peng \cite{LLPLMA18} 
 solved this problem completely. 
They 
gave the corresponding improvement on 
the result of the signless Laplacian spectral radius 
in Theorem \ref{thmln16b}. 
Interestingly, the extremal graphs in their theorems are quite different from those in
Theorems \ref{thmln16b}.  

Recall that $H_{n,\delta} 
=K_{\delta} \vee (K_{n-2 \delta} \cup I_{\delta})$, 
we denote $X=
\{v\in V(H_{n,\delta}) : d(v)=\delta\}, 
Y= \{v\in V(H_{n,\delta}): d(v)=n-1\}$ 
and $Z=\{v\in V(H_{n,\delta}): d(v)=n- \delta -1\}$. 
Let $E_1(H_{n,\delta})$ be the set of those edges of $H_{n,\delta}$ 
whose both endpoints are from $Y\cup Z$. 
We define 
\[ \mathcal{H}_{n,\delta}^{(1)} = 
\left\{ H_{n,\delta} \setminus E': E' \subseteq E_1(H_{n,\delta})
 ~\text{with}~|E'|\le 
\lfloor {\delta^2}/{4}\rfloor \right\}. \]
Here, we denote by $H_{n,\delta} \setminus E'$ the graph obtained from 
$H_{n,\delta}$ by deleting all edges of the edge set $E'$.  
Similarly, for the graph $L_{n,\delta}$, 
we denote 
$X=\{v\in V(L_{n,\delta}): d(v)=\delta \}, 
Y=\{v\in V(L_{n,\delta}): d(v)=n-1\}$ 
and $Z=\{v\in V(L_{n,\delta}): d(v)=n-\delta -1\}$. 
The notation is clear although we used the same alphabets 
to denote the sets of vertices. 
It is easy to see that $Y$ contains only one vertex. 
We use $E_1(L_{n,\delta})$ to denote 
the set of edges of 
$L_{n,\delta}$ whose both endpoints are from $Y\cup Z$. 
We define 
\[ \mathcal{L}_{n,\delta}^{(1)} = 
\Bigl\{ L_{n,\delta} \setminus E': E' \subseteq E_1(L_{n,\delta})
 ~\text{with}~|E'|\le 
\lfloor {\delta}/{4}\rfloor \Bigr\}. \]
First of all, we can show that 
if $G\in \mathcal{H}_{n,\delta}^{(1)} \cup 
\mathcal{L}_{n,\delta}^{(1)}$, 
then $G$ contains no Hamilton cycle and 
$q(G)\ge 2(n-\delta -1)$. 
Indeed, recall the subsets 
$X,Y$ and $Z$ defined as above. 
For each case, we define a vector $\bm{h}$ such that 
$h_v =1$ for every $v\in Y\cup Z$ and 
$h_v=0$ for every $v\in X$. 
Note that $q(K_{n-\delta } \cup I_{\delta} ) =q(K_{n-\delta }) 
=2(n-\delta -1)$ 
and $\bm{h}$ is a corresponding eigenvector. 
If $G\in  \mathcal{H}_{n,\delta}^{(1)}$, then we get 
\[  \bm{h}^TQ(G)\bm{h} - \bm{h}^TQ ({K_{n-\delta } \cup I_{\delta}}) \bm{h} 
= \delta^2 - 4|E'| \ge 0. \]
By the Rayleigh Formula, we have 
\[ q(G) 
 \ge \frac{ \bm{h}^TQ (G)\bm{h} }{ \bm{h}^T\bm{h}  }  
 \ge \frac{\bm{h}^TQ ({K_{n-\delta} \cup I_{\delta}} )\bm{h}}{\bm{h}^T\bm{h}} 
 =2(n-\delta-1). 
\] 
Similarly, we can show that 
$q(G) \ge 2(n-\delta -1)$ for every $G\in  \mathcal{L}_{n,\delta}^{(1)}$.

\begin{theorem}[Li--Liu--Peng \cite{LLPLMA18}] \label{thmllp}
Assume that $\delta \ge 1$ and 
$n\ge \delta^4+\delta^3+4\delta^2 +\delta+6$. 
Let $G$ be a connected graph with $n$ vertices and minimum 
degree $\delta (G)\ge \delta$. If 
\[ q(G)\ge 2(n-\delta -1),  \]
then $G$ has a Hamilton cycle unless 
$G\in \mathcal{H}_{n,\delta}^{(1)}$ 
or $G\in \mathcal{L}_{n,\delta}^{(1)}$. 
\end{theorem}

In 2021, Zhou, Broersma, Wang and Lu \cite{Zhou2021} proved 
a slight improvement on Theorem \ref{thmllp}. 
Their proofs  are based on the Bondy--Chv\'{a}tal
closure, a degree sequence condition, 
and  the Kelmans transformation.

\subsection{Hamilton cycle in balanced bipartite graph}

Let $G$ be a bipartite graph with vertex sets $X$ and $Y$. 
The bipartite graph $G$ is called balanced if $|X|=|Y|$. 
We remark here that if  a bipartite graph has Hamilton cycle, 
then it must be balanced. 
So we  consider the existence of Hamilton cycle 
only in balanced bipartite graph.  

\begin{lemma}[{See \cite[p. 490]{Bondy08}}]
Let $G$ be a balanced bipartite graph on $2n$ vertices
 with degree sequence 
$d_1\le d_2 \le \cdots \le d_{2n}$. 
If $G$ is not Hamiltonian,    
then there exists an integer $i\le \frac{n}{2}$ 
such that $d_i \le i$ and  $d_n\le n-i$. 
\end{lemma}

Motivated by the work of Erd\H{o}s \cite{erdos62} 
in Theorem \ref{thmerdos62}, 
Moon and Moser \cite{Moon63} provided a corresponding result 
for the balanced bipartite graphs.

\begin{theorem}[Moon--Moser \cite{Moon63}] \label{thmmm}
Let $G$ be a balanced bipartite graph on $2n$ vertices.  
If the minimum degree $\delta (G)\ge \delta $ for some $1\le \delta \le n/2$ and 
\[  e(G) >  n(n-\delta) + \delta^2, \]
then $G$ has a Hamilton cycle. 
\end{theorem} 

The condition $ \delta \le {n}/{2}$ is well comprehensible 
since Moon and Moser \cite{Moon63} also pointed out that 
if $G$ is a balanced bipartite on $2n$ vertices with $\delta (G) > {n}/{2}$, 
then $G$ must be Hamiltonian. 
This is a bipartite version of the Dirac theorem.

Let  $B_{n,\delta}$ be the bipartite graph  
obtained from the complete bipartite graph $K_{n,n}$ 
by deleting all edges in its one subgraph $K_{\delta,n-\delta}$. 
More precisely, the two vertex parts of $B_{n,\delta}$ 
are $V=V_1\cup V_2$ and $U=U_1 \cup U_2$ 
where $|V_1|=|U_1|=\delta$ and $|V_2|=|U_2|=n-\delta$, 
we join all edges between $V_1$ and $U_1$, and all edges between 
$V_2$ and $U$. We denote 
\[ 
 \boxed{B_{n,\delta} := K_{n,n} \setminus K_{\delta ,n-\delta}. }  
\] 
It is easy to see that 
$e(B_{n,\delta})= n(n-\delta) + \delta^2$ and $B_{n,\delta}$ 
contains no Hamilton cycle. This implies that 
 the condition in Moon--Moser's theorem is best possible.

For a bipartite graph, Lu, Liu and Tian \cite{Lumei} gave a sufficient condition for a balanced bipartite graph being Hamiltonian 
in terms of the number of edges and the spectral radius of its quasi-complement. 
Let $K_{n,n-1} +e$ 
be the bipartite graph obtained from $K_{n,n-1}$ 
by adding a pendent edge to one of vertices in the part of size $n$. 

\begin{theorem}[Lu--Liu--Tian \cite{Lumei}] \label{thmllt}
Let $G$ be a balanced bipartite graph on $2n\ge 4$ vertices. 
If the minimum degree $\delta (G)\ge 1$ and 
\[  e(G)\ge n(n-1)+1, \]
 then $G$ is Hamiltonian unless 
$G=K_{n,n-1} +e$. 
\end{theorem}

Liu, Shiu and Xue \cite{Liuruifang}
gave sufficient conditions on the spectral 
radius for a bipartite graph being Hamiltonian. 
Their results  extended Theorem \ref{thmllt}. 
Moreover, they also provided  tight sufficient conditions 
on the signless Laplacian spectral radius for a graph to be 
Hamiltonian and traceable, which improve the results of 
Yu and Fan \cite{YuGuidong}.  
Let $K_{n,n-2}+4e$ be a bipartite graph obtained 
from $K_{n,n-2}$ by adding two vertices which are adjacent 
to two common vertices with degree $n-2$ in $K_{n,n-2}$.

\begin{theorem}[Liu--Shiu--Xue \cite{Liuruifang}]
Let $G$ be a balanced bipartite graph on $2n\ge 8$ vertices.
If  $\delta (G)\ge 2$ and 
 \[   \lambda (G)\ge \sqrt{n(n-2)+4}, \]
 then $G$ is Hamiltonian unless 
$G=K_{n,n-2} +4e$. 
\end{theorem}

For the Hamiltonicity of balanced bipartite graphs, 
Li and Ning \cite[Theorem 1.10]{LiBinlong} also proved the spectral version of 
the Moon--Moser Theorem \ref{thmmm}.  

\begin{theorem}[Li--Ning \cite{LiBinlong}]   \label{thmlnbi}
Suppose that $\delta\ge 1$ and $n\ge (\delta+1)^2$. 
Let $G$ be a $2n$-vertex balanced bipartite graph.  
If the minimum degree $\delta (G)\ge \delta$  
and 
\[ \lambda (G) \ge \lambda (B_{n,\delta}), \]
then either $G$ has a Hamilton cycle or $G=B_{n,\delta}$.  
\end{theorem}

\begin{theorem}[Li--Ning \cite{LiBinlong}]   \label{thmlnbiq}
Suppose that $\delta \ge 1$ and $n\ge (\delta+1)^2$. 
Let $G$ be a $2n$-vertex balanced bipartite graph.  
If the minimum degree $\delta (G)\ge \delta$  
and 
\[ q(G) \ge q (B_{n,\delta}), \]
then either $G$ has a Hamilton cycle or $G=B_{n,\delta}$.  
\end{theorem}

Recall that Theorem \ref{thmniki} extended slightly 
Theorem \ref{thmln16a}. 
For the case of bipartite graphs,  
we can see that $\lambda (B_{n,\delta}) 
>\lambda (K_{n,n-\delta}) =\sqrt{n (n-\delta)}$ 
since 
$K_{n,n-\delta}$ is a proper subgraph of $B_{n,\delta}$. Motivated by 
this observation, 
Ge and Ning \cite[Theorem 1.4]{GNLMA20} 
proved the following improvement on Theorem \ref{thmlnbi}.

\begin{theorem}[Ge--Ning \cite{GNLMA20}] \label{thm321}
Suppose that $\delta\ge 1$ and $n\ge \delta^3 +2\delta +4$. 
Let $G$ be a balanced bipartite graph on $2n$ vertices. 
If the minimum degree $\delta (G)\ge \delta$ and 
\[  \lambda (G)\ge \sqrt{n (n-\delta)}, \]
then either $G$ has a Hamilton cycle or $G=B_{n,\delta}$. 
\end{theorem}

We remark here that Jiang, Yu and Fang 
\cite[Theorem 1.2]{JYF2019} 
proved independently the result in Theorem \ref{thm321} 
under a weak condition 
$n\ge \max\{\frac{\delta^3}{2} + \frac{\delta^2}{2} + \delta +3, 
(\delta +2)^2\}$ with the same method only by some careful calculations. 
In 2020, Liu, Wu and Lai \cite{LWL2020} 
unified these several former spectral Hamiltonian results 
on balanced bipartite graphs and complementary graphs.  
In addition, Lu \cite{Lu2020} 
extended some spectral conditions for the Hamiltonicity
of balanced bipartite graphs. 

Correspondingly, 
Li, Liu and Peng \cite[Theorem 4]{LLPLMA18} 
also gave an improvement on Theorem \ref{thmlnbiq} 
by observing that $q (B_{n,{\delta} }) 
>q (K_{n,n-\delta}) =2n- \delta $. 
For the graph $B_{n,{\delta}}$, 
let $S$ and $T$ be the vertex sets 
such that the degree of vertices 
from $T$ is either $n$ or $n-{\delta}$. 
Let $X= \{v\in S: d(v)={\delta}\}, 
Y=\{v\in T : d(v)=n\}, 
W=\{v\in T: d(v)=n- {\delta}\}$ and 
$Z=\{v\in S: d(v)=n\}$. 
We can see from the definition that 
$S=X\cup Z$ and $T=Y\cup W$. 
We denote $E_1(B_{n,{\delta}})$ by 
those edges of $B_{n,{\delta}}$ whose 
both endpoints are from $Y\cup W\cup Z$. 
Let 
\[ \mathcal{B}_{n,{\delta}}^{(1)}= 
\left\{ B_{n,{\delta}} \setminus E' : 
E'\subseteq E_1(B_{n,{\delta}}) ~\text{with}~ 
|E'|\le \lfloor {{\delta}^2}/{4}\rfloor \right\}.  \]
Similarly we can show that 
$q(G)\ge q(K_{n,n-{\delta}} \cup I_{\delta})=2n- {\delta}$ 
for every $G\in \mathcal{B}_{n,{\delta}}^{(1)}$. 
Indeed, let $\bm{f}$ be 
an eigenvector of $K_{n,n-{\delta}} \cup I_{\delta}$ 
corresponding to $2n-{\delta}$. 
We can define the vector 
$\bm{f}$ as $f_u=1$ for vertex $u$ with $d(u)=n$, 
and $f_u=1-\frac{\delta}{n}$ for vertex 
$u$ with $d(u)=n-{\delta}$. 
Then 
\[  \bm{f}^TQ(G) \bm{f}  - 
\bm{f}^T Q(K_{n,n-{\delta}} \cup I_{\delta}) \bm{f} 
= \delta^2 - 4|E'|\ge 0. \]
The Rayleigh Formula gives 
\[  q(G) 
 \ge \frac{ \bm{f}^TQ(G)\bm{f} }{ \bm{f}^T\bm{f}  }  
 \ge \frac{\bm{f}^TQ({K_{n,n-\delta} \cup I_{\delta}}) \bm{f}}{\bm{f}^T\bm{f}} 
 =2n-\delta .\]

\begin{theorem}[Li--Liu--Peng \cite{LLPLMA18}]
Assume that $\delta \ge 1$ and 
$n\ge \delta^4+3\delta^3+5\delta^2 +5\delta+4$. 
Let $G$ be a connected graph with $n$ vertices and minimum 
degree $\delta (G)\ge \delta$. If 
\[ q(G)\ge 2n-\delta ,  \]
then $G$ has a Hamilton cycle 
unless $G\in \mathcal{B}_{n,\delta}^{(1)}$. 
\end{theorem}

\subsection{Problem for $k$-Hamiltonicity} 

A graph $G=(V,E)$ is \emph{$k$-Hamiltonian} if for all $X\subset V$ with $|X|\leq k$, the subgraph induced by $V\setminus X$ is Hamiltonian. Thus  $0$-Hamiltonian 
graph is the same as the general  Hamiltonian graph. 
Similarly, 
a graph $G$ is $k$-\textit{edge-Hamiltonian} if any collection of vertex-disjoint paths with at most $k$ edges altogether belong to a Hamiltonian cycle in $G$.
 In \cite{CKL1970}, 
 it is obtained that for a graph $G$,
 if $\delta(G)\geq \frac{n+k}2,$
  then $G$ is $k$-Hamiltonian. 
  Clearly, when $k=0$, it reduces to the Dirac theorem. 
  
  \begin{theorem}[Chartrand et al. \cite{CKL1970}, 
Bondy--Chv\'{a}tal \cite{Bondy}]  \label{thmCKL}
A graph $G$ is $k$-Hamiltonian 
if and only if the closure graph $\mathrm{cl}_{n+k} (G)$ is $k$-Hamiltonian. 
In particular, if $d(u)+d(v) \ge n+k$ for all non-edges $\{u,v\}$, 
then $\mathrm{cl}_{n+k} (G) =K_n$ and $G$ is $k$-Hamiltonian. 
\end{theorem}

  In 1972, 
Chv\'{a}tal \cite{Chvatal72} proved 
the following theorem, which 
characterizes the sufficient condition of degree sequence 
of $k$-Hamiltonian graphs.

\begin{theorem}[Chv\'{a}tal \cite{Chvatal72}]   \label{thm323}
Let $G$ be an $n$-vertex graph with degree sequence 
$d_1\le d_2 \le \cdots \le d_n$.   
Suppose $n\geq 3$, and $0\leq k\leq n-3$. If 
$G$ is not $k$-Hamiltonian, then there exists $1\leq i<\frac{n-k}{2}$ such 
that 
$d_i\leq i+k  $ and $ d_{n-i-k}\leq n-i-1$. 
\end{theorem}

 In \cite{FengLAA17}, Feng et al. obtained sufficient conditions    
 for a  graph  to be $k$-Hamiltonian. 

\begin{theorem}[Feng et al. \cite{FengLAA17}] \label{thmf42}
Let $k\ge 1$, and let $G$ be a graph of order $n\ge k+6$. 
If $e(G) \geq {n-1 \choose 2} + k+1$, then either $G$ is $k$-Hamiltonian or  $G=K_{k+1}\vee(K_1\cup K_{n-k-2})$.
\end{theorem}

In this survey, we shall provide an extension by giving minimum degree. 
When $\delta =k+1$, the following theorem reduces to Theorem \ref{thmf42}. 
For convenience, we denote by 
\[ 
  \boxed{H_{n,k,\delta } := K_\delta \vee ( K_{n-2\delta +k} \cup I_{\delta -k})   } 
\] 
 and 
\[ 
 \boxed{ L_{n,k,\delta}:=K_{k+1}\vee( K_{n-\delta-1}\cup K_{\delta-k}). }
\] 
It is easy to see that 
both $H_{n,k,\delta } $ and $L_{n,k,\delta}$ 
are not $k$-Hamiltonian and has minimum degree 
$\delta (H_{n,k,\delta })=\delta (L_{n,k,\delta})=\delta$. 
 In particular, when $k=0$, we can see that 
 $H_{n,0,\delta}$ is  the same as $H_{n,\delta}$, 
 and $L_{n,0,\delta}$ is the same as $L_{n,\delta}$, which are 
 defined in equations (\ref{eqhnd}) and (\ref{eqlnd}).  

\begin{theorem} \label{thmhnkd}
Let  $\delta \ge k+1\ge 1$ and $n\ge 6\delta -5k$. 
If $G$ is a graph on $n$ vertices with 
  minimum degree $\delta (G)\ge \delta $ and 
\begin{equation*}
 e(G)\ge e(H_{n,k,\delta }), 
\end{equation*}
then either $G$ is $k$-Hamiltonian 
or $G=H_{n,k,\delta }$. 
\end{theorem}

\begin{proof}
Suppose that $G$ is not $k$-Hamiltonian. 
According to~Theorem 
\ref{thmCKL}, we know that  
the closure graph $H:=\mathrm{cl}_{n+k}(G)$ is also not $k$-Hamiltonian. 
Let $d_1\le d_2\le \cdots \le d_n$ be the degree sequence of $H$.  
By Theorem~\ref{thm323}, 
there exists an integer $1\le m \le \frac{n-k-1}{2}$  
such that $d_m \le m+k$ and $d_{n-m-k} \le n-m-1$. 
Hence, we have 
\begin{align}
\notag 2e(H)=\sum\limits_{i=1}^n d_i 
&\le m(m+k)+ (n-2m-k)(n-m-1)+(m+k)(n-1) \\
&=3m^2 -(2n-2k-1)m +n^2-n.  \label{eq11}
\end{align}
First of all, we denote  
$f(m)=3m^2-(2n-2k-1)m$. Since $\delta \le \delta (G)\le \delta (H)\le m+k$,  
we can see that $\delta -k\le m \le \frac{n-k-1}{2}$. 
The quadratic function $f(m)$ is minimized at $m_0=\frac{2n-2k-1}{6}$.  
Note that $n\ge 6\delta -5k$ implies  $m_0-(\delta -k) > \frac{n-k-1}{2}-m_0$. 
Therefore,  we have $f(m)$ is maximized at $m=\delta -k$. Then 
\[ 2e(G)\le 2e(H) \le f(\delta -k) +n^2-n=
2e(K_{\delta} \vee ( K_{n-2\delta+k}\cup I_{\delta -k})).  \]
By the condition, we get 
$e(G)=e(K_{\delta} \vee ( K_{n-2\delta +k} \cup I_{\delta -k}))$. 
All the inequalities above become equalities, that is 
$m=\delta -k$ and $G=\mathrm{cl}_{n+k}(G)$, moreover 
the degrees of $G$ are $d_1=\cdots =d_{\delta -k}=\delta, 
d_{\delta -k+1}=\cdots =d_{n- \delta }=n-\delta +k-1$
 and $d_{n-\delta +1}=\cdots =d_n=n-1$. 
Next we shall prove that 
$G=K_{\delta} \vee ( K_{n-2\delta +k} \cup I_{\delta -k})$. 

Recall that $d_1,d_2,\ldots ,d_n$ are the degrees of $G$, 
and the corresponding vertices are $v_1,v_2, \ldots ,v_n$.  
We next denote by $X=\{v_1,v_2,\ldots ,v_{\delta -k}\}$ 
and $Y=\{v_{\delta -k+1},\ldots ,v_n\}$. 
We next show that  the induced subgraph $G[Y]=K_{n-\delta +k}$. 
Otherwise, we choose two non-adjacent vertices $v_r,v_s \in Y$. 
By noting that $G=\mathrm{cl}_{n+k}(G)$, then the  non-adjacency of $v_r$ and $v_s$ 
implies that  
$ d_r+d_s  <n+k. $ 
On the other hand, we know from the definition of $Y$ that  
$d_r+d_s \ge 2(n-\delta +k-1)>n+k$, a contradiction. 
 Thus we get $G[Y]=K_{n-\delta +k}$. 
 
 The previous result on degree sequence implies that 
there are $\delta$ vertices in $Y$ with degree $n-1$,  
we now denote these vertices by $F=\{v_{n-\delta +1},\ldots ,v_n\}$.  
Let $G[X,F]$ be the induced bipartite subgraph in $G$ 
between vertex sets $X$ and $F$. 
Finally, we shall show that $G[X,F]=K_{\delta -k,\delta}$. 
Note that $G=\mathrm{cl}_{n+k}(G)$, 
which implies that every two vertices with 
the degree sum at least $n+k$ are adjacent.  
Thus, every vertex in $X$ is adjacent to 
every vertex in $F$. 
So we can see that $G[X,F]=K_{\delta -k,\delta }$, 
and then $X$ is an independent set in $G$. 
Hence we get $G=K_{\delta} \vee ( K_{n-2\delta +k}\cup 
I_{\delta -k})$.  
\end{proof}

 Recently, by utilizing   the degree sequences and the closure concept,  
  Liu et al.  \cite{LiuDMGT} generalized Theorem \ref{thmniki} to  $k$-Hamiltonian graphs.
Their results can be considered as the spectral counterpart for the above Dirac-type condition.
We mention here that 
there is a tiny typo at the end of the proof in \cite[Theorem 4]{LiuDMGT}  
since the extremal graph is not the only one.  
Clearly, the graph $H_{n,k,\delta } 
= K_\delta \vee ( K_{n-2\delta +k} \cup I_{\delta -k})$ is not 
$k$-Hamiltonian and 
$\lambda (H_{n,k,\delta }) > n-\delta +k -1$. 
By a careful modification, the correct result should be the following. 

 \begin{theorem}[Liu et al. \cite{LiuDMGT}]  \label{thmain}
Let $k \geq 0$, $\delta\geq k+2$ and $n\geq n_0(k,\delta)$ where 
\begin{align*}
 n_0(k,\delta ) &=\max\{ 2\delta^2-2k\delta+2\delta-k+2, (\delta -k )(k^2+2k+5)+1 \}. 
 \end{align*}
If $G$ is a connected graph of order $n$ with minimum degree $\delta(G)\geq \delta$ and 
  $$\lambda(G)\geq n-\delta+k-1,$$
then $G$ is $k$-Hamiltonian unless $G=H_{n,k,\delta }$ 
or $G=L_{n,k,\delta}$.
\end{theorem}

Since  $\lambda(H_{n,k,\delta })$ contains $K_{n-\delta+k}$ as a subgraph, we have 
$\lambda(H_{n,k,\delta } )\geq n-\delta+k-1$. 
Note that $\lambda (H_{n,k,\delta }) > \lambda (L_{n,k,\delta})$. 
So we can immediately 
get the spectral version of Theorem \ref{thmhnkd} 
for graphs with sufficient large order.

 \begin{corollary}\label{thmaincol1}
Let $k \geq 0$, $\delta\geq k+2$ and $n\geq n_0(k,\delta)$. 
If $G$ is a connected graph of order $n$ with minimum degree $\delta(G)\geq \delta$ and 
  $$\lambda(G)\geq\lambda(H_{n,k,\delta } ) ,$$
then $G$ is $k$-Hamiltonian unless $G=H_{n,k,\delta }$. 
\end{corollary}

Recall that 
$H_{n,k,\delta } = K_\delta \vee ( K_{n-2\delta +k} \cup I_{\delta -k})$.  
Let $X$ be the set of $\delta -k$ vertices forming 
by the independent set $I_{\delta -k}$, 
$Y$ be the set of $\delta$ vertices  corresponding to 
 the clique $K_{\delta}$ 
and $Z$ be the set of the remaining 
$n-2\delta +k$ vertices consisting of the clique 
$K_{n-2\delta +k}$. 
This notation is clear although it is not standard. 
We write $E_1(H_{n,k,\delta })$ for the 
set of edges of $H_{n,k,\delta }$ 
whose both endpoints are from $Y\cup Z$. 
Furthermore, we define 
the family $\mathcal{H}_{n,k,\delta}^{(1)}$ of graphs 
as below. 
\[ \mathcal{H}_{n,k,\delta}^{(1)} = 
\left\{ H_{n,k,\delta} \setminus E': E' \subseteq E_1(H_{n,k,\delta})
 ~\text{with}~|E'|\le 
\lfloor {\delta (\delta -k)}/{4}\rfloor \right\}. \]
Here, we write $H_{n,k,\delta} \setminus E'$ for the graph obtained from 
$H_{n,k,\delta}$ by deleting all edges of the edge set $E'$.  
Similarly, for the graph $L_{n,k,\delta} 
=K_{k+1}\vee( K_{n-\delta-1}\cup K_{\delta-k})$,  
we denote  
$X$ by the set of vertices corresponding to 
the clique $K_{\delta -k}$, 
$Y$ by the set of vertices 
corresponding to the clique $K_{k+1}$, 
and $Z$ by the set of the 
remaining $n-\delta -1$ vertices. 
We write $E_1(L_{n,k,\delta })$ for the 
set of edges of $L_{n,k,\delta }$ 
whose both endpoints are from $Y\cup Z$. 
Moreover, we define 
the family $\mathcal{L}_{n,k,\delta}^{(1)}$ of graphs 
as follows. 
\[ \mathcal{L}_{n,k,\delta}^{(1)} = 
\left\{ L_{n,k,\delta} \setminus E': E' \subseteq E_1(L_{n,k,\delta})
 ~\text{with}~|E'|\le 
\lfloor {(k+1) (\delta -k)}/{4}\rfloor \right\}. \]

Recently, Li and Peng \cite{LLF2021} 
presented 
the following 
 signless Laplacian spectral conditions 
 for $k$-Hamiltonian graphs with large minimum 
 degree. 

\begin{theorem}[Li--Peng \cite{LLF2021}] 
Let $k\ge 0,\delta \ge k+2$ and 
$n$ be sufficiently large. 
If $G$ is an $n$-vertex graph  with minimum degree $\delta(G)\geq \delta$ and 
  $$q (G)\geq 2(n-\delta+k-1),$$
then $G$ is $k$-Hamiltonian 
unless $G\in \mathcal{H}_{n,k,\delta }^{(1)}$ 
or $G\in \mathcal{L}_{n,k,\delta}^{(1)}$.
\end{theorem}

 \begin{corollary} 
Let $k \geq 0$, $\delta\geq k+2$ and $n$ be sufficiently large. 
If $G$ is an $n$-vertex graph with minimum degree $\delta(G)\geq \delta$ and 
  $$q(G)\geq  q(H_{n,k,\delta } ) ,$$
then $G$ is $k$-Hamiltonian unless $G=H_{n,k,\delta }$. 
\end{corollary}

\subsection{Problem for $k$-edge-Hamiltonicity}  

Recall that a graph $G$ is $k$-\textit{edge-Hamiltonian} if any collection of vertex-disjoint paths with at most $k$ edges altogether belong to a Hamiltonian cycle in $G$.

  \begin{theorem}[Kronk \cite{Kro1969}, Bondy--Chv\'{a}tal \cite{Bondy}] 
A graph $G$ is $k$-edge-Hamiltonian 
if and only if the closure graph $\mathrm{cl}_{n+k} (G)$ is $k$-edge-Hamiltonian. 
In particular, if $d(u)+d(v) \ge n+k$ for all non-edges $\{u,v\}$, 
then $\mathrm{cl}_{n+k} (G) =K_n$ and $G$ is $k$-edge-Hamiltonian. 
\end{theorem}

\begin{lemma}[Kronk \cite{Kronk}]   \label{lemmakr}
Let $G$ be an $n$-vertex graph with degree sequence 
$d_1\le d_2 \le \cdots \le d_n$.   
 Suppose $n\geq 3$, and $0 \leq k \leq n-3$. 
If $G$ is  not $k$-edge-Hamiltonian, 
then there exists $k+1\leq i < \frac{n+k}{2}$ such that 
 $d_{i-k}\leq i $ and $ d_{n-i}\leq n-i+k-1$. 
\end{lemma}

\begin{theorem}[Feng et al. \cite{FengLAA17}] 
Let $k\ge 1$ and $G$ be a graph of order $n\geq k+6$. 
If 
\[ e(G) \geq  {n-1 \choose 2} + k+1, \]
then either $G$ is $k$-edge-Hamiltonian 
or  $G = K_{k+1} \vee ({K_{1}} \cup K_{n-k-2})$.
\end{theorem}

To avoid unnecessary calculations, 
we do not attempt to get the best bound 
on the order of graphs in the proof.

 \begin{theorem}[Li--Peng \cite{LLF2021}]   \label{thm21}
Let $k \geq 0$, $\delta\geq k+2$ and $n$ be sufficiently large.  
If $G$ is an $n$-vertex graph  with minimum degree $\delta(G)\geq \delta$ and 
  $$\lambda(G)\geq n-\delta+k-1,$$
then $G$ is $k$-edge-Hamiltonian unless $G=H_{n,k,\delta }$ 
or $G=L_{n,k,\delta}$.
\end{theorem}

Since  $\lambda(H_{n,k,\delta })$ contains $K_{n-\delta+k}$ as a proper subgraph, we have 
$\lambda(H_{n,k,\delta } )> n-\delta+k-1$. 
Moreover, applying the Kelmans operations on 
$L_{n,k,\delta}$, we can get a proper subgraph of $H_{n,k,\delta}$, 
this implies 
 $\lambda (H_{n,k,\delta }) > \lambda (L_{n,k,\delta})$; 
 see, e.g., \cite[Theorem 2.12]{LiBinlong}. 
 With this observation in mind, Theorem \ref{thm21} 
 implies the following corollary, which is an extension on Theorem \ref{thmln16a}.

 \begin{corollary} 
Let $k \geq 0$, $\delta\geq k+2$ and $n$ be sufficiently large. 
If $G$ is an $n$-vertex graph  with minimum degree $\delta(G)\geq \delta$ and 
  $$\lambda(G)\geq\lambda(H_{n,k,\delta } ) ,$$
then $G$ is $k$-edge-Hamiltonian unless $G=H_{n,k,\delta }$. 
\end{corollary}

Moreover, Li and Peng \cite{LLF2021} also 
presented 
the following sufficient 
 conditions on the signless Laplacian spectral radius 
 for $k$-Hamiltonian graphs with large minimum 
 degree. 

\begin{theorem}[Li--Peng \cite{LLF2021}]  \label{thm23}
Let $k\ge 0,\delta \ge k+2$ and 
$n$ be sufficiently large. 
If $G$ is an $n$-vertex graph  with minimum degree $\delta(G)\geq \delta$ and 
  $$q (G)\geq 2(n-\delta+k-1),$$
then $G$ is $k$-edge-Hamiltonian 
unless $G\in \mathcal{H}_{n,k,\delta }^{(1)}$ 
or $G\in \mathcal{L}_{n,k,\delta}^{(1)}$.
\end{theorem} 

As a consequence, we get the following corollary. 

 \begin{corollary} 
Let $k \geq 0$, $\delta\geq k+2$ and $n$ be sufficiently large. 
If $G$ is an $n$-vertex graph with minimum degree $\delta(G)\geq \delta$ and 
  $$q(G)\geq  q(H_{n,k,\delta } ) ,$$
then $G$ is $k$-edge-Hamiltonian unless $G=H_{n,k,\delta }$. 
\end{corollary}

\section{Spectral  problem for path coverability}

\subsection{Problem for Hamilton path} 

Let $G$ be a simple graph. 
A path is called a {\it Hamilton path} if 
it contains all vertices of $G$. 
When $G$ contains a Hamiltonian path,  
we usually  say that  $G$ is \textit{traceable}.  
It is well-known that 
 $G$ contains a Hamilton path if and only if 
$G\vee K_1$ contains a Hamilton cycle. 
With the help of this relation, almost all results involving Hamilton cycle 
can imply the corresponding results for Hamilton path. 
Thus, almost all problems involving Hamilton path can be reduced to 
the problem of Hamilton cycle in this way. 
For any non-negative integer $q$, 
a graph $G$ with $n\ge 3$ vertices is called $q$-traceable 
if any removal of at most $q$ vertices of $G$ results in a traceable graph.

\begin{theorem}[Dirac \cite{Dirac52}, Ore \cite{ore60}]
If $G$ is a graph on $n\ge 3$ vertices with minimum degree 
$\delta (G)\ge \frac{n-1}{2}$, then $G$ has a Hamilton path. 
\end{theorem}

  \begin{theorem}[Bondy--Chv\'{a}tal \cite{Bondy}] 
A graph $G$ is traceable  
if and only if the closure graph $\mathrm{cl}_{n-1} (G)$ is traceable. 
In particular, if $d(u)+d(v) \ge n-1$ for all non-edges $\{u,v\}$, 
then $\mathrm{cl}_{n-1} (G) =K_n$ and $G$ is traceable. 
\end{theorem}

\begin{lemma}[Chv\'{a}tal \cite{Chvatal72}]
Let $G$ be an $n$-vertex graph with degree sequence 
$d_1\le d_2 \le \cdots \le d_n$.    
If $G$  has no Hamilton path, 
then there exists $i< \frac{n+1}{2}$ such that 
$d_i \le i-1$ and $d_{n-i+1} \le n-i-1$. 
\end{lemma}

\begin{theorem}[Ore \cite{ore60}, Bondy \cite{Bondy72}] \label{thm43}
Let $G$ be a graph on $n\ge 3$ vertices. If 
\[ e(G)\ge {n-1 \choose 2}, \]
then  $G$ has a Hamilton path unless $G = K_1\cup K_{n-1}$ 
or $n=4$ and $G=K_{1,3}$.  
\end{theorem}

\begin{theorem}[Fiedler--Nikiforov \cite{FiedlerNikif}]
Let $G$ be a graph on $n$ vertices and 
\[  \lambda (G)\ge n-2, \]
then either $G$ has a Hamilton path or $G=K_{n-1} \cup K_1$. 
\end{theorem} 

The proof is short and similar with 
that in Theorem \ref{thmFN}. 

\begin{proof}
The well-known Stanley inequality asserts that 
$\lambda (G)\le -1/2 + \sqrt{2e(G) + 1/4}$, which together with $\lambda (G)\ge n-2$, implies that 
$e(G) \ge {n-1 \choose 2}$. 
Hence, 
for $n\neq 4$, 
the desired result follows from  Theorem \ref{thm43}. 
For $n=4$, if $G$ is the another possible exception   $K_{1,3}$,  
we can calculate that $\lambda (K_{1,3})=\sqrt{3}$, 
which contradicts the condition $\lambda (G)\ge n-2$. This completes the proof.  
\end{proof}

In 2013, Yu and Fan \cite{YuGuidong} provided the 
condition for the existence of Hamiltonian paths
 in terms of the signless Laplacian spectral radius of the graph. 

\begin{theorem}[Yu--Fan \cite{YuGuidong}]
Let $G$ be a graph on $n$ vertices and 
\[  q (G)\ge 2(n-2), \]
then  $G$ has a Hamilton path or $G=K_{n-1} \cup K_1$, or $n=4$ and $G=K_{1,3}$. 
\end{theorem}

\begin{proof}
An important inequality in \cite{FengPIMB09} states that 
$q(G)\le \frac{2e(G)}{n-1} + n-2$, 
which together with $q(G)\ge 2(n-2)$ yields that 
$e(G)\ge {n-1 \choose 2}$.  
Hence, 
for $n\neq 4$, 
the desired result follows from  Theorem \ref{thm43}. 
For $n=4$, when $G$ is the another possible exception   $K_{1,3}$,  
we can calculate that $q (K_{1,3})= 4$. 
This completes the proof.  
\end{proof}

Note that the extremal graph in Fielder--Nikiforov's result  
is $K_{n-1} \cup K_1$, which is not connected. 
Moreover, if a graph contains a Hamilton path, then it must be connected. 
In 2012, Lu, Liu and Tian \cite{Lumei} generalized this spectral condition 
 to connected graphs.

\begin{theorem}[Lu--Liu--Tian \cite{Lumei}] \label{thmllt}
Let $G$ be a  graph on $n\ge 7$ vertices. 
If $G$ is connected and 
\[ \lambda (G)\ge \sqrt{(n-3)^2 +2}, \]
then $G$ contains a Hamilton path. 
\end{theorem}

We remark here that the original statement 
of Theorem \ref{thmllt} (see \cite[Theorem 3.4]{Lumei}) 
requires the restriction $n\ge 5$, this is a typo 
because for $n=6$, we can verify that 
$\lambda (K_2 \vee I_4)=\frac{1+\sqrt{33}}{2} \approx 
3.372 > \sqrt{3^2+2}\approx 3.316$ and 
$K_2 \vee I_4$ does not contain Hamilton path. 
By a careful examination, the restriction should be modified to $n\ge 7$. 

\begin{theorem}[Ning--Ge \cite{NGLMA15}]
Let $G$ be a graph on $n\ge 4$ vertices with $\delta (G)\ge 1$. 
If 
\[ \lambda (G)> n-3,\] 
then $G$ has a Hamilton path or 
$G\in \{K_1\vee (K_{n-3}\cup I_2),K_2 \vee I_4,K_1\vee (K_{1,3} \cup K_1)  \}$. 
\end{theorem}

When the minimum degree is involved in contrast with the results in  \cite{FiedlerNikif}, using the adjacency spectral radius, Li and Ning \cite{LiBinlong}   obtained the following result.

\begin{theorem}[Li--Ning \cite{LiBinlong}]      \label{LiBinlong} 
Let $\delta \geq 0$ and $n \geq \max\{6\delta +10, (\delta^2+7\delta +8)/2\}$. 
If $G$ is a graph of order $n$ with minimum degree 
$\delta(G) \geq \delta$  and
$$\lambda(G) \geq \lambda(K_{\delta} \vee (K_{n-2\delta-1} \cup I_{\delta+1})),$$
then $G$ has a Hamilton path, unless $G=K_{\delta} \vee (K_{n-2\delta 
-1} \cup I_{\delta+1})$.
\end{theorem}

\begin{theorem}[Li--Ning \cite{LiBinlong}]   
Let $\delta \geq 0$ and $n \geq \max\{6\delta +10, (3\delta^2+9\delta +8)/2\}$. 
If $G$ is a graph of order $n$ with minimum degree 
$\delta(G) \geq \delta$  and
\[ q(G) \geq q (K_{\delta} \vee (K_{n-2\delta-1} \cup I_{\delta+1})),\]
then $G$ has a Hamilton path, unless $G=K_{\delta} \vee (K_{n-2\delta 
-1} \cup I_{\delta+1})$.
\end{theorem}

Theorem \ref{LiBinlong} was  generalized by Nikiforov 
\cite[Theorem 1.5]{Nikiforov} as below. 

\begin{theorem}[Nikiforov \cite{Nikiforov}] \label{Nikipath}
Let $\delta \geq 1 $ and $n \geq \delta^3 +\delta^2+2\delta + 5$.
If $G$ is a graph on $n$ vertices with 
minimum degree $\delta(G) \geq \delta$
and
\[  \lambda(G) \geq n-\delta-2, \]
 then $G$ has a Hamilton path unless 
 $G = K_{\delta} \vee (K_{n-2\delta-1} \cup I_{\delta+1})$ 
 or $G = K_{\delta+1}\cup K_{n-\delta-1} $.
\end{theorem}

The following version of the signless Laplacian 
spectral radius is a special case of 
Theorem \ref{thm418}, which will be 
introduced in next subsection.

\begin{theorem}[Cheng et al. \cite{Cheng2021}]
Let $\delta \geq 1 $ and $n \geq 
\delta^4 + 9\delta^3 +24\delta^2 +23 \delta +15$.
If $G$ is a connected graph on $n$ vertices with 
minimum degree $\delta(G) \geq \delta$
and
\[  q (G) \geq 2(n-\delta-2), \]
 then either $G$ has a Hamilton path or  
 $G$ is a subgraph of $K_{\delta} \vee (K_{n-2\delta-1} \cup I_{\delta+1})$ 
 by removing at most 
 $\lfloor {\delta (\delta +1)}/{4}\rfloor$ edges.
\end{theorem}

Moreover, comparing  
the Moon--Moser theorem with 
the Erd\H{o}s theorem,   
we can similarly consider the 
existence of Hamilton path in bipartite graphs. 
Note that if a bipartite graph has a Hamilton path, then it must be 
balanced and nearly balanced, i.e., 
the sizes of two vertex parts differs at most one.

In 2017, Li and Ning \cite{LNLAA17} established several spectral analogues of Moon and Moser's theorem on Hamilton paths in balanced bipartite graphs and nearly balanced bipartite graphs. 
We remark that one main ingredient of their proofs 
is a structural stability result  involving Hamilton paths in balanced bipartite graphs with given minimum degree and number of edges. 
To some extent, the line of proof is similar with that in \cite{LiBinlong}.

In 2020,  
Wei and You \cite[Theorem 1.5]{WY2020} further extended 
Nikiforov's result (Theorem \ref{Nikipath}) by improving the bound of order $n$ by almost a half. 
In addition, the authors \cite[Theorem 1.8]{WY2020}
also obtained spectral sufficient conditions for a graph to be traceable among graphs or balanced bipartite graphs with large minimum degree,  
which extended a result of Li and Ning 
for balanced bipartite graph being traceable. 
Futhermore, Liu, Wu and Lai \cite{LWL2020} 
unified these several former spectral Hamiltonian results 
on balanced bipartite graphs and complementary graphs.

\subsection{Problem for path coverability}

In this section, we  consider the $k$-path-coverable problem.
A graph
$G$ is {\it $k$-path-coverable} if 
$V(G)$ can be covered by $k$ or fewer vertex-disjoint paths. In particular, 
1-path-coverable is the same as  traceable. 
So the concept of $k$-path coverability 
is a natural extension of the Hamilton path. 
The disjoint path cover problem is strongly related to the well-known hamiltonian problem (one may refer to \cite{Lihao} for a survey), which is among the most fundamental
ones in graph theory, and   attracts much attention in theoretical computer science. However, this problem is NP-complete \cite{Steiner}, therefore, finding their guaranteed sufficient conditions becomes an interesting work. In \cite{LiJianping}, such Ore-type condition is obtained.

  \begin{theorem}[Bondy--Chv\'{a}tal \cite{Bondy}]  \label{thmpath}
A graph $G$ is $k$-path coverable 
if and only if the closure graph $\mathrm{cl}_{n-k} (G)$ is $k$-path coverable. 
In particular, if $d(u)+d(v) \ge n-k$ for all non-edges $\{u,v\}$, 
then $\mathrm{cl}_{n-k} (G) =K_n$ and $G$ is $k$-path coverable. 
\end{theorem}

\begin{theorem}    \cite{Bondy, Lesniak}  \label{lemmabl}
Let $G$ be an $n$-vertex graph with degree sequence 
$d_1\le d_2 \le \cdots \le d_n$. 
 Suppose   $k\geq 1$. 
 If $G$ is  not $k$-path-coverable, then 
 there exists $1\leq i < \frac{n-k}{2}$ such that 
  $d_{i+k}\leq i $ and $ d_{n-i}\leq n-i-k-1$. 
 \end{theorem}

 \begin{theorem}[Feng et al. \cite{FengLAA17}] \label{thmf59}
Let $k\ge 2$, and let $G$ be a graph of order $n\geq 5k+6$. If
\begin{equation*}
e(G) \geq {n-k-1 \choose 2} + k+1,\label{kpceq}
\end{equation*}
then either $G$ is $k$-path-coverable 
or $G = K_{1} \vee ( K_{n-k-2}\cup I_{k+1} )$.
\end{theorem}

For every integers $k\ge 1$ and $\delta \ge 1$,  we denote
\[ 
 \boxed{ {B}_{n,k,\delta}:=
  K_{\delta}\vee ( K_{n-2\delta -k} \cup {I_{\delta +k}}). }
\] 
Clearly, $ {B}_{n,k,\delta}$ 
is not $k$-path-coverable and 
$ e( {B}_{n,k,\delta})= {n-\delta -k \choose 2} + (\delta +k)\delta $.  
We next extend Theorem \ref{thmf59} by 
introducing the minimum degree. 
Clearly, when we set $\delta =1$, the following theorem reduces to Theorem \ref{thmf59}. 

\begin{theorem}
Let $k\ge 1$ and $G$ be a graph on $n\ge  5k+6\delta$ vertices.  
If the minimum degree $\delta (G)\ge \delta$ and 
\begin{equation*}
e(G) \ge e( {B}_{n,k,\delta}), 
\end{equation*}
then $G$ is $k$-path-coverable 
or $G= {B}_{n,k,\delta}$. 
\end{theorem}

\begin{proof}
Assume that $G$ is not $k$-path-coverable.  
By Theorem~\ref{thmpath}, we know that 
the closure graph $H:=\mathrm{cl}_{n-k}(G)$ is also not $k$-path-coverable. 
We denote by $d_1\le d_2 \le \cdots \le d_n$ 
the degree sequence of $H$.  
By Theorem~\ref{lemmabl}, 
there exists an integer $1\le m \le \frac{n-k-1}{2}$ such that 
$d_{m+k}\le m$ and $d_{n-m}\le n-m-k-1$. Thus we get 
\begin{align} 
2e(H)=\sum\limits_{i=1}^n d_i 
&\le (m+k)m +(n-2m-k)(n-m-k-1) +m (n-1) \notag \\
&=3m^2 -(2n-4k-1)m + (n-k)(n-k-1). \label{eq42}
\end{align}
We denote by $f(m):=3m^2 -(2n-4k-1)m$.  
Since 
 $\delta \le \delta (G) \le \delta (H) \le m \le \frac{n-k-1}{2}$. 
The function $f(m)$ is minimized at  $m_0=\frac{2n-4k-1}{6}$. 
By the condition $n\ge 6\delta +5k$, we have 
$m_0-\delta > \frac{n-k-1}{2}-m_0$. 
Hence, the function $f(m)$ is maximized at $m=\delta $, which yields 
\[ 2e(G)\le 2e(H)\le f(\delta )+(n-k)(n-k-1) =
2e(K_{\delta}\vee ( K_{n-2\delta -k}\cup I_{\delta +k})). \]
Combing the condition, 
we get $e(G)=e(K_{\delta}\vee ( K_{n-2\delta -k}\cup I_{\delta +k}))$, 
and all the above inequalities in (\ref{eq42}) become equalities. 
Hence, we have $m=\delta$ and $G=\mathrm{cl}_{n-k}(G)$, 
the degrees of $G$ are 
$d_1=\cdots =d_{\delta +k}=\delta,d_{\delta +k +1}=\cdots =d_{n-\delta }=n-\delta -k-1$ 
and $ d_{n-\delta +1}=\cdots =d_n=n-1$. Next, we shall prove that 
$G=K_{\delta}\vee ( K_{n-2\delta -k}\cup I_{\delta +k})$. 

Let $v_1,v_2,\ldots ,v_n $ be the vertices of $G$ 
corresponding to the degrees 
$d_1$, $d_2$, $\ldots$, $d_n$. 
We denote by $X:=\{v_1,v_2,\ldots ,v_{\delta +k}\}$ 
and $Y:=\{v_{\delta +k+1},\ldots ,v_n\}$. 
First of all, we show that 
the induced subgraph $G[Y]$ forms a clique. 
Otherwise, let $u$ and $v$ be two non-adjacent vertices in $Y$.  
From the definition of the $(n-k)$-closure graph, we know that 
$ d_G(u) +d_G(v) < n-k$.   
On the other hand, we have $d_G(u)+d_G(v)\ge 2(n-\delta -k -1)>n-k$, 
a contradiction.  
So we get $G[Y]=K_{n-\delta -k}$.  
Let $F$ be the set of vertices in $Y$ with degree $n-1$, that is, 
$F=\{ v_{n-\delta +1}, \cdots ,v_n\}$. 
We denote by $G[X,F]$  the induced bipartite subgraph of $G$ 
between vertex sets $X$ and $F$. 
Finally, we shall prove $G[X,F]=K_{\delta +k,\delta}$. 
Since each vertex of $F$ has degree $n-1$, 
so it adjacent to every vertex of $X$.  
Note that the vertex of $X$ has degree $\delta = |F|$, 
so $X$ is an independent set of $G$. 
Then $G[X,F]=K_{\delta +k,\delta}$. 
Hence, we get $G=K_{\delta}\vee ( K_{n-2\delta -k} \cup I_{\delta +k})$. 
\end{proof}

In 2020, by generalizing the results in Theorem  \ref{LiBinlong} and Theorem \ref{Nikipath}, Liu et al. \cite{LiuDMGT} 
obtained the following sufficient conditions   
by using the adjacency spectral radius.

 \begin{theorem}[Liu et al. \cite{LiuDMGT}] \label{thm13}
Let $k\ge 1,\delta \ge 2$ and 
$n \ge n_0(\delta,k)$ where 
\[ n_0(\delta,k)=\max \{\delta^2(\delta +k)+\delta +k+5, 5k+6\delta +6\}.  \] 
 If $G$ is a connected graph on $n$ vertices with minimum degree
$\delta (G)\ge \delta $ and 
\[ \lambda (G) \ge n-\delta -k -1, \]
then $G$ is $k$-path-coverable 
unless $G= {B}_{n,k,\delta}$.
\end{theorem}

Since   $K_{n-\delta -k}$ is a proper subgraph 
of ${B}_{n,k,\delta}$, we have 
\[ \lambda( {B}_{n,k,\delta} ) >
\lambda (K_{n-\delta -k}) = n-\delta-k-1. \]  
Hence, we  get  immediately the following 
corollary.

 \begin{corollary}[Liu et al. \cite{LiuDMGT}]  \label{thmaincol2}
Let $k \geq 1$, $\delta\geq 2$ and $n\geq n_0(\delta ,k)$. 
If $G$ is a connected graph of order $n$ with 
minimum degree $\delta(G)\geq \delta$ and 
  $$\lambda(G)\geq \lambda( {B}_{n,k,\delta} ) ,$$
then $G$ is $k$-path-coverable 
unless $G= {B}_{n,k,\delta}$.
\end{corollary}

In the sequel, we will consider the $Q$-index version  of Theorem \ref{thm13}. 
Recall that 
${B}_{n,k,\delta}=
  K_{\delta}\vee ( K_{n-2\delta -k} \cup {I_{\delta +k}})$. 
For the graph ${B}_{n,k,\delta}$,
 let $X=\{v \in V({B}_{n,k,\delta}):d(v)=k\}, 
Y=\{v \in V({B}_{n,k,\delta} ):d(v)=n-1 \}$,
 and $Z=\{v \in V({B}_{n,k,\delta}):d(v)=n-k-\delta -1 \}$. 
We denote by $E_1({B}_{n,k,\delta})$ the edge set of ${B}_{n,k,\delta}$ whose endpoints are both from $Y\cup Z.$ Moreover, we define
\begin{gather*}
 \mathcal{B}^{(1)}_{n,k,\delta}=\left\{ {B}_{n,k,\delta} 
 \setminus E': 
\mbox{$ E'\subseteq E_1({B}_{n,k,\delta})$ with $|E'|\leq
\left\lfloor { (\delta +k)\delta }/{4}\right\rfloor$}\right\}. 
\end{gather*}
Here, the symbol ${B}_{n,k,\delta}\setminus 
E'$ stands for the subgraph of ${B}_{n,k,\delta}$ 
by deleting all edges  from the edge set $E'$. 
Recently, Cheng, Feng, Li and Liu  \cite{Cheng2021} proved the following extension 
on the sufficient conditions of the existence of a Hamilton path.

\begin{theorem}[Cheng et al. \cite{Cheng2021}]  \label{thm418}
Let $k\ge 1,\delta \ge 2$ and $n \ge n_1(\delta ,k)$ where 
\[ n_1(\delta ,k)= 
(\delta^2 +k\delta +7\delta +6k +9)(\delta^2 +k\delta +1) \] 
If $G$ is a connected graph on $n$ vertices and minimum degree
$\delta (G)\ge \delta $ such that
\[ q(G) \ge 2(n-\delta -k -1), \]
then either $G$ is $k$-path-coverable 
or  $G\in \mathcal{B}^{(1)}_{n,k,\delta}$.
\end{theorem}

 \begin{corollary}
Let $k \geq 1$, $\delta \geq 2$ and $n\geq n_1(\delta ,k)$. 
If $G$ is a connected graph of order $n$ with 
minimum degree $\delta(G)\geq \delta$ and 
  $$q(G)\geq q( {B}_{n,k,\delta} ) ,$$
then $G$ is $k$-path-coverable 
unless $G= {B}_{n,k,\delta}$.
\end{corollary}

\section{Spectral  problem for connectivity}

\subsection{Problem for $k$-connectivity}

A connected graph $G$ is said to be \emph{$k$-connected} (or \emph{$k$-vertex-connected}) if it has more than $k$ vertices and remains connected whenever fewer than $k$ vertices are deleted. 
In particular, $1$-connected graphs are the general  connected graphs. 
The connectivity of $G$, written as $\kappa(G)$,  
is defined as the maximum integer $k$ such that $G$ is $k$-connected. 
In other words, the connectivity 
of a graph $G$ is the minimum number of vertices that we need to delete 
to make $G$ disconnected. 
From the definition, we can see that $\kappa (G)\ge k$ 
is equivalent to say that $G$ is $k$-connected. 

When we talk about the connectivity and the eigenvalues, perhaps the most famous one is
due to Fiedler \cite{Fiedler73} which states that the second smallest Laplacian eigenvalue is at most the connectivity for any non-complete graph. 
For adjacency eigenvalues, the relation between the connectivity, edge-connectivity and the eigenvalues was reported in \cite{CioabaGuCzech, GuxiaofengJGT, 
ZhouWang18, Zhouqianna20c}.

Let $G$ be a simple graph of order $n\geq k+1$. It is known \cite[Page 4]{Bollobas78} that  if
$
\delta (G) \geq \frac12(n+k-2),
$
 then $G$ is  $k$-connected.
In  1969, Bondy  \cite{BondyConnectivity} 
provided a sufficient condition on the degree sequence. 

  \begin{theorem}[Bondy--Chv\'{a}tal \cite{Bondy}]  
A graph $G$ is $k$-connected  
if and only if the closure graph $\mathrm{cl}_{n+k-2} 
(G)$ is $k$-connected. 
In particular, if $d(u)+d(v) \ge n+k-2$ for all non-edges $\{u,v\}$, 
then $\mathrm{cl}_{n+k-2} (G) =K_n$ and $G$ is $k$-connected. 
\end{theorem}

\begin{lemma}[Bondy \cite{BondyConnectivity}]
 Suppose that $n\ge k+1$ and  $G$ is an $n$-vertex graph with degree sequence 
$d_1\le d_2 \le \cdots \le d_n$. 
If $G$ is not $k$-connected, 
then there exists $1\leq i\le\frac{n-k+1}{2}$ such that 
$d_i\leq i+k-2 $ and $d_{n-k+1}\leq n-i-1$.
\end{lemma}

Sometimes we can characterize the vertex-connectivity
 using the information
about co-degrees of vertices in our graph. Such result was used in [61]
to determine the vertex-connectivity of dense random $d$-regular graphs. 
 
 \begin{theorem}[Krivelevich et al. \cite{KSVW01}]
 Let $G=(V,E)$ be a $d$-regular graph on $n$ vertices such that 
 $\sqrt{n}\log n < d \le 3n/4$ and the number of common neighbors for 
 every two distinct vertices in $G$ is $(1+o(1))d^2/n$. 
 Then the graph $G$ is $d$-vertex-connected. 
 \end{theorem}

\begin{theorem}[Wu--Zhang--Feng \cite{Zhangpengli19}]
Let $G$ be a connected graph on $n\ge 5$ vertices. 
If the minimum degree $\delta (G)\ge 2$ and 
\[  e(G)\ge {n-2 \choose 2} +3, \]
then $G$ is $2$-connected unless $G=K_1\vee (K_{n-3} \cup K_2)$. 
\end{theorem}

\begin{theorem}[Feng et al. \cite{FengLAA17}] \label{thmfeng42}
Let $G$ be a graph of order $n\geq k+1\ge 2$. 
If 
\[  e(G) \geq {n-1 \choose 2}+k-1, \]
then either $G$ is $k$-connected or $G = K_{k-1} \vee (K_1 \cup K_{n-k})$.
\end{theorem}

The above result was extended by 
Feng et al. \cite[Theorem 3.2]{FengMonoshMath}. 
Clearly, when we set $\delta =k-1$, then Theorem \ref{thmfmo} 
reduces to Theorem \ref{thmfeng42}.  
For convenience, we denote 
\[ 
 \boxed{ A_{n,k,\delta} :=K_{k-1}\vee(K_{\delta-k+2}\cup K_{n-\delta-1}). } 
\] 
It is easy to see that $A_{n,k,\delta}$ is not $k$-connected and 
$ e(A_{n,k,\delta})= 
\frac{1}{2}n^2 - (\delta -k +\tfrac{5}{2})n + (\delta +1)(\delta -k+2)$. 
In 2017, Feng et al. \cite[Theorem 3.2]{FengMonoshMath} 
proved that  $A_{n,k,\delta}$ is the unique graph 
that is not $k$-connected and 
has the maximum number of edges.

\begin{theorem}[Feng et al. \cite{FengMonoshMath}]  \label{thmfmo}
Let $G$ be a connected graph on $n$ vertices with minimum degree 
$\delta (G)\ge \delta$. If 
\[  e(G)\ge e(A_{n,k,\delta}),  \]
then either $G$ is $k$-connected or 
$G=A_{n,k,\delta}$. 
\end{theorem}

 From the spectral point of view, we naturally ask the following question: can one find a sufficient spectral condition for a connected graph to be $k$-connected?  
  For pseudo-random graphs, Krivelevich and Sudakov 
  \cite[Theorem 4.1]{KS06} provided a sufficient condition on 
  the second largest eigenvalue in absolute value for regular graphs. 

\begin{theorem} \cite{KS06} 
Let $G$ be a $d$-regular graph on $n$ vertices with $d\le n/2$. 
We denote $\mu = \max\{ |\lambda_i| : i\in [n], \lambda_i \neq d \}$. Then the connectivity of $G$ satisfies 
$  \kappa (G) \ge d- {36\mu^2}/{d}$. 
\end{theorem}

In 2019, Wu, Zhang and Feng  \cite{Zhangpengli19} presented 
the sufficient spectral condition of a connected graph 
to be $2$-connected with small minimum degree. 

\begin{theorem}[Wu--Zhang--Feng \cite{Zhangpengli19}]
Let $G$ be a connected graph on $n\ge 5$ vertices. 
If the minimum degree $\delta (G)\ge 2$ and 
\[  \lambda (G) \ge \sqrt{(n-3)^2+4}, \]
then $G$ is $2$-connected. 
\end{theorem}

\begin{theorem}[Feng et al. \cite{FengMonoshMath}] \label{th31}
Let $\delta\geq k\geq 3$ and 
$n\geq n_0(k,\delta)$ where 
$$n_0(k,\delta)=(\delta-k+2)(k^2-2k+4)+3.$$
If $G$ is a connected graph of order $n$ with 
minimum degree $\delta(G)\geq \delta$ and 
  $$\lambda(G)\geq n-\delta+k-3,$$
then $G$ is $k$-connected 
unless $G=A_{n,k,\delta}$. 
\end{theorem}

Since $\lambda(A_{n,k,\delta})\geq n-\delta+k-3,$ we immediately have

\begin{corollary}[Feng et al. \cite{FengMonoshMath}]  
\label{corollary31}
Let $\delta\geq k\geq 3$ and $n\geq n_0(k,\delta)$.
Let $G$ be a connected graph of order $n$ and minimum degree $\delta(G)\geq \delta$.
    If
  $$\lambda(G)\geq \lambda(A_{n,k,\delta}),$$
then $G$ is $k$-connected unless $G=A_{n,k,\delta}$. 
\end{corollary}

Recall that 
${A}_{n,k,\delta}
=K_{k-1}\vee(K_{\delta-k+2}\cup  {K_{n-\delta-1}})$. 
For the graph ${A}_{n,k,\delta}$, let
 $X=\{v \in V({A}_{n,k,\delta}):d(v)=\delta\}$,  
 $Y=\{v \in V({A}_{n,k,\delta}):d(v)=n-1 \}$ 
 and  $Z=\{v \in V({A}_{n,k,\delta}):d(v)=n-\delta+k-3\}$. 
Let $E_1({A}_{n,k,\delta})$ denote the edge set of ${A}_{n,k,\delta}$ whose endpoints are both from $Y\cup Z.$
 We define
$$ \mathcal{A}^{(1)}_{n,k,\delta}=\left\{  {A}_{n,k,\delta} 
\setminus E': \mbox{$ E'\subseteq E_1({A}_{n,k,\delta})$ with $|E'|\leq \left\lfloor { (\delta-k+2)(k-1)}/{4}\right\rfloor$}\right\}.$$
Here, we denote by $A_{n,k,\delta} \setminus E'$ the graph obtained from 
$A_{n,k,\delta}$ by deleting all edges of the edge set $E'$.  
Recently, it was proved in \cite{Zhang2017} that 

\begin{theorem}[Zhang et al. \cite{Zhang2017}] \label{th31}
Let $\delta\geq k\geq 3$ and $n\geq n_1(k,\delta)$ 
where 
$$n_1(k,\delta)=  (k^2+2k-3)\delta^2- (2k^3-k^2-17k+8)\delta+ (k^4-3k^3-8k^2+23k+4).$$ 
If $G$ is a connected graph of order $n$ with 
minimum degree $\delta(G)\geq \delta$ and 
  $$q (G)\geq 2(n-\delta +k-3),$$
then $G$ is $k$-connected 
unless $G\in \mathcal{A}^{(1)}_{n,k,\delta}$. 
\end{theorem}

\begin{corollary}
Let $\delta\geq k\geq 3$ and $n\geq n_1(k,\delta)$.
Let $G$ be a connected graph of order $n$ and minimum degree $\delta(G)\geq \delta$.
    If
  $$q (G)\geq q(A_{n,k,\delta}),$$
then $G$ is $k$-connected unless $G=A_{n,k,\delta}$. 
\end{corollary}

\subsection{Problem for  $k$-edge-connectivity}

Similarly to vertex-connectivity, we now define the edge-connectivity of a graph. 
A simple graph $G$ is \emph{$k$-edge-connected} if it has at least two vertices and remains connected whenever fewer than $k$ edges are deleted. 
The edge-connectivity of $G$, written as $\kappa' (G)$, 
is defined as the maximum $k$ such that 
$G$ is $k$-edge-connected. 
In particular, $1$-edge-connected graphs are  the general connected graphs. 
Clearly, it is well-known that $\kappa (G) \le \kappa' (G)$. 
Moreover, the edge-connectivity 
is always at most the minimum degree of a graph.

  \begin{theorem}[Bondy--Chv\'{a}tal \cite{Bondy}]  
A graph $G$ is $k$-edge-connected  
if and only if the closure graph $\mathrm{cl}_{n+k-2} 
(G)$ is $k$-edge-connected. 
In particular, if $d(u)+d(v) \ge n+k-2$ for all non-edges $\{u,v\}$, 
then $\mathrm{cl}_{n+k-2} (G) =K_n$ and $G$ is $k$-edge-connected. 
\end{theorem}

\begin{lemma}[Bauer et al. \cite{BauerEdgeCon}]  \label{lem512}
Let $G$ be an $n$-vertex graph with degree sequence 
$d_1\le d_2 \le \cdots \le d_n$. 
 Suppose $n\ge k+1$, and $d_{1}\geq k\ge 1$.  
If $G$ is not $k$-edge-connected, 
then there exists $k+1\leq i\leq \lfloor {n}/{2}\rfloor$ such that 
 $d_{i-k+1}\leq i-1$, $d_{i}\leq i+k-2 $ and 
 $  d_{n}\leq n-i+k-2$.
In particular, if $k\ge\lfloor {n}/{2}\rfloor$, then $G$ is  $k$-edge-connected.
\end{lemma}

For adjacency eigenvalues, the relation between the connectivity, edge-connectivity and the eigenvalues was reported in \cite{Cioaba10LAA, CioabaGuCzech, GuxiaofengJGT, LiuhuiqingLumei, OSuilSebiSIAM}. 

  First of all, we introduce some results that only deal with regular graphs. 
  
  \begin{theorem}[{Krivelevich--Sudakov \cite{KS06}}]
Let $G$ be a $d$-regular graph on $n$ vertices.   If 
$\lambda_2 (G)\le d-2$, then $\kappa' (G)=d$. 
  \end{theorem}

  Extending and improving the  
 the previous results, Cioab\u{a} \cite{Cioaba10LAA} obtained that
 \begin{theorem}[Cioab\u{a} \cite{Cioaba10LAA}] \label{Cioaba10LAA}
 Let $d\geq k\geq 2$. If the second largest eigenvalue $\lambda_2$ of a $d$-regular graph satisfies
$$
\lambda_2<d -\frac{ (k-1)n}{(d+1)(n-d-1)},
$$
then the  edge-connectivity of $G$ is at least $k$.
 \end{theorem}

Now, we consider the $k$-edge-connected property. 
Since every $k$-connected graph is also $k$-edge-connected. 
So  \autoref{thmfeng42} also provided a sufficient condition 
for a graph to be $k$-edge-connected. 
Next, we shall present a sufficient condition with a smaller lower bound, 
which is slightly different from  Theorem \ref{thmfeng42}.

\begin{theorem}[Feng et al. \cite{FengLAA17}]
Let $G$ be a graph of order $n\ge k+1$. If $k\ge 2$ and $e(G) \geq \frac12(n^2-(k+4)n+2k^2+2k+4)$, then $G$ is $k$-edge-connected.
\end{theorem} 

\begin{proof}
Suppose that $G$ is not $k$-edge-connected. Then from Lemma \ref{lem512}, there exists an integer $k+1\leq i\leq \frac{n}{2}$ such that $d_{i-k+1}\leq i-1$,  $d_{i}\leq i+k-2$ and $d_{n}\leq n-i+k-2$. In particular, note that $n\ge 2k+2$. We have
\begin{align*}
2e(G) &\le  (i-k+1)(i-1)+(k-1)(i+k-2)+(n-i)(n-i+k-2) \\
&= n(n-1)+(k-1)(n+k-1)+2i^2-(2n+k-1)i.
\end{align*}

Suppose $f(x)=2x^2-(2n+k-1)x$ with $k+1\leq x \leq \frac{n}{2}$. Then it is easy to check that $f(k+1)-f(\frac{n}{2})=\frac{(n-k-3)(n-2k-2)}{2}\ge 0$, with equality if and only if $n=2k+2$. It follows that $f_{\max}(x)=f(k+1)=-(k+1)(2n-k-3)$, and we have
\begin{align*}
2e(G) &\le n(n-1)+(k-1)(n+k-1)-(k+1)(2n-k-3) \\
&= n^2-(k+4)n+2k^2+2k+4.
\end{align*}

Thus $e(G) = \frac{1}{2}(n^2-(k+4)n+2k^2+2k+4)$, and we have $i=k+1$, and 
the degree sequence is given as 
$d_1=d_2=k$, $d_3=\cdots=d_{k+1}=2k-1$ and $d_{k+2}=\cdots=d_n=n-3$. 

Next, we show that if $G$ has this degree sequence, then it must be $k$-edge-connected, which will be a contradiction. It suffices to show that for any partition $V(G)=A\cup B$, where $|A|+|B|=n$ and $1\le |A|\le |B|$, we have $e(A,B)\ge k$.

Since $G$ has minimum degree $\delta(G)=k$, if $|A|=1$, then $e(A,B)\ge k$, and if $|A|=2$, then $e(A,B)\ge 2(k-1)\ge k$.
%
%

Now, let $|A|\ge 3$. Note that $n\ge 2k+2$ and $|A|\le |B|$ imply $|B|\ge k+1$. Firstly, if $|B|=k+1$, then we also have $|A|=k+1$. Thus by the symmetry of $A$ and $B$, we may assume that $A$ has $k$ vertices with degree at least $2k-1$ in $G$. Then, each such vertex in $A$ has at least $(2k-1)-k\ge 1$ neighbour in $B$, and hence $e(A,B)\ge k$.

Secondly, suppose that $|B|\ge k+2$. If $A$ contains a vertex $v$ with $d_G(v)=n-3$, then $v$ has at least $|B|-2\ge k$ neighbours in $B$, and $e(A,B)\ge k$. Otherwise, all vertices with degree $n-3$ in $G$ lie in $B$. Since $|A|\ge 3$, each such vertex in $B$ has at least $|A|-2\ge 1$ neighbour in $A$. Since there are $n-(k+1)\ge (2k+2)-(k+1)>k$ such vertices, it follows again that $e(A,B)\ge k$. 
This completes the proof.
\end{proof}

\subsection{Problem for Hamilton-connectivity}  

A graph is Hamilton-connected 
if for every pair of vertices $u, v$, there is a Hamilton
path from $u$ to $v$.  
It is well-known \cite[p. 474]{Bondy08} that 
a graph $G$ is traceable 
from every vertex if and only if $G\vee K_1$ is Hamilton-connected.

\begin{theorem}[Dirac \cite{Dirac52}, Ore \cite{ore60}] 
If $G$ is a graph on $n\ge 3$ vertices with minimum degree 
$\delta (G)\ge \frac{n+1}{2}$, then $G$ is Hamilton-connected. 
\end{theorem}

\begin{theorem}[Bondy--Chvatal~\cite{Bondy}] \label{lem4444}
Let $G$ be a simple graph on $n$ vertices.  
Then $G$ is Hamilton-connected if and only if 
the closure graph $\mathrm{cl}_{n+1}(G)$ is  Hamilton-connected. 
In particular, if $d(u)+d(v) \ge n+1$ for all non-edges $\{u,v\}$, 
then $\mathrm{cl}_{n+1} (G) =K_n$ and $G$ is Hamilton-connected. 
\end{theorem}

\begin{theorem}[Berge~\cite{Berge1976}] \label{bergethm1976}
Let $G$ be an $n$-vertex graph with degree sequence 
$d_1\le d_2 \le \cdots \le d_n$.  
If $G$ is not Hamilton-connected, 
then there exists an integer $2\le i\le \frac{n}{2}$ such that 
$d_{i-1}\le i$ and $d_{n-i} \le n-i$. 
\end{theorem} 

The following sufficient condition on the number of edges is due to Ore \cite{Ore1963}.

\begin{theorem}[Ore \cite{Ore1963}]
Let $G$ be a graph on $n$ vertices. If 
\[ e(G)\ge {n-1 \choose 2} +3, \]
 then 
$G$ is Hamilton-connected. 
\end{theorem}

In 2017, 
Zhou and Wang \cite{ZW2017} gave some sufficient conditions on the number of edges, the spectral radius and the signless Laplacian spectral radius for a graph to be Hamilton-connected. 
And they also studied sufficient conditions on the number of edges, the spectral radius and the signless Laplacian spectral radius for a graph to be traceable from every vertex. 

\begin{theorem}[Zhou--Wang \cite{ZW2017}] \label{thm2017a}
Let $G$ be a connected graph on $n\ge 6$ vertices with minimum degree $\delta (G)\ge 3$. If 
\[ e(G)\ge {n-2 \choose 2} +6,\]
 then $G$ is Hamilton-connected 
unless $G=K_3 \vee (K_{n-5} \cup I_2)$ and 
the following more exceptions: when $n=8$ and $G\in \{K_4 \vee I_4, K_3 \vee (K_{1,3} \cup K_1), 
K_3\vee (K_{1,2} \cup K_2), K_2 \vee K_{2,4}\}$; 
when $n=9$ and $G=K_4\vee (K_2 \cup I_3)$; 
when $n=10$ and $G\in \{K_5 \vee I_5, K_4 \vee (K_{1,4}\cup K_1), 
K_3 \vee K_{2,5}\}$; 
when $n=11$ and $G=K_4 \vee (K_{1,3} \cup K_2)$;  
 when $n=12$ and $G=K_6 \vee I_6$. 
\end{theorem}

\begin{theorem}[Zhou--Wang \cite{ZW2017}] \label{thm2017b}
Let $G$ be a connected graph on $n\ge 6$ vertices with 
minimum degree $\delta (G)\ge 3$. If 
\[  \lambda (G) \ge \sqrt{(n-3)^2 +10}, \]  
then $G$ is Hamilton-connected. 
\end{theorem}

\begin{theorem}[Zhou--Wang \cite{ZW2017}] \label{thm2017c}
Let $G$ be a connected graph on $n\ge 6$ vertices with 
minimum degree $\delta (G)\ge 3$. If 
\[ q (G) \ge 2(n-3) + \frac{14}{n-1}, \]  
then $G$ is Hamilton-connected 
unless $n=8$ and $G=K_4\vee I_4$. 
\end{theorem}

In 2020, Zhou, Wang and Lu \cite{Zhouqianna20b}   
provided a further improvement on Theorems 
\ref{thm2017a}, \ref{thm2017b} and 
\ref{thm2017c}. More precisely, 
they proved that 
if $n\ge 11$ and $G$ satisfies $\delta (G)\ge 3$ 
and $e(G)\ge {n-3 \choose 2} +13$, then $G$ is Hamilton-connected unless $G$ belongs to some exceptions.  

\begin{theorem}[Zhou--Wang--Lu \cite{Zhouqianna20b}]
Let $G$ be a connected graph on $n \ge 14$ vertices with minimum degree $\delta \ge 3$. If 
\[ \lambda (G) > n-3,\]
 then $G$ is Hamilton-connected 
unless $G=K_{3 } \vee ( K_{n-5}\cup {I_{2}} ) $ 
or $G=K_2 \vee (K_{n-4} \cup K_{2})$. 
\end{theorem}

\begin{theorem}[Zhou--Wang--Lu \cite{Zhouqianna20b}]
Let $G$ be a connected graph on $n \ge 13$ vertices with minimum degree $\delta \ge 3$. If 
\[ q(G) > 2(n-3) + \frac{6}{n-1}, \]
then $G$ is Hamilton-connected 
unless $G=K_{3 } \vee ( K_{n-5}\cup {I_{2}} )$.  
\end{theorem}

For notational convenience, we denote 
\[ 
\boxed{S_{n,\delta} := 
K_{\delta } \vee ( K_{n-2\delta +1}\cup {I_{\delta -1}} ) }
\] 
and 
\[ 
\boxed{  T_{n,\delta} := 
K_2 \vee (K_{n-\delta -1} \cup K_{\delta -1}). }
\] 
Clearly, we can verify that both 
$S_{n,\delta}$ and $T_{n,\delta}$ are not Hamilton-connected 
and $\delta (S_{n,\delta}) =\delta (T_{n,\delta})=\delta$.  
Moreover, we have 
$  e(S_{n,\delta}) = {n-\delta +1 \choose 
2} + \delta (\delta -1) $
and 
$ e(T_{n,\delta }) = {n-\delta +1 \choose 2} + {\delta -1 \choose 2} 
+2(\delta -1)$. 
Trivially, we have 
$S_{n,2}=T_{n,2}$ and 
$e(S_{n,\delta}) > e(T_{n,\delta})$ for every $\delta \ge 3$. 

\begin{theorem}
Let $\delta \ge 2$ 
and $G$ be a graph on $n\ge 6\delta -5$ vertices. 
If $\delta (G)\ge \delta $ and  
\begin{equation*}
e(G) \ge e(S_{n,\delta}), 
\end{equation*}
then either $G$ is Hamilton-connected or  
$G=S_{n,\delta}$. 
\end{theorem}

\begin{proof}
Assume that $G$ is not Hamilton-connected. 
By Theorem \ref{lem4444}, 
we know that  the closure graph 
$H:=\mathrm{cl}_{n+1}(G)$ is also not Hamilton-connected. 
Let $d_1'\le d_2' \le \cdots \le d_n'$ be the degree sequence 
of $H$.  
By Theorem \ref{bergethm1976}, 
there is an integer $2\le m\le \frac{n}{2}$ such that 
$d_{m-1}\le m$ and $d_{n-m}\le n-m$. 
Hence, we can get 
\begin{equation} \label{eq14}
\begin{aligned}
2e(H)=\sum\limits_{i=1}^n d_i' 
&\le (n-1)m +(n-2m+1)(n-m) +m(n-1) \\
&=3m^2- (2n+3)m +n^2+n. 
\end{aligned}
\end{equation}
Let $f(m)=3m^2-(2n+3)m$ be a function on $m$.   
Note that $\delta \le \delta (G)\le \delta (H)\le m\le \frac{n}{2}$, 
and the symmetric line of $f(m)$ is $m_0=\frac{2n+3}{6}$. 
The condition $n\ge 6\delta -5$ 
implies $\frac{n}{2}-\frac{2n+3}{6}< \frac{2n+3}{6}-\delta $. 
So the maximum value of $f(m)$ 
is attained at $m=\delta $, therefore 
\[ 2e(G)\le 2e(H) \le f(\delta ) +n^2+n 
=2e(S_{n,\delta}), \]
where $S_{n,\delta} = K_{\delta } \vee (I_{\delta -1} \cup K_{n-2\delta +1})$. 
The condition in the theorem implies that $e(G)=e(H)
=e(S_{n,\delta})$, 
and all inequalities in (\ref{eq14}) must be equalities, 
that is, $m=\delta$ and $G=\mathrm{cl}_{n+1}(G)$, 
and hence the degree sequence of $G$ 
is $d_1=\cdots =d_{\delta -1}=\delta$, 
$d_{\delta}=\cdots =d_{n-\delta }=n-\delta $, 
 and $d_{n-\delta +1}=\cdots =d_n=n-1$. 
Next we shall show that $G=S_{n,\delta}$.  

Let $v_1,v_2,\ldots ,v_n$ be vertices of $G$ corresponding to the degrees $d_1,d_2,\ldots ,d_n$. 
We denote by $X=\{v_1,v_2,\ldots ,v_{\delta -1}\}$ and 
$Y=\{v_{\delta},v_{\delta +1},\ldots ,v_n\}$. 
First of all, we shall prove that the induced subgraph 
$G[Y]=K_{n-\delta +1}$. Otherwise, 
there are two non-adjacent vertices $v_r,v_s\in G[Y]$.  
By the definition of the $(n+1)$-closure graph, we know that 
\[ d_G(v_r)+d_G(v_s) =d_H(v_r)+d_H(v_s) <n+1. \]
On the other hand, 
we have $d_G(v_r)+d_G(v_s)\ge 2(n- \delta )>n+1$, a contradiction. 
Thus, the induced subgraph $G[Y]=K_{n-\delta +1}$.  
Note that $Y$ contains $\delta$ vertices with degree $n-1$, 
we denote these vertices by 
$F=\{v_{n-\delta +1},\ldots ,v_n\}$.   
Every vertex of $F$ is adjacent to the vertex of $X$. 
Since  the vertex of $X$has degree $\delta$, 
we can get that $X$ is an independent set in $G$, 
thus we have $G[X,F]=K_{\delta-1,\delta}$. 
The above discussion reveals that 
$G=K_{\delta } \vee (I_{\delta -1} \cup K_{n-2\delta +1})$.  
This completes the proof. 
\end{proof}

Recently, Chen et al. \cite{CZ2018}, Yu et al. \cite{YuGuidongJAM}, and 
Wei et al. \cite{WYL2019} independently 
showed the spectral version of Hamilton-connectedness 
for graphs 
with large minimum degree. 

\begin{theorem}[Wei et al. \cite{WYL2019}]
Let $\delta \ge 3$ and $n\ge n_0({\delta})$ where 
\[   n_0({\delta})= 
 \max\{ 6\delta, (\delta^3 - \delta^2 +2\delta +8)/2\}. \] 
 If $G$ is a graph on $n$ vertices with 
 the minimum degree $\delta (G) \ge \delta $ and 
\[ \lambda (G) \ge n-k,\]
 then $G$ is 
Hamilton-connected unless $G=S_{n,\delta}$ 
or $G=T_{n,\delta}$. 
\end{theorem}

\begin{corollary}
Let $\delta \ge 3$ and $n\ge n_0({\delta})$. 
If $G$ is a graph on $n$ vertices with 
 the minimum degree $\delta (G) \ge \delta $ and 
\[ \lambda (G) \ge \lambda (S_{n,\delta}),\] 
then $G$ is 
Hamilton-connected unless $G=S_{n,\delta}$. 
\end{corollary}

In 2020, Zhou, Wang and Lu \cite{Zhouqianna20a} 
presented 
the signless Laplacian spectral conditions 
for Hamilton-connected graphs with large minimum degree. 

For the graph $S_{n,\delta}$, let $X=\{v\in V(S_{n,\delta}) : 
d(v)=\delta \}, Y=\{v\in V(S_{n,\delta}): d(v)=n-1\}$ 
and $Z=\{v\in V(S_{n,\delta}) : d(v)=n-\delta \}$. 
Let $E_1(S_{n,\delta})$ be the subset of edges of 
$S_{n,\delta}$ whose both endpoints are from 
$Y\cup Z$. We define the family of graphs as below 
\[ \mathcal{S}_{n,\delta}^{(1)} = 
\left\{ S_{n,\delta} \setminus E': 
E'\subseteq E_1(S_{n,\delta}) ~ \text{and}~ 
|E'|\le \bigl\lfloor {\delta (\delta-1)}/{4} \bigr\rfloor \right\} . \]
Here, we denote by $S_{n,\delta} \setminus E'$ the graph obtained from 
$S_{n,\delta}$ by deleting all edges of the edge set $E'$.  
Similarly, for the graph $T_{n,\delta}$, 
let $X=\{v\in V(T_{n,\delta}): d(v)=\delta\} , 
Y=\{v\in V(T_{n,\delta}) : d(v)=n-1\}$ and 
$Z=\{v\in V(T_{n,\delta}) : d(v)=n-\delta \}$. 
Let $E_1(T_{n,\delta})$ be the subset of $E(T_{n,\delta})$ 
containing the edges whose both endpoints are from $Y\cup Z$. 
We denote  
\[ \mathcal{T}_{n,\delta}^{(1)} = 
\Bigl\{ T_{n,\delta} \setminus E': 
E'\subseteq E_1(T_{n,\delta}) ~ \text{and}~ 
|E'|\le \bigl\lfloor {(\delta -1)}/{2} \bigr\rfloor \Bigr\} . \]
When $G$ is a subgraph of 
$S_{n,\delta}$ or $T_{n,\delta}$, 
we can easily see  that 
$G$ is not Hamilton-connected. 
Moreover, we can verify that 
 $\delta (G)=\delta$ 
 and $q(G)\ge 2(n-\delta )$ 
 for every $G\in \mathcal{S}_{n,\delta}^{(1)} 
\cup \mathcal{T}_{n,\delta}^{(1)}$. 

\begin{theorem}[Zhou--Wang--Lu \cite{Zhouqianna20a}]
Let $\delta \ge 2$ and $n\ge n_1(\delta)$ where 
\[ n_1(\delta) =
 \delta^4 +5\delta^3 +2\delta^2 +8 \delta +12. \] 
If $G$ is a graph of order $n$ with 
  minimum degree $\delta (G) \ge \delta$ and 
\[ q(G)\ge 2(n-\delta),\] 
then $G$ is Hamilton-connected unless 
$G\in \mathcal{S}^{(1)}_{n,\delta}$ 
or $G\in \mathcal{T}^{(1)}_{n,\delta}$. 
\end{theorem}

\begin{corollary}
Let $\delta \ge 3$ and $n\ge n_1({\delta})$. 
If $G$ is a graph on $n$ vertices with 
  minimum degree $\delta (G) \ge \delta $ and 
\[ q (G) \ge q (S_{n,\delta}),\] 
then $G$ is 
Hamilton-connected unless $G=S_{n,\delta}$. 
\end{corollary}

Recently, Xu, Zhai and Wang \cite{XZW2021} 
present some sufficient conditions for a graph 
to be Hamilton-connected in terms of size, 
spectral radius and signless Laplacian spectral radius, respectively, which extend some corresponding results.

\section{Spectral  problem for deficiency}  

\subsection{Problem for given matching number} 

A matching in a graph is a set 
of disjoint edges, 
the matching number of $G$, denoted by $\alpha' (G)$, 
is the maximum size of a matching in $G$. 
A perfect matching in $G$ is a matching covering all vertices.  

In 1959, Erd\H{o}s and Gallai \cite{EG59} proved 
the following extremal problem 
for $M_{k+1}$, a matching of $k+1$ edges; 
see \cite[p. 58]{Bollobas78} and \cite[Theorem 2]{AF85} for alternative proofs.

\begin{theorem}[Erd\H{o}s--Gallai \cite{EG59}] \label{EG59}
For any $n\ge 2k+1$, if $G$ is an $n$-vertex  graph without copy of $M_{k+1}$, 
then 
\[ e(G)\le  \max \left\{ {2k+1 \choose 2}, 
{k \choose 2} + (n-k)k  \right\}.  \]
Moreover, the extremal are the following. \\
(1) If $2k+1 \le n< \frac{5k+3}{2}$, 
then $K_{2k+1}\cup I_{n-2k-1}$ is the unique extremal graph.   \\
(2) If $n= \frac{5k+3}{2}$, then there are two extremal graphs  $K_{2k+1}\cup I_{n-2k-1}$ and 
$K_k \vee I_{n-k}$.  \\
(3) If $n > \frac{5k+3}{2}$, then $K_k \vee I_{n-k}$ 
is the unique extremal graph. 
\end{theorem}

Recently, Wang \cite{Wang20} proved that 
the two extremal graphs in Theorem \ref{EG59} not only 
have the 
maximum number of $s$-cliques  for every $s\ge 2$, 
but also attain the maximum number of  $K_s \vee I_t$, 
which is the graph obtained from $K_{s,t}$ 
by replacing the part of size $s$ by a clique of the same size. 
We infer the interested reader to \cite{Wang20,DNPWY20} 
for recent progress. 

In 2007, Feng, Yu and Zhang \cite{FYZ07} 
proved the spectral version of Erd\H{o}s--Gallai's theorem 
and determined the largest spectral radius with given matching number. First of all, we can see that 
$K_k \vee I_{n-k}$ does not contain a matching of size $k + 1$ and 
\[ \lambda (K_k \vee I_{n-k}) = 
\frac{k-1 + \sqrt{(k-1)^2 +4k (n-k)}}{2}.  \]

\begin{theorem}[Feng et el.  \cite{FYZ07}]   
Let $G$ be a graph on $n$ vertices with matching number $k$. \\
(1) If $n=2k$ or $2k+1$, then 
$\lambda (G) \le \lambda (K_n)$ with equality if and only if $G=K_n$. \\
(2) If $2k+2 \le n <3k+2$, 
then $\lambda (G) \le 2k$ with equality if and only if 
$G=K_{2k+1} \cup I_{n-2k-1}$. \\
(3) If $n=3k+2$, then $\lambda (G) \le 2k$ 
with equality if and only if $G=K_k \vee I_{n-k}$ 
or $G=K_{2k+1}\cup I_{n-2k-1}$. \\
(4) If $n>3k+2$, then $\lambda (G)\le 
\lambda (K_k \vee I_{n-k})$ 
with equality if and only if $G=K_k\vee I_{n-k}$. 
\end{theorem}

In 2008, Yu \cite{Yu2008} presented the spectral version 
for the signless Laplacian radius. 
To some extent, the results and the extremal graphs 
are similar with that of 
the adjacency spectral radius. 
Some tedious calculation yields 
\[  q(K_k\vee I_{n-k}) = \frac{n+2k-2+ \sqrt{(n+2k-2)^2 - 8k^2 +8k
}}{2}. \]

\begin{theorem}[Yu \cite{Yu2008}]
Let $G$ be a graph on $n$ vertices with matching number $k$. \\
(1) If $n=2k$ or $2k+1$, then 
$q (G)\le q(K_k)$, with equality if and only if $G=K_n$. \\
(2) If $2k+2 \le n < \frac{5k+3}{2}$, 
then $q(G) \le 4k$, with equality if and only if 
$G=K_{2k+1} \cup I_{n-2k-1}$. \\
(3) If $n=\frac{5k+3}{2}$, then $q(G)\le 4k$, 
with equality if and only if $G=K_k \vee I_{n-k}$ 
or $G=K_{2k+1}\cup I_{n-2k-1}$. \\ 
(4) If $n>\frac{5k+3}{2}$, then $q (G)\le 
q (K_k \vee I_{n-k})$ 
with equality if and only if $G=K_k\vee I_{n-k}$. 
\end{theorem}

\subsection{Problem for perfect  matching}

The study of the relation  between the eigenvalues and the matching number was initiated by Brouwer and Haemers \cite{Brouwer}. For regular graphs,
they obtained
\begin{theorem}\cite{Brouwer}
Let $G$ be a connected $k$-regular graph on $n$ vertices with adjacency eigenvalues
$k =\lambda_1\geq \lambda_2   \geq \lambda_3\geq  \cdots \geq \lambda_n$. If $n$ is even and
$$
\lambda_3  \leq \left\{
           \begin{array}{ll}
              k-1+\frac{3}{k+1}, & \hbox{if $k$ is even,} \\[2mm]
             k-1+\frac{3}{k+2}, & \hbox{if $k$ is odd,}
           \end{array}
         \right.
$$
then $G$ has a perfect matching.
\end{theorem}
Subsequently,    Cioab\u{a}, Gregory and Haemers 
\cite{Cioaba07LAA, Cioaba09JCTB} refined and generalized the above result to obtain the following result for the $(r+1)$-th largest eigenvalue. 

\begin{theorem} \cite{Cioaba07LAA}
Let $G$ be a connected $k$-regular graph on $n$ vertices with eigenvalues $k =\lambda_1\geq \lambda_2   \geq \lambda_3\geq  \cdots \geq \lambda_n$. Assume $r> 0$ is an integer such that $n  \equiv r {(\mod 2)}$. If
$$
k-\lambda_{r+1}  > \left\{
           \begin{array}{lll}
            0.1475,              & \hbox{if $k=3$,} \\[2mm]
              1-\frac{3}{k+1}-\frac{1}{(k+1)(k+2)}, & \hbox{if $k$ is even,} \\[2mm]
             1-\frac{3}{k+2}-\frac{1}{(k+2)^2}, & \hbox{if $k\geq 5$ is odd,}
           \end{array}
         \right.
$$
then $G$ has a   matching of size at least $\frac{n-r}{2}+1$.
\end{theorem}

Further, they found an explicit expression $\rho(k)$ in \cite{Cioaba09JCTB}, and proved that if $\lambda_3(G) < \rho(k)$, then $G$ has a matching of size $ \lfloor\frac n2\rfloor $.
 Other related results can be found in \cite{LiuhuiqingLumei,  OSuilSebiSIAM}.

In what follows, we  shall consider the sufficient conditions on the existence of perfect matching 
in terms of the number of edges. First of all, we define 
\[ \boxed{  R_n :=K_{1}\vee( K_{n-3}\cup I_2). } \]  
It is easy to verify that 
$R_n$ has no perfect matching 
and $e(R_n) = {n-2 \choose 2} +2$.  
In 2021, 
Suil O \cite{Suil21} proved that 
if an $n$-vertex ($n\ge 10$) graph has 
more   edges than $R_n$, 
then it has  a perfect matching. More precisely, 
they showed the following theorem even for graphs with small 
order. 

\begin{theorem}[Suil O \cite{Suil21}]   \label{th31} 
Let $n\ge 10$ be an even integer or $n=4$. 
If $G$ is a connected graph of order $n$ with  
$e(G) > \binom{n-2}{2}+2$. 
then $G$ has a perfect matching. 
For  $n=6$ or $n=8$, if 
$e(G)>9$ or $e(G) >18$, respectively, then 
$G$ has a perfect matching. 
\end{theorem}

We mention here that the extremal graphs in Theorem \ref{th31}
can be characterized easily. 
By a careful examination, 
we get show that if $e(G) \ge 
{n-2 \choose 2} +2$, then $G$ has a perfect matching 
unless  $n\ge 12$ and $G=R_n$, or 
$n=6$ and $G\in \{R_6, K_{2}\vee I_{4} , K_{2,4}, K_{1}\vee(K_{1}\cup K_{1,3})\}$,  
or $n=8$ and $G\in \{R_8,K_1\vee K_{2,5}, K_2\vee(K_1\cup K_{1,4}), 
K_2\vee(K_2\cup K_{1,3}) , K_3\vee I_{5}\}$, or $n=10$ and $G\in \{R_{10}, K_{4}\vee I_{6}\}$. 
The proof is a standard  case analysis, so we omit the details.  
Moreover, Suil O \cite{Suil21} proved the following 
lower bound for spectral radius in an $n$-vertex graph 
to guarantee the existence of a perfect matching.

\begin{theorem}[Suil O \cite{Suil21}] \label{thm68}
Let $n$ be an even integer and $G$ be an $n$-vertex graph. \\
(1) If $n=4$ or $n\ge 8$ and  $  \lambda (G) 
> \lambda (R_n)$,   
then $G$ has a perfect matching. \\
(2) If $n=6$ and $\lambda (G)> \frac{1+\sqrt{33}}{2}$, 
then $G$ has a perfect matching. 
\end{theorem}

The condition is best possible since 
both $R_n$ and 
$K_2 \vee I_4$ have no perfect matching.    
Some calculations yield that 
$\lambda (K_2 \vee I_4)= \frac{1+\sqrt{33}}{2} \approx 
3.372 $ and $\lambda (R_6) \approx  3.177$. 
Moreover, it is not hard to see that  
$\lambda (R_n)$ is the largest root of 
\[  x^3-(n-4)x^2 - 
(n-1)x +2(n-4)=0. \]
We here make a brief proof. 
Let $\bm{x}=(x_1,\ldots ,x_n)$ be the eigenvector 
corresponding to $\lambda = 
\lambda (R_n)$. 
By the symmetry, we know that the entries of eigenvector 
corresponding to the two vertices of degree $1$ are equal,  
we may denote by $x_1=x_2=x$. Similarly, 
the entries of eigenvector corresponding to 
the $n-3$ vertices with degree $n-3$ are denoted by 
$x_3=\cdots =x_{n-1}=y$, and 
the entry of eigenvector corresponding to 
the vertex with degree $n-1$ is denoted by $x_n=z$. 
The eigen-equation $A\bm{x}=\lambda \bm{x}$ implies that 
\[ \begin{cases}
\lambda x= z, \\
\lambda y = (n-4) y + z, \\
\lambda z= 2x + (n-3)y. 
\end{cases} \]
We can obtain that $x=\frac{1}{\lambda }z$ and 
$y=\frac{1}{\lambda -n +4} z$, and thus 
\[  \lambda ^3-(n-4)\lambda^2 - (n-1)\lambda +2(n-4)=0 \]
Next, we shall show that 
\[   n-3 <  
\lambda (R_n) < \sqrt{(n-3)^2+1}.\]  
Since $K_{n-2}$ is a proper subgraph of 
$K_1 \vee (K_{n-3} \cup I_2)$, the second inequality 
holds immediately. 
Let $f(x)=x^3-(n-4)x^2 - 
(n-1)x +2(n-4)$. We can compute  the derivation
$f'(x)=3x^2 -2(n-4)x -(n-1)$, which has two 
distinct roots $x_1=\frac{1}{3}(n-4 - \sqrt{n^2-5n+13})$  and 
$x_2=\frac{1}{3}(n-4 + \sqrt{n^2-5n+13})$. 
Note that $f(x_2) <0$. 
An easy calculation reveals that $x_2 < \sqrt{(n-3)^2+1}$ and $f(\sqrt{(n-3)^2+1}) >0$ for $n\ge 8$, hence 
$\sqrt{(n-3)^2+1} $ is  greater than the largest root of 
$f(x)$. By the above discussion, we know that 
the largest root of $f(x)$ is approximately equal to 
$n-3$. Given a graph $G$, it is not straightforward to verify 
the condition $\lambda (G) > \lambda (R_n)$. 
Hence, it is possible and meaningful to relax the condition to 
$\lambda (G)\ge n-3$ and characterize more extremal graphs.

 In 2020, Liu, Pan and Li \cite{LPL2020}
 presented the  result  about the relationship between the
signless Laplacian spectral radius and prefect matching in graphs. 
Recall that $R_n =K_{1}\vee( K_{n-3}\cup I_2)$ and 
$R_n$ has no perfect matching.

\begin{theorem}[Liu--Pan--Li \cite{LPL2020}] 
\label{thmLPL}
Let $n$ be  even 
and $G$ be an $n$-vertex connected graph. \\ 
(1) If $n=4$ or $n\ge 10$ and 
$q(G)> q(R_n)$, then 
$G$ has a perfect matching. \\
(2) If $n=6$ and $q(G) > 4+2\sqrt{3}$, 
then $G$ has a perfect matching. \\
(3) If $n=8$ and $q(G) > 6+2\sqrt{6}$, 
then $G$ has a perfect matching.
\end{theorem}

First of all, we can show that $q(R_n)$ 
is the largest root of the equation 
\[  x^3 - (3n-7)x^2 + n(2n-7)x- 2(n^2-7n +12)=0.  \]
Moreover, we can see that $K_2\vee I_4$ 
and $K_3\vee I_5$ have no perfect matching. 
By some calculations, we can get that 
$q(K_2\vee I_4) =4+2\sqrt{3}\approx 7.464$ and 
$q(R_6)\approx 6.909$. For $n=8$, we 
have  $q(K_3\vee I_5) =6+2\sqrt{6}\approx 10.899$ 
and $q(R_8)\approx 10.513$.

\subsection{Problem for graph deficiency}

The \textit{deficiency }of a graph $G$, denoted by def$(G)$, is the number of unmatched 
vertices  under a maximum matching in $G$. In particular, $G$ has a perfect matching if and only if $\mathrm{def}(G)=0$. 
And $G$ has a nearly perfect matching if and only if $\mathrm{def}(G)=1$. 

Let $G=(V,E)$ be a simple graph on $n$ vertices and 
 $\alpha' (G)$ denote the number of edges  in
a maximum matching  of  $G$.  
To generalize the result of Tutte \cite{Tut1947}, 
Berge \cite{Berge} obtained the  well-known Berge--Tutte formula. 

\begin{theorem}[Berge--Tutte]
Let $S\subseteq V$ be a set of vertices. 
We denote by $\mathrm{odd}(G-S)$ 
 the number of components with odd vertices 
in the induced subgraph  $G-S$.  Then 
\begin{equation*}
\label{eqtb}
\alpha ' (G) = \frac{1}{2} 
\bigl( n-\max_{S \subseteq V}\{ \mathrm{odd}(G-S)-|S|\} \bigr). 
\end{equation*}
\end{theorem}
With the help of the Berge--Tutte formula, we can get 
\[ \mathrm{def}(G)= \max_{S \subseteq V}\{ \mathrm{odd}(G-S)-|S|\}. \]

We say that $G$ is \emph{$\beta$-deficient} if def$(G)\leq \beta$. 
In particular, $G$ is $0$-deficient if and only if 
$G$ has a perfect matching. And $G$ is $1$-deficient 
if and only if $G$ has a perfect matching or 
a nearly perfect matching.

\begin{theorem}[Bondy--Chvatal~\cite{Bondy}] 
Let $G$ be a simple graph on $n$ vertices.  
Then $G$ is $\beta$-deficient if and only if 
the closure graph $\mathrm{cl}_{n-\beta -1}(G)$ is  $\beta$-deficient. 
In particular, if $d(u)+d(v) \ge n-\beta -1$ for all non-edges $\{u,v\}$, 
then $\mathrm{cl}_{n-\beta -1} (G) =K_n$ and $G$ is $\beta$-deficient. 
\end{theorem}

\begin{lemma}   \cite{BauerSurvey, LasVergnas}  
Let $G$ be a connected graph on $n$ vertices with degree sequence 
$d_1\le d_2 \le \cdots \le d_n$. 
Suppose $0 \leq \beta \leq n$ and $n\equiv \beta \ (\mod 2)$. 
If $G$ is  not $\beta$-deficient, then 
there exists $1\leq i\leq \frac{n+\beta-2}{2}$ such that 
$d_{i+1}\leq i-\beta $ and $ d_{n+\beta-i}\leq n-i-2$.
\end{lemma}

\begin{lemma}[Feng et al. \cite{FengLAA17}] 
Let $G$ be a graph of order $n\geq 11$. Let $0\le\beta\le n$ with $n\equiv \beta \ (\textup{mod} \ 2)$. If $e(G) \geq \frac12(n^2-5n-4\beta+10)$, then $G$ is $\beta$-deficient, unless  $\beta=0$ and $G=K_{1}\vee(I_2\cup K_{n-3} )$, or $\beta=1$ and $G=I_2\cup K_{n-2}$.
\end{lemma}

For every integers 
$\beta \ge 0$ and $\delta \ge 1$, we denote
\[ 
\boxed{ R_{n,\delta, \beta}:=K_{\delta}\vee(K_{n-2\delta-\beta-1}\cup 
{I_{\delta+\beta+1}}). } 
\] 
It is easy to see that $R_{n,\delta , \beta}$ 
 is a graph with minimum degree $\delta$ and 
 deficiency $\beta +1$, so it is not $\beta$-deficient. 
 Furthermore, we can easily prove the 
 following theorem.  
 
 \begin{theorem} 
For every $\beta \ge 0$ and $\delta \ge 1$, 
there exists an $n_0(\delta , \beta)$ such that 
for every $n\geq  n_0(\delta,\beta)$ with $n\equiv \beta (\mod 2)$,  
if $G$ is an $n$-vertex connected graph with minimum degree $\delta(G)\geq \delta $ and  
 \[  e(G)\geq e(R_{n,\delta, \beta}),\] 
then $G$ is $\beta$-deficient unless $G=R_{n,\delta, \beta}$.
\end{theorem}
 
In 2018, Liu et al. \cite{WJLiuLAMA} proved the 
following spectral extremal problem for $\beta$-deficient.

\begin{theorem}[Liu et al. \cite{WJLiuLAMA}]  \label{th32}
Let $\beta \ge 0$ and $\delta \ge 1$ 
and  $n\equiv \beta (\mod 2)$ with 
$n\ge \widetilde{n_0}(\delta , \beta)$ where 
\[ \widetilde {n_0}(\delta , \beta)= 
\max\{ \delta^3 \!+\! \delta^2 (\beta \!+\!1) \!+ \! 
\delta \!+\! \beta \!+\! 6, 
2\delta^2 \!+\! 2\delta \beta \!+\! 7\delta \!+ \! 3\beta \!+ \! 7, 
(\delta \!+\! \beta \!+\! 2)(\delta \!+\! 1)\delta \}. \]
If $G$ is an $n$-vertex connected graph with minimum degree $\delta(G)\geq \delta $ and  
  $$\lambda(G)\geq n-\delta -\beta -2,$$
then $G$ is $\beta$-deficient unless $G=R_{n,\delta, \beta}$. 
\end{theorem}

\begin{corollary}[Liu et al. \cite{WJLiuLAMA}]   \label{th132}  
Let $\beta \ge 0$ and $\delta \ge 1$ 
and  $n\equiv \beta (\mod 2)$ with 
$n\ge  {n_0}(\delta , \beta)$ where 
\[ n_0(\delta,\beta):=\max\{7+7\delta+2\delta^2+2\delta \beta+ 3\beta, (\delta+\beta+2)(\delta+1)\delta\}.\] 
If $G$ is an $n$-vertex connected graph with minimum degree $\delta(G)\geq \delta $ and  
  $$\lambda(G)\geq \lambda(R_{n,\delta, \beta}),$$
then $G$ is $\beta$-deficient unless $G=R_{n,\delta, \beta}$.
\end{corollary}

For some special cases, one can relax 
the bound on order of graphs. 
For example, when $\delta =1$ and $\beta =0$, 
i.e., the existence of perfect matching, 
the desired result holds for all $n\ge 8$; 
see the previous Theorem \ref{thm68}.

Furthermore, 
Liu et al. \cite{Huangzheng19} obtained a tight $Q$-spectral condition  for a graph to be $\beta$-deficient. 
In order to state their main result, 
we need to fix some notation. 
Recall that ${R}_{n,\delta,\beta }=
K_{\delta}\vee(K_{n-2\delta-\beta-1}\cup 
{I_{\delta +\beta+1}})$. 
In the graph ${R}_{n,\delta,\beta}$, we
denote by 
$ X=\{v \in V({R}_{n,\delta,\beta}):d(v)=\delta \},
 Y=\{v \in V({R}_{n,\delta ,\beta}):d(v)=n-1 \}$ 
and  $Z=\{v \in V({R}_{n,\delta,\beta}):d(v)=n-\beta-\delta-2\}$. 
Let $E_1({R}_{n,\delta,\beta})$ denote the set of edges of 
${R}_{n,\delta,\beta}$ whose endpoints are both in $Y\cup Z.$ We further define the following family of graphs. 
\begin{eqnarray*}
 \mathcal{R}^{(1)}_{n,\delta,\beta}=\left\{ {R}_{n,\delta,\beta} \setminus E':  \text{$ E'\subseteq E_1
 ({R}_{n,\delta,\beta})$ with $|E'|\leq \left\lfloor { (\delta+\beta+1)\delta }/{4}\right\rfloor$}\right\}.
\end{eqnarray*} 
Here, we denote by $R_{n,\delta,\beta} \setminus E'$ the graph obtained from 
$R_{n,\delta,\beta}$ by deleting all edges of the edge set $E'$.   
As a $Q$-spectral counterpart, we obtain the following 
analogue of signless Laplacian spectral radius. 

\begin{theorem}[Liu et al. \cite{Huangzheng19}]\label{th12}
Let $\beta \ge 0$ and $\delta \ge 1$ 
and  $n\equiv \beta (\mod 2)$ with 
$n\ge  {n_1}(\delta , \beta)$ where 
\[ n_1(\delta,\beta)= \delta^4+(2 \beta +7)\delta^3+(\beta^2+11\beta+15)\delta^2+(4\beta^2+13\beta+20)\delta
 +10\beta+21.\] 
If $G$ is a connected graph on  $n$ vertices  with the  minimum degree $\delta(G) \geq \delta $ and 
 \[ q(G)\geq 2(n-\delta -\beta-2),\] 
then either $G$ is $\beta$-deficient or 
$G\in \mathcal{R}^{(1)}_{n,\delta ,\beta}$.
\end{theorem}

\begin{corollary} 
Let $\beta \ge 0$ and $\delta \ge 1$ 
and  $n\equiv \beta (\mod 2)$ with 
$n\ge  {n_1}(\delta , \beta)$. 
If $G$ is an $n$-vertex connected graph with minimum degree $\delta(G)\geq \delta $ and  
  $$q (G)\geq q (R_{n,\delta, \beta}),$$
then $G$ is $\beta$-deficient unless $G=R_{n,\delta, \beta}$.
\end{corollary}

For some special cases, one can relax 
the bound on order of graphs. 
For example, when $\delta =1$ and $\beta =0$, 
i.e., the existence of perfect matching, 
the desired result holds for all $n\ge 10$; 
see the previous Theorem \ref{thmLPL}.

\section*{Acknowledgements}
The first author is now  a  doctoral student (2019--2023) 
 under the guidance of Prof. Yuejian Peng (Hunan University). 
 He would like to thank Prof. Peng    
who  provided  many constructive suggestions for his academic study.  
The authors would like to  express sincere thanks to 
Vladimir Nikiforov for kind discussions, which considerably improves the presentation. 
Thanks also go to  Dr. Xiaocong He and Dr. Peng-Li Zhang for
reading an earlier draft of the paper. 
Last but not least, the authors would like to thank anonymous reviewers and editors for their valuable comments and suggestions to improve the presentation of this paper. 
This work was supported by NSFC (Grant No. 11871479, 12071484 and 11931002),  
Hunan Provincial Natural Science Foundation 
(Grant No. 2020JJ4675 and  2018JJ2479) 
and Mathematics and Interdisciplinary Sciences Project of CSU.

\end{document}